\documentclass[12pt,a4paper,twoside,reqno]{amsart}
\title[Noncommutative PVA and CD algebras]{Noncommutative Poisson vertex algebras and Courant--Dorfman algebras}

\author[L. \'Alvarez-C\'onsul]{Luis \'Alvarez-C\'onsul}
\address{Instituto de Ciencias Matem\'aticas (CSIC-UAM-UC3M-UCM)\\ Nicol\'as Cabrera 13--15, Cantoblanco\\ 28049 Madrid, Spain}
\email{l.alvarez-consul@icmat.es}

\author[D. Fern\'andez]{David Fern\'andez}
\address{Mathematics Research Unit, Luxembourg University \\ Maison du Nombre, 6 Avenue de la Fonte \\ L-4364 Esch-sur-Alzette, Luxembourg}
\email{david.fernandez@uni.lu}

\author[R. Heluani]{Reimundo Heluani}
\address{Instituto de Matem\'atica Pura e Aplicada (IMPA) \\ Estrada Dona Castorina 110 \\ 22460-320 \\ Rio de Janeiro\\ Brazil}
\email{heluani@potuz.net}

\usepackage{hyperref}
\makeatletter

\usepackage{amssymb,latexsym}
\usepackage{amsmath,amsthm}
\usepackage[latin1]{inputenc}
\usepackage{mathrsfs} 
\usepackage[dvipsnames]{xcolor}
\usepackage{mathtools}
\usepackage{tikz-cd} 
\usepackage{enumitem}
\usepackage{bbold}
\usepackage{upgreek}

\usepackage{bbm}

\theoremstyle{plain}
\newtheorem{theorem}{Theorem}[section]
\newtheorem{lemma}[theorem]{Lemma}
\newtheorem{corollary}[theorem]{Corollary}
\newtheorem{proposition}[theorem]{Proposition}

\theoremstyle{definition}
\newtheorem{definition}[theorem]{Definition}
\newtheorem{definition-theorem}[theorem]{Definition-Theorem}

\theoremstyle{remark}
\newtheorem{remark}[theorem]{Remark}

\newcommand{\secref}[1]{\S\ref{#1}}

\numberwithin{equation}{section} 

\newcommand{\surj}{\to\kern-1.8ex\to}

\newcommand{\lto}{\longrightarrow}
\newcommand{\lra}[1]{\stackrel{#1}{\longrightarrow}}

\newcommand{\defeq}{\mathrel{\mathop:}=} 
\renewcommand{\(}{\left(}
\renewcommand{\)}{\right)}

\DeclarePairedDelimiterX\set[1]\{\}{%

#1
}

\newcommand{\NN}{\mathbb{N}}

\newcommand{\ZZ}{\mathbb{Z}}
\newcommand{\kk}{{\Bbbk}} 
 
\newcommand{\op}{{\operatorname{op}}} 
\newcommand{\e}{{\operatorname{e}}} 
\newcommand{\du}{\operatorname{d}\!}
\newcommand{\out}{{\operatorname{out}}} 
\newcommand{\inn}{{\operatorname{inn}}} 
\newcommand{\Id}{{\operatorname{Id}}} 
\newcommand{\Rep}{{\operatorname{Rep}}}

\newcommand{\MM}{{\operatorname{M}}}
\newcommand{\?}{\mbox{$-$}}

\def\cE{\mathscr{E}}
\def\cO{\mathscr{O}}
\newcommand{\cA}{\Omega}

\renewcommand{\sb}[1]{
  \left\{#1\right\}}

\newcommand{\db}[1]{
  \left\{\mkern-6mu\left\{#1\right\}\mkern-6mu\right\}}
\newcommand{\lb}[1]{
  \left\{\mkern-6mu\left\{#1\right\}\mkern-6mu\right\}}
\newcommand{\cc}[1]{
  [\mkern-3mu[#1]\mkern-3mu]}

\newcommand{\ii}[1]{
\langle\mkern-4mu\langle#1\rangle\mkern-4mu\rangle}

\newcommand{\Sym}{\operatorname{Sym}}

\newcommand{\Hom}{\operatorname{Hom}}

\newcommand{\Der}{\operatorname{Der}}

\newcommand{\DR}[2]{{\operatorname{DR}_R^{#1}\!{#2}}}
\newcommand{\DDer}{\operatorname{\mathbb{D}er}}
\newcommand{\D}{\operatorname{\mathbb{D}er}} 
\newcommand{\ev}{\operatorname{ev}}

\newcommand{\Spec}{\operatorname{Spec}}

\newcommand{\GL}{{\operatorname{GL}}}

\newcommand{\dR}{{\operatorname{DR}}}
\newcommand{\pprime}{\prime\prime}
\newcommand{\wt}[1]{\left\lvert#1\right\rvert} 
\newcommand{\tr}{{\operatorname{tr}}}

\newcommand{\commV}{\mathcal{V}}
\newcommand{\comma}{\fontshape{ui}\selectfont\text{a}}
\newcommand{\commb}{\fontshape{ui}\selectfont\text{b}}%
\newcommand{\comme}{\fontshape{ui}\selectfont\text{e}}%
\newcommand{\commf}{\fontshape{ui}\selectfont\text{f}}%
\newcommand{\commg}{\fontshape{ui}\selectfont\text{g}}%
\newcommand{\commH}{{\fontshape{ui}\selectfont\text{H}}}
\newcommand{\commX}{{\fontshape{ui}\selectfont\text{X}}}
\newcommand{\commY}{{\fontshape{ui}\selectfont\text{Y}}}

\newcommand{\commalpha}{\upalpha}
\newcommand{\commbeta}{\upbeta}

\newcommand{\lr}[1]{
  \left\{\mkern-6mu\left\{#1\right\}\mkern-6mu\right\}}
  
\DeclareFontFamily{U}{BOONDOX-calo}{\skewchar\font=45 }
\DeclareFontShape{U}{BOONDOX-calo}{m}{n}{
  <-> s*[1.05] BOONDOX-r-calo}{}
\DeclareFontShape{U}{BOONDOX-calo}{b}{n}{
  <-> s*[1.05] BOONDOX-b-calo}{}
\DeclareMathAlphabet{\mathcalboondox}{U}{BOONDOX-calo}{m}{n}
\SetMathAlphabet{\mathcalboondox}{bold}{U}{BOONDOX-calo}{b}{n}
\DeclareMathAlphabet{\mathbcalboondox}{U}{BOONDOX-calo}{b}{n}

\def\Wd{{\mathcalboondox{d}}} 
\def\Wi{{\mathcalboondox{i}}} 
\def\WL{{\mathcalboondox{L}}} 

\thanks{}

\begin{document}

\begin{abstract}
We introduce the notion of double Courant--Dorfman algebra 
and prove that it satisfies the so-called Kontsevich--Rosenberg principle, 
that is, a double Courant--Dorfman algebra induces Roytenberg's 
Courant--Dorfman algebras on the affine schemes parametrizing 
finite-dimensional representations of a noncommutative algebra.
The main example is given 
by the direct sum of double derivations and noncommutative differential 
1-forms, possibly twisted by a closed Karoubi--de Rham 3-form. To show 
that this basic example satisfies the required axioms, we first prove a 
variant of the Cartan identity $[L_X,L_Y]=L_{[X,Y]}$ for double 
derivations and Van den Bergh's double Schouten--Nijenhuis bracket. This 
new identity, together with noncommutative versions of the other Cartan 
identities already proved by Crawley-Boevey--Etingof--Ginzburg and Van den Bergh, establishes the differential calculus on noncommutative 
differential forms and double derivations and should be of independent 
interest. Motivated by applications in the theory of noncommutative 
Hamiltonian PDEs, we also prove a one-to-one correspondence between double 
Courant--Dorfman algebras and double Poisson vertex algebras, introduced 
by De Sole--Kac--Valeri, that are freely generated in 
degrees $0$ and $1$.\\

\noindent\textbf{Keywords:} Double Courant--Dorfman algebras, noncommutative Cartan differential calculus, double Poisson vertex algebras, double Poisson algebras, double Dorfman bracket, Kontsevich--Rosenberg principle.

\end{abstract}

\maketitle
\tableofcontents

\section{Introduction}
\label{sec:intro}

\subsection{}\label{sub:intro.intro}
Let $A$ be a (unital associative) finitely generated algebra over a ground field $\kk$ of characteristic zero and $\MM_N(\kk)$ the algebra of $N\times N$-matrices.
A standard practice to study the representation theory of $A$ is to apply geometric methods to the representation spaces $\Rep(A,N)=\Hom_{\operatorname{Alg}}(A,\MM_N(\kk))$, endowed with their canonical affine-scheme structures (see, e.g.,~\cite{Brion12}). 
This method underpins an approach to noncommutative geometry, based on the so-called Kontsevich--Rosenberg principle~\cite{Kon94a,KR00}, whereby a property of an associative algebra $A$ should be regarded as `geometric' if it induces the corresponding geometric structure on its representation schemes. 
In this paper, we define a noncommutative version of Courant--Dorfman algebras. 
We prove that they satisfy the Kontsevich--Rosenberg principle and are related to the theory of double Poisson vertex algebras.
Courant--Dorfman algebras are an algebraic generalization introduced by Roytenberg~\cite{Roy09} of Courant algebroids~\cite{LWX97}, in which the symmetric bilinear form may be degenerate (see Definition~\ref{def:CD-algebra-def-comm}).
The formalism of double Poisson vertex algebras was developed by De Sole, Kac and Valeri~\cite{DSKV15} in their study of noncommutative Hamiltonian PDEs (see~\secref{sub:intro.noncommutative-PDEs}). 

\subsection{}\label{sub:intro.Cartan-id}

The Kontsevich--Rosenberg principle has been applied effectively to a range of geometric structures, such as symplectic and Poisson structures by Crawley-Boevey, Etingof and Ginzburg~\cite{CBEG07} and Van den Bergh~\cite{VdB08}, respectively.
An essential feature of~\cite{CBEG07,VdB08} is to replace differential forms and vector fields over a manifold by noncommutative differential forms and double derivations over an associative algebra $A$ 
and develop, as much as necessary, a differential calculus for these objects in order to construct the so-called bisymplectic forms and double Poisson brackets 
(throughout this paper, we work relatively over an associative algebra $R$, but to simplify the exposition, in this introduction we will focus on the `absolute case' $R=\kk$). 
Here, a double derivation on $A$ (see~\secref{sec:nc-diff-forms-double-derivations}) is a derivation $X\colon A\to A^{\otimes 2}_\out$ on the $A$-bimodule with underlying vector space $A^{\otimes 2}_\out=A\otimes A$ (tensor product over $\kk$) and outer multiplication, defined for all $a,b,a',a''\in A$ by $a'(a\otimes b)a''\defeq(a'a)\otimes(ba'')$.

The central ingredient of Van den Bergh's double Poisson algebras~\cite[Definition 2.3.2]{VdB08} (reviewed in~\secref{sec:double-Poisson}) is a double Poisson bracket, that is, a bilinear map
\begin{equation}\label{eq:intro:double-bracket}
\db{\?,\?}\colon A\times A\to A\otimes A,
\end{equation}
that satisfies appropriate noncommutative variants of skewsymmetry, the Leibniz rule and the Jacobi identity. 
The differential calculus of~\cite{VdB08} also involved
a noncommutative version of the usual Lie bracket of vector fields, called the double Schouten--Nijenhuis bracket, 
\begin{equation}\label{eq:sub:intro.Cartan-id.3}
\db{X,Y}\in(\DDer A \otimes A)\oplus (A\otimes \DDer A),
\end{equation}
for elements $X,Y$ of the $A$-bimodule $\DDer A$ of double derivations.
It will be convenient to use the decomposition (omitting the summation signs)
\[
\lr{X,Y}=\lr{X,Y}'_l\otimes \lr{X,Y}''_l+\lr{X,Y}'_r\otimes \lr{X,Y}''_r,
\]
with $\lr{X,Y}'_l,\lr{X,Y}''_r\in\DDer_RA$, $\lr{X,Y}''_l,\lr{X,Y}'_r\in A$ (see \eqref{eq:R-linear-double-SN-bracket.2} and \eqref{eq:Sweedler.R-linear-double-SN-bracket.1}). 

Noncommutative versions of three Cartan identities on differential forms were part of this differential calculus (see~\cite[(2.7.2)]{CBEG07},~\cite[(A.2), (A.3)]{VdB08}), but a fourth one was missing, namely a noncommutative version of the identity $[L_\commX,L_\commY]=L_{[\commX,\commY]}$ involving Lie derivatives along vector fields $\commX,\commY$.
This omission was not an oversight. Indeed, it was commented by Crawley-Boevey, Etingof and Ginzburg~\cite[Remark 2.7.3]{CBEG07} that the bracket $[L_X,L_Y]$ of the Lie derivatives along $X,Y\in\DDer A$ was very likely equal to the (appropriately defined) Lie derivative with respect to $\db{X,Y}$. 
The first result of our paper (see Theorem~\ref{thm:Cartan-identities.1}) is the following version of this Cartan identity: 
\begin{equation}\label{eq:sub:intro.Cartan-id.4}
\db{L_X,L_Y}=L_{\db{X,Y}}.
\end{equation}
As anticipated by Crawley-Boevey, Etingof and Ginzburg, the proof depends on an appropriate definition of the Lie derivative with respect to $\db{X,Y}$ (see~\eqref{eq:triple-derivations.7.b}). 
Expanding the right-hand side of~\eqref{eq:sub:intro.Cartan-id.4} (see~\eqref{eq:lemma:Cartan-identities.1.2.b}), this identity can also be written as 
\begin{equation}\label{eq:sub:intro.Cartan-id.5}
\begin{aligned}
L_{\lb{X,Y}}&=L_{\lb{X,Y}'_l}\otimes\lb{X,Y}''_l+\lb{X,Y}'_r\otimes L_{\lb{X,Y}''_r}
\\
&\qquad -i_{\lb{X,Y}'_l}\otimes\du{\lb{X,Y}''_l} +\du{\lb{X,Y}'_r}\otimes i_{\lb{X,Y}''_r}.
\end{aligned}
\end{equation}
It should be noted that the last two terms on the right-hand side of~\eqref{eq:sub:intro.Cartan-id.5} have no commutative counterparts.
The identity~\eqref{eq:sub:intro.Cartan-id.4}, together with~\cite[(2.7.2)]{CBEG07} and~\cite[(A.2), (A.3)]{VdB08}, uncover a noncommutative variant of the Cartan differential calculus used in ordinary geometry, providing a positive answer to~\cite[Remark 2.7.3]{CBEG07} (see Corollary~\ref{cor:answer-Crawley-Boevey-Etingof-Ginzburg-remark.1}). Additionally,~\eqref{eq:sub:intro.Cartan-id.4} gives rise to a new Cartan identity in the corresponding `reduced' differential calculus introduced in~\cite[\S 2.8]{CBEG07} (see Proposition~\ref{prop:reduced-Cartan-identities.1}). 

The new Cartan identity~\eqref{eq:sub:intro.Cartan-id.4} is applied in particular to obtain nontrivial examples of our noncommutative Courant--Dorfman algebras (see~\secref{sec:exact-DC-alg}).
At this point we should mention that the methods of noncommutative geometry based on noncommutative differential forms and double derivations have been applied to other geometric structures, such as quasi-Poisson and quasi-symplectic geometries by Van den Bergh~\cite{VdB08,VdB08a}, Poisson vertex algebras by De Sole, Kac and Valeri~\cite{DSKV15} (see~\secref{sub:intro.noncommutative-PDEs}) and multiplicative Poisson vertex algebras 
by Casati and Wang~\cite{CW} and Fairon and Valeri~\cite{FV}.
We expect that~\eqref{eq:sub:intro.Cartan-id.4} (or its `reduced version'~\eqref{eq:prop:reduced-Cartan-identities.1.e}) has applications in the study of some of these geometric structures too.

\subsection{}\label{sub:intro.PDEs}
Before going into the details about Courant--Dorfman algebras, it may be helpful to recall some motivation and background. 

Poisson algebras provide a well-known approach to the dynamics of classical Hamiltonian systems, whose time evolution is governed by ordinary differential equations
\begin{equation}\label{eq:sub:intro.PDEs.1}
\frac{\du{f}}{\du{t}}=\sb{H,f},
\end{equation}
where the Hamiltonian $H$ is an element of a Poisson algebra $P$, with bracket $\sb{\?,\?}$, and $f=f(t)\in P$ is an arbitrary observable evolving with the time parameter $t$. 
Many Hamiltonian systems in classical field theory admit a similar algebraic formulation, in which the Poisson algebra is replaced by a Poisson vertex algebra (PVA), that is, a (unital associative) commutative algebra $\mathcal{V}$, equipped with a derivation $\partial$ and a `$\lambda$-bracket'
\begin{equation}\label{eq:sub:intro.PDEs.2}
\commV\otimes\commV\longrightarrow\commV[\lambda], \qquad\comma\otimes\commb \longmapsto \{\comma_\lambda\commb\} = \sum_{n\geq 0}\comma_{(n)}\commb\,\lambda^n
\end{equation}
(where $\mathcal{V}[\lambda]\defeq\mathcal{V}\otimes\kk[\lambda]$), satisfying appropriate axioms that generalize the axioms of Poisson algebras.
In this case, the evolution equation~\eqref{eq:sub:intro.PDEs.1} becomes a Hamiltonian partial differential equation (PDE) with $\sb{\?,\?}=\sb{\?{}_\lambda\?}|_{\lambda=0}$, where now the Hamiltonian is a local functional $\smallint h\in\mathcal{V}/\partial\mathcal{V}$ and $f=f(x,t)$ is interpreted as a function of some independent variables $x$ and the time parameter $t$.
This approach, and its generalization given by non-local PVAs, 
has been particularly fruitful to interpret and unify interesting integrable Hamiltonian PDEs, such as the KdV equation (see, e.g.,~\cite{BDSK09,DSK13} for details).
 
\subsection{}\label{sub:intro.Lee-filtration-and-arc-spaces}
PVAs receive their name because the quasi-classical limits of Borcherds' vertex algebras are always of this type (see, e.g., \cite{kac:vertex} for background on vertex algebras).
Every vertex algebra $V$ has a canonical decreasing filtration
\[
V=F_{0}V\supset F_{1}V\supset\dots\supset F_pV\supset F_{p+1}V\supset\cdots,
\]
known as the \emph{Li filtration} \cite{li05}, and the associated graded algebra
\begin{equation*}
\mathcal{V}=\mathrm{gr}_FV=\oplus_{p \geq 0} F_{p}V/F_{p+1}V
\end{equation*}
carries a structure of graded Poisson vertex algebra.

The canonical example of a PVA is given by the ring of functions of the \emph{arc space} of a Poisson manifold. If $P$ is a Poisson algebra, we put $J_0 = P$, and let $J$ be the unital associative commutative algebra with a derivation of degree 1, freely generated (as a differential algebra) by $J_0$.
For example, we will have $J_1 = \Omega^1_{J_0}$, the module of K\"ahler differentials of $J_0$, and $\partial\colon J_0 \rightarrow J_1$ will be the universal derivation. Then the Poisson bracket of $J_0$ can be uniquely extended to a $\lambda$-bracket on $J$ by sesquilinearity and the Leibniz rule; see \cite{arakawa}. 
In the particular case where $P = \kk[x_1,\dots,x_n]$ is a polynomial algebra, we have that $J=\kk[x_1^{(p)},\dots,x_n^{(p)}]_{p \geq 0}$ is a differential polynomial algebra.
The derivation is defined by $\partial x_i^{(p)}=x_i^{(p+1)}$ and the degrees of the generators are given by $\deg x_i^{(p)}=p$. 

It was noted in \cite{arakawa} that the quasi-classical limit of any vertex algebra is a quotient of the free algebra $J$ described in the previous paragraph.
Indeed, given a vertex algebra $V$, one considers its associated graded algebra with respect to the Li filtration, $\mathcal{V}=\mathrm{gr}_FV$. Then $P\defeq\mathcal{V}_0$ is a Poisson algebra (known as Zhu's $C_2$-quotient) and by the universal property, it follows that $\mathcal{V}$ is a quotient of the arc algebra $J$ of $P$ described above. 

\subsection{}\label{sub:intro.noncommutative-PDEs}
It is a remarkable fact that many 
integrable Hamiltonian equations admit generalizations in which the field variables take their values in a (noncommutative) associative algebra.
These generalizations become significantly richer and require, as a non-trivial step, proper definitions for concepts such as Hamiltonian, symmetry and first integral (see, e.g.,~\cite{Kupershmidt.2000,Mikhailov-Sokolov.2000,Olver-Sokolov.1998}). 
It has been observed that the intricate calculations involved in these generalizations can be properly interpreted with tools of noncommutative symplectic or Poisson geometry (see, e.g.,~\cite[\S 4.3]{CBEG07} and~\cite{ORS}). 
Just as Barakat, De Sole and Kac~\cite{BDSK09} laid down the foundations of the theory of Poisson vertex algebras aimed at the study of Hamiltonian PDEs, De Sole, Kac and Valeri~\cite{DSKV15} used tools from~\cite{VdB08} to develop a formalism of double Poisson vertex algebras with applications in the theory noncommutative Hamiltonian PDEs (see~\secref{sec:DPVA}).

Returning to the main aim of this paper, there are at least two approaches one can undertake to obtain a noncommutative version of the notion of Courant--Dorfman algebra.
As explained below, the two ways we consider are compatible with each other and satisfy the Kontsevich--Rosenberg principle. 

\subsection{}\label{sub:intro.Roytenberg-Severa}

A first approach is motivated by the \v Severa--Roytenberg one-to-one correspondence \cite{Roy00,Sev00} between  Courant algebroids and symplectic $\NN Q$-manifolds of  degree 2, that is, non-negatively graded manifolds endowed with a graded symplectic form of degree 2 and a symplectic homological vector field of degree 1. 
Following this correspondence, but applied to a graded version of the bisymplectic noncommutative differential forms introduced in~\cite{BCER12}, the first two authors found a suitable noncommutative variant of Courant algebroid, called double Courant algebroid~\cite[Definition~7.1]{ACF17} (see also~\cite{Fer16}).
In this paper, we extend this notion to cover a noncommutative analogue of the more general Courant--Dorfman algebras.
The axioms are formulated using tools, introduced by De Sole, Kac and Valeri~\cite[Remark 2.2]{DSKV15}, that are especially well suited for calculations involving noncommutative analogues of Leibniz algebras.
The resulting mathematical object, which is called double Courant--Dorfman algebra and is the central concept introduced in this work, has applications in the theory of double PVAs, as discussed in \S\ref{sub:intro.the-equivalence}. 
It is a 5-tuple $(A,E,\ii{\?,\?},\partial,\cc{\?,\?})$, where $\partial\colon A\to E$ is a derivation on an $A$-bimodule $E$, 
$\ii{\?,\?}\colon E\times E\to A\otimes A$ is a pairing (see~\secref{sec:pairings}), and
\begin{equation}\label{eq:sub:intro.Roytenberg-Severa.1}
\cc{-,-}\colon E\times E\lto (E\otimes A)\oplus (A\otimes E)
\end{equation}
is a bilinear map, called the double Dorfman bracket, satisfying appropriate variants of the axioms of a Courant--Dorfman algebra, listed in Definition~\ref{CD-Definition}. In the rest of~\secref{sec:CD-algebras}, we prove some properties about Courant--Dorfman algebras that will be used throughout the rest of the paper, and finish with a definition of an exact double Courant--Dorfman algebra. 

\subsection{}\label{sub:intro.standard-CD-alg}

A second more specific approach is to look for a noncommutative analogue of the standard exact Courant algebroid, possibly with a twist introduced by \v Severa~\cite{Sev00}. The main ingredient of this (commutative) Courant algebroid is the Dorfman bracket on the direct sum $TM\oplus T^*M$ of the tangent and the cotangent bundles to a manifold $M$, defined by 
\begin{equation}\label{eq:intro.standard-CD-alg.1}
[\commX+\commalpha,\commY+\commbeta]_{\commH}\defeq[\commX,\commY]+\WL_{\commX}\commbeta-\Wi_{\commY}\Wd{\commalpha}+\Wi_{\commX}\Wi_{\commY}\commH,
\end{equation}
for sections $\commX+\commalpha,\commY+\commbeta$ of $TM\oplus T^*M$, where $\commH$ is a closed differential 3-form twisting the bracket (see~\secref{sec:KR-exact-double-CD} for notation). 
Following the most basic ideas of~\cite{CBEG07,VdB08} explained in~\secref{sub:intro.Cartan-id}, one can formulate a noncommutative version of this prototypical Courant algebroid. We consider the double Dorfman bracket on the direct sum $\DDer A\oplus\Omega^1A$ of the $A$-bimodules of double derivations and noncommutative differential 1-forms, 
defined by
\begin{equation}\label{eq:intro.standard-CD-alg.2}
\cc{X+\alpha,Y+\beta}_H\defeq\db{X,Y}+L_X\beta-(i_Y\du{\alpha})^\sigma+i_X\iota_YH,
\end{equation}
for all $X+\alpha,Y+\beta\in\DDer A\oplus\Omega^1A$, where $\db{\?,\?}$ is the double Schouten--Nijenhuis bracket~\eqref{eq:sub:intro.Cartan-id.3}, the operators $L_X,i_X$ and $\iota_Y$ are part of the differential calculus over double derivations (see~\secref{sec:diff-calculus}), $(\?)^\sigma$ is an involution defined in~\eqref{eq:cyclic-permutation.1} (for $n=2$) and the analogue of \v Severa's twist is now given by a closed Karoubi--de Rham 3-form $H$ (see~\eqref{eq:def-Karoubi-de-Rham}). 

The second main result of this paper, split in Theorem~\ref{thm:standard-DCA} (corresponding to $H=0$) and Theorem~\ref{thm:twist-CD-algebra-double} (for arbitrary $H$), states that this construction provides a double Courant--Dorfman algebra in the sense of Definition~\ref{CD-Definition}. 
The proofs, carried out in detail in~\secref{sec:exact-DC-alg}, are parallel, but substantially more intricate than the proofs of the corresponding geometric results. 
As in commutative geometry, the full list of Cartan identities are required.
In fact, the main reason why this example was not 
considered in~\cite{ACF17} was the lack of the identity~\eqref{eq:sub:intro.Cartan-id.4} (see, however, Remark~\ref{rem:twisted-CDA.1}).
Another source of complication that does not appear in commutative geometry is the presence of the involution $(\?)^\sigma$ permuting the two factors of the tensor product of $\Omega^1A$ and $A$ in~\eqref{eq:intro.standard-CD-alg.2}; this involution is needed to prove several axioms of Definition~\ref{CD-Definition}, such as the Leibniz rule and the Jacobi identity.

\subsection{}\label{sub:intro.Lie-Rinehart}
Our main result is inspired by the following link between Poisson vertex algebras and Courant--Dorfman algebras. 
Consider a graded PVA of weight $-1$, \emph{i.e.}, a 
commutative algebra $\commV$, equipped with an $\NN$-grading $\commV=\oplus_{n \geq 0}\commV_n$ on its underlying vector space, a derivation $\partial\colon\commV\to\commV$ of weight $1$, and a $\lambda$-bracket~\eqref{eq:sub:intro.PDEs.2}
of weight $-1$, 
\emph{i.e.}, such that $\comma_{(n)}\commb\in\commV_{i+j-n-1}$ for all $\comma\in\commV_i$, $\commb\in\commV_j$.
These operations are required to satisfy sesquilinearity, the Leibniz rule, 
skew-symmetry and the Jacobi identity (see~\cite[\S 0]{DSKV15}). 
It was observed by Ekstrand and Zabzine~\cite{EZ11}, following works of Bressler \cite{bressler1}, Roytenberg \cite{Roy09} and the third author \cite{Hel09}, 
that the pair $\cO=\commV_0$, $\cE=\commV_1$, equipped with the following data, is a Courant--Dorfman algebra:
\begin{enumerate}
\item 
the symmetric bilinear form $\langle-,-\rangle\colon \cE\otimes_{\cO} \cE\to \cO$ is given by $\langle a, b\rangle = a_{(1)}b$,
\item 
the derivation $\partial\colon \cO \to \cE$
is obtained by restriction of $\partial\colon\commV\to\commV$,
\item 
the bilinear map 
$[\?,\?]\colon\cE\times\cE\to\cE$, called the \emph{Dorfman bracket}, is defined by $[\comma,\commb]=\comma_{(0)}\commb$.
\end{enumerate}
%
If $\langle\?,\?\rangle$ is non-degenerate, then one recovers the definition of Courant algebroids of \cite{Roy00}. 

Conversely, given a Courant--Dorfman algebra $(\cO,\cE,\langle\?,\?\rangle,\partial,[\?,\?])$, one canonically constructs a graded PVA $\commV$ of weight $-1$, freely generated in weights $0$ and $1$, such that $\commV_0=\cO$ and $\commV_1=\cE$, with the $\lambda$-bracket defined on generators by
\begin{equation}\label{eq:intro.lambda-commutator.generators}
\sb{\comme_\lambda\commf}\defeq[\comme,\commf]+\lambda\langle\comme,\commf\rangle,
\qquad \comme,\commf\in\cE.
\end{equation}
This gives a one-to-one correspondence between Courant--Dorfman algebras and graded Poisson vertex algebras of degree $-1$, freely generated in weights $0$ and $1$.

The above one-to-one correspondence is reminiscent of a simpler one between Lie--Rinehart algebras (also known as Lie algebroids in geometry) and graded Poisson algebras of degree $-1$, freely generated in degrees $0$ and $1$. In this case, for each Lie--Rinehart algebra, given by a commutative algebra $\cO$, an $\cO$-module $\cE$, a $\kk$-linear Lie bracket $[\?,\?]$ on $\cE$ and an anchor $\pi\colon\cO\to\Der\cE$, one canonically constructs a Poisson algebra of degree $-1$ on $P=\Sym^\bullet_\cO \cE$, by extending the Lie bracket from $P_1=\cE$ using the Leibniz rule. 

As mentioned in \secref{sub:intro.Lee-filtration-and-arc-spaces}, PVAs arise as classical limits of vertex algebras, 
thus, the latter can be viewed as quantizations or deformations of the 
former. The above correspondence between Courant--Dorfman algebras and 
PVAs appeared first in this quantized context. Indeed, in \cite{gerbes2} (see also 
their previous work \cite{MSV99} and the independent construction of \cite{bd}), the 
authors construct a gerbe of vertex algebras, known as chiral differential 
operators. The Courant--Dorfman algebra and its associated PVA of \secref{sub:intro.Lie-Rinehart}
appear as the lower degrees spaces (degrees 0 and 1) and the associated 
graded PVA of this construction.

\subsection{}\label{sub:intro.the-equivalence}

The main result of this paper is a one-to-one correspondence between double Courant--Dorfman algebras and graded double PVAs of degree $-1$, freely generated in weights $0$ and~$1$ (Theorem~\ref{theorem-CD-DPVA}).
This follows by a noncommutative version of the construction in~\secref{sub:intro.Lie-Rinehart}. 

Recall from~\secref{sub:intro.noncommutative-PDEs} that De Sole, Kac and Valeri~\cite{DSKV15} introduced double PVAs as noncommutative versions of PVAs.
The key ingredient here is the so-called double $\lambda$-bracket
\begin{equation}\label{eq:intro:double-lambda-bracket}
\db{\?{}_\lambda\?}\colon\mathcal{V}\otimes\mathcal{V}\to(\mathcal{V}\otimes\mathcal{V})[\lambda],
\qquad a\otimes b\longmapsto \lr{a{}_{\lambda}b}=\sum_{\mathit{n}\in\mathbb{N}}(a_{\mathit{n}}b)\lambda^{\mathit{n}}, 
\end{equation}
on an (associative) differential algebra $(\mathcal{V},\partial)$, where $a_{\mathit{n}}b\in\mathcal{V}\otimes\mathcal{V}$. 
This is a cross between a double bracket~\eqref{eq:intro:double-bracket} and a $\lambda$-bracket~\eqref{eq:sub:intro.PDEs.2}, which is required to satisfy noncommutative versions of sesquilinearity, the Leibniz rule, skew-symmetry and the Jacobi identity (see~\secref{sec:DPVA}). 
Adding an $\NN$-grading on the vector space underlying $\mathcal{V}$ (called `weight') and imposing the condition that the double $\lambda$-bracket~\eqref{eq:intro:double-lambda-bracket} has weight $-1$, one obtains precisely the data
needed to define a double Courant--Dorfman algebra (see~\secref{sub:intro.Roytenberg-Severa}), namely $A=\mathcal{V}_0$, $E=\mathcal{V}_1$, $\partial$ is the restriction of the given derivation on $\mathcal{V}$, $\ii{a,b}=a_{0}b$ and $\cc{a,b}=a_{1}b$. In~\secref{V-CD}, we prove that these data do satisfy the conditions of Definition~\ref{CD-Definition}. Conversely, using the identities of~Lemma~\ref{lem:CD-identities}, in~\secref{sec:DCD-DPVA} we prove that a double Courant--Dorfman algebra $(A,E,\ii{\?,\?},\partial,\cc{\?,\?})$ determines a graded double PVAs of degree $-1$, 
freely generated in weights $0$ and $1$, whose double Dorfman bracket is defined on generators $e,f\in E$ of degree $1$ by the formula (cf.~\eqref{eq:intro.lambda-commutator.generators})
\begin{equation}\label{eq:intro.double-lambda-commutator.generators}
\db{e_\lambda f}\defeq\cc{e,f}+\lambda\ii{e,f},
\qquad e,f\in E.
\end{equation}

\subsection{}\label{sub:intro.KR-principle}
Our last main result is concerned with the Kontsevich--Rosenberg principle for double Courant--Dorfman algebras
(see~\secref{sub:intro.intro}, and~\secref{sec:KR-double-CD} for details). 
To prove this principle in the context of this article, it is useful to attach an $N\times N$ matrix $(a_{ij})$ 
of formal variables 
to each $a\in A$, and give a presentation of the coordinate ring $A_N$ of the affine scheme $\Rep(A,\MM_N(\kk))$ as the commutative ring generated by $\{a_{ij}\mid a\in A,\,1\leq i,j\leq N\}$, modulo the 
relations (where we sum over repeated indices) 
\begin{equation}\label{eq:intro:KR-relations-defining-AV}
(za)_{ij}=za_{ij},\; (a+b)_{ij}=a_{ij}+b_{ij},\; (ab)_{ij}=a_{ik}b_{kj},\; 1_{ij}=\delta_{ij}\quad (z\in\kk,\, a,b\in A).
\end{equation}
A similar construction~\cite[\S 3.3]{VdB08a} applied to an $A$-bimodule $E$ determines an $A_N$-module $E_N$ generated by the set of indeterminates $\{e_{ij}\mid e\in E,1\leq i,j\leq N\}$, modulo the relations
\begin{equation}\label{eq:intro:KR-relaciones-bimodules}
(e+f)_{ij}=e_{ij}+f_{ij},\;
(ae)_{ij}=a_{ik}e_{kj},\;
(ea)_{ij}=a_{kj}e_{ik} \quad(a\in A,\, e,f\in E).
\end{equation}
These presentations 
provide a convenient way to obtain geometric structures on the representation schemes, 
and have been applied in particular by Van den Bergh~\cite[Proposition 7.5.2]{VdB08} and De Sole, Kac and Valeri~\cite[Theorem 3.22]{DSKV15} to prove that a double Poisson bracket on $A$ and a structure of double PVA on $\mathcal{V}$ induce a Poisson bracket on $A_N$ and a structure of PVA on $\mathcal{V}_N$, respectively. 
%
Applying the same techniques, combined with the identities of Lemma~\ref{lem:CD-identities}, we prove in Theorem~\ref{thm:KR-double-Courant-Dorfman} that double Courant--Dorfman algebras satisfy the Kontsevich--Rosenberg principle too. 
In the last~\secref{sec:KR-exact-double-CD}, we turn our attention to the standard exact double Courant--Dorfman algebras considered in~\secref{sub:intro.standard-CD-alg}. 
We show in Proposition~\ref{prop:KR.standard-complex-DCA.1} that for a double Courant--Dorfman algebra of this type, 
the Courant--Dorfman algebra induced by Theorem~\ref{thm:KR-double-Courant-Dorfman} is a standard exact Courant--Dorfman algebra, with Dorfman bracket given by~\eqref{eq:intro.standard-CD-alg.1}, provided that the underlying algebra $A$ is smooth in the sense of Cuntz--Quillen (see Definition~\ref{def:smoothness}). The smoothness assumption here is required to get the identification $(\DDer A)_N=\Der A_N$~\cite[Proposition 3.3.4]{VdB08a}, so we expect that a possible approach to remove this condition would be to work in the derived representation scheme~\cite[\S 5]{BKR13}. As a preliminary result to prove Proposition~\ref{prop:KR.standard-complex-DCA.1}, we also show that the contraction, the reduced contraction and the Lie derivative along double derivations studied in~\secref{sec:diff-calculus} satisfy the Kontsevich--Rosenberg principle (see Lemma~\ref{lem:KR.standard-complex-DCA.1}).

\subsection{}\label{sub:intro.the-table}

As shown in \cite{vaintrob}, there exists yet another geometric interpretation for the notion of a Lie algebroid $(\cO, \cE)$ in terms of supermanifolds: the exterior algebra $\wedge_\cO \cE$ is the algebra of functions of a symplectic supermanifold of weight $1$. Roytenberg~\cite{Roy00} reformulated this interpretation as a one-to-one correspondence in the language of $\mathbb{N}Q$-manifolds, and subsequently, the first two authors \cite{ACF15} obtained a similar one-to-one correspondence in the noncommutative realm between bisymplectic $\mathbb{N}Q$-algebras of weight 1 and double Poisson algebras.
Thus this article completes the chart below, by showing the relation between double Courant--Dorfman algebras and double Courant algebroids, noncommutative symplectic $\mathbb{N}Q$-geometries and double Poisson vertex algebras:

\begin{center}
\begin{tabular}[center]{|c|c|c|}
\hline
\textbf{Graded geometry} & \textbf{Geometry} & \textbf{Algebra} \\ 
\hline
\hline
   
\begin{tabular}{@{}c@{}}Symplectic $\mathbb{N}Q$-manifold \\  of weight 1\end{tabular}
& \begin{tabular}{@{}c@{}}Lie algebroids/  \\  Lie--Rinehart {algebras}\end{tabular}
& Poisson
\\
\hline

\begin{tabular}{@{}c@{}}Symplectic $\mathbb{N}Q$-manifolds  \\  of weight 2\end{tabular}
   & \begin{tabular}{@{}c@{}}Courant algebroids/ \\   Courant--Dorfman algebras\end{tabular}
 &  PVA
 \\ 
\hline
\hline

\begin{tabular}{@{}c@{}}Bisymplectic $\mathbb{N}Q$-algebras \\  of weight 1\end{tabular}
& \begin{tabular}{@{}c@{}}Double Lie algebroids/  \\  Double Lie--Rinehart {algebras}\end{tabular}
& Double Poisson
 \\ 
\hline
          \begin{tabular}{@{}c@{}}Bisymplectic $\mathbb{N}Q$-algebras  \\  of weight 2\end{tabular}
    &    \begin{tabular}{@{}c@{}}Double Courant algebroids/ \\  Double Courant--Dorfman algebras\end{tabular}
& Double PVA
  \\ 
\hline
\end{tabular}
\end{center}

\subsection*{Acknowledgements}
The first author is partially supported by the Spanish Ministry of Science and Innovation, through the `Severo Ochoa Programme for Centres of Excellence in R\&D' (CEX2019-000904-S) and grant PID2019-109339GB-C31.
The second author was supported by the Alexander von Humboldt Stiftung in the framework of an Alexander von Humboldt professorship endowed by the German Federal Ministry of Education and Research. 
During the initial stage of this work, he was supported by IMPA and CAPES through their postdoctorate of excellence fellowships at UFRJ. He deeply acknowledges their support and excellent working conditions.
The third author is partially supported by CNPq grant number 305688/2019-7. 

The authors wish to thank Henrique Bursztyn, Alejandro Cabrera, Maxime Fairon, Mario Garcia-Fernandez and Estanislao Herscovich for useful conversations. Special thanks are due to Daniele Valeri for very useful email correspondence. 

\section{Notation and conventions} 
\label{sec:notation}

Throughout this article, an algebra will mean an associative (not necessarily commutative) unital algebra over a base field $\kk$ of characteristic 0. 
All unadorned notation will mean over the base field $\kk$. 
In particular, $\otimes\defeq \otimes_\kk$ and $\Hom(-,-):=\Hom_\kk(-,-)$, so a `linear map' and a `bilinear map' will mean a $\kk$-linear map and a $\kk$-bilinear map, respectively. The opposite algebra and the enveloping algebra of an algebra $A$ are denoted $A^{\op}$ and $A^{\e}\defeq
A\otimes A^{\op}$, respectively. We will identify the category of $A$-bimodules and the category of (left) $A^{\e}$-modules.
We will follow the usual convention that tensor products take precedence over direct sums; so, for instance, the target space in \eqref{eq:sub:intro.Roytenberg-Severa.1} will simply be written as $E\otimes A\oplus A\otimes E$.
The sets of integers and non-negative integers will be denoted $\mathbb{Z}$ and $\mathbb{N}$, respectively.

As our work relies heavily on the seminal papers~\cite{CBEG07,CQ95,DSKV15,VdB08}, we will adopt conventions and constructions from them. 
First, we fix an algebra $R$ and often work in the category of \emph{$R$-algebras}~\cite[\S 2]{CQ95}.
An object of this category is an algebra $A$ equipped with a `structure' homomorphism $R\to A$, and a morphism $A\to B$ between two $R$-algebras is an algebra homomorphism whose composition with the structure homomorphism of $A$ is the structure homomorphism of $B$. The category of algebras is recovered in the `absolute' case $R=\kk$.

Given two $A$-bimodules $E$ and $F$, their tensor product $E\otimes F$ has two commuting $A$-bimodule structures, called the \emph{outer} and the \emph{inner} $A$-bimodule structures~\cite{CBEG07,VdB08}.
The corresponding $A$-bimodules are denoted $(E\otimes F)_{\out}$ and $(E\otimes F)_{\inn}$, with the bimodule multiplication given for all $a',a''\in A$ and $e\otimes f\in E\otimes F$ by
\begin{subequations}\label{eq:outer-inner-multiplication.1}
\begin{align}\label{eq:outer-multiplication.1}
a'(e\otimes f) a''&\defeq (a'e)\otimes (fa'')\quad  \text{ on } (E\otimes F)_\out,
\\\label{eq:inner-multiplication.1}
a'*(e\otimes f) *a''&\defeq (ea'')\otimes (a'f) \quad \text{ on } (E\otimes F)_\inn.
\end{align}\end{subequations}
In particular, the $A$-bimodule underlying $A$ determines two $A$-bimodules $A^{\otimes 2}_{\out}$ and $A^{\otimes 2}_{\inn}$ with the same underlying vector space $A^{\otimes 2}=A\otimes A$. Unless explicitly stated, we endow the tensor product $E\otimes F$ (and in particular $A^{\otimes 2}_{\out}$) with the outer $A$-bimodule structure.  
Furthermore, we make systematic use of Sweedler's notation, consisting of dropping the summation sign, that is, if $v\in E\otimes F$, we write
\begin{equation}\label{eq:Sweedler-general-convention}
v=v^{\prime}\otimes v^{\pprime}\in E\otimes F, 
\end{equation}
where $v^{\prime}\in E$ and $v^{\pprime}\in F$.

More generally~\cite[\S 1.3]{DSKV15}, given $n\in\mathbb{N}_{\geq 2}$, for every $i=0,\dots ,n-1$, we define the $i$-th left and right $A$-module structures of $A^{\otimes n}$ by
\begin{equation}\label{square-op}
\begin{aligned}
a^\prime*_i(a_1\otimes\cdots \otimes a_n)&\defeq a_1\otimes\cdots\otimes a_i\otimes a^\prime a_{i+1}\otimes\cdots\otimes a_n,
\\
(a_1\otimes\cdots\otimes a_n)*_i a^{\pprime}&\defeq a_1\otimes\cdots\otimes a_{n-i}a^{\pprime} \otimes\cdots\otimes a_n,
\end{aligned}
\end{equation}
for all $a^{\prime},a^{\pprime},a_1,\dots,a_n\in A$, 
where the index $i$ denotes the number of `jumps'. Thus the outer multiplication~\eqref{eq:outer-multiplication.1} is $a^{\prime}*_0(a_1\otimes a_2) *_0 a^{\pprime}$, whilst the inner multiplication~\eqref{eq:inner-multiplication.1} is $a^{\prime}*_1(a_1\otimes a_2)*_1 a^{\pprime}$. Similarly, 
given $A$-bimodules $E_i$ and $e_i\in E_i$ for $i=1,2,3$, we define
\begin{equation}\label{jumping-notation}
\begin{aligned}
e_1\otimes_1 (e_2\otimes e_3) &\defeq e_2\otimes e_1\otimes e_3\in E_2\otimes E_1\otimes E_3,
\\
(e_1\otimes e_2)\otimes_1 e_3&\defeq e_1\otimes e_3\otimes e_2\in E_1\otimes E_3\otimes E_2.
\end{aligned}
\end{equation}

The \emph{dual} of an $A$-bimodule $E$ is the $A$-bimodule~\cite[\S 5.3]{CBEG07}
\begin{equation}\label{eq:dual-def}
E^\vee:=\Hom_{A^\e}(E,A^{\otimes 2}_\out),
\end{equation}
where the $A$-bimodule structure is induced by the one in $A^{\otimes 2}_\inn$ (\emph{i.e.}, $(a\cdot\phi\cdot b)e=a*(\phi e)*b=(\phi^{\prime}e)b\otimes a(\phi^{\pprime}e)$, for all $\phi\in E^\vee$, $e\in E$, $a,b\in A$).

Given vector spaces $V_1,\dots,V_n$, the symmetric group $\mathbb{S}_n$ on the set $\{1,\dots,n\}$ acts on the tensor products $V_1\otimes\cdots\otimes V_n$ by permuting the factors. Using the notation of~\cite[\S 2.1]{VdB08}, 
\begin{equation}\label{eq:permtutation-action.1}
\tau_s\colon V_1\otimes\cdots\otimes V_n\lto V_{s^{-1}(1)}\otimes\cdots\otimes V_{s^{-1}(n)},
\end{equation}
is the linear map determined by an element $s\in\mathbb{S}_n$, given by the formula
\begin{equation*}
\label{eq:permtutation-action.2}
\tau_s(v)=v_{s^{-1}(1)}\otimes\cdots\otimes v_{s^{-1}(n)},
\end{equation*}
for all $v=v_1\otimes\cdots\otimes v_n$, with $v_i\in V_i$.
If $s=(1\dots n)$ is the cyclic permutation that sends $1$ to $2$, we use the notation
\begin{equation}\label{eq:cyclic-permutation.1}
\big(v_1\otimes\cdots\otimes v_{n-1}\otimes v_n\big)^\sigma\defeq v_n\otimes v_1\otimes\cdots\otimes v_{n-1}.
\end{equation}

\section{Basics on noncommutative geometry}
\label{sub:basics}

%
The basic idea of the approach to noncommutative geometry initiated in~\cite{CBEG07, VdB08} is that noncommutative differential forms and double derivations over a noncommutative algebra (see~\secref{sec:nc-diff-forms-double-derivations}) should play the roles of differential forms and vector fields over a manifold. 
This is a fruitful view point that provided the basic ingredients in Van den Bergh's construction of the so-called double Poisson structures and double Schouten--Nijenhuis bracket of polyvector fields.  Both structures and the related notion of pairing on a bimodule, reviewed in~\secref{sec:double-Poisson}--\secref{sec:pairings}, will be used throughout this paper.

\subsection{Noncommutative differential forms and double derivations}
\label{sec:nc-diff-forms-double-derivations}

Let $R$ be an algebra, $A$ an $R$-algebra (with unit $1$) and $E$ an $A$-bimodule.
A linear map $\theta\colon A\to E$ is called a \emph{derivation} if it satisfies the Leibniz rule
$\theta(ab) =(\theta a)b+a(\theta b)$ for all $a,b\in A$.
A derivation $\theta$ is called \emph{$R$-linear} (cf.~\cite[\S 2.1]{CBEG07}), or an \emph{$R$-derivation} (cf.~\cite[p. 260]{CQ95}), if it vanishes on the image of $R$. 
The vector space of $R$-linear derivations $\theta\colon A\to E$ is denoted $\Der_R(A,E)$; if $E=A$ (regarded as a bimodule over itself), we simply write $\Der_RA$.

The $A$-bimodule $\Omega^1_{R} A$ of \emph{noncommutative relative differential 1-forms} on $A$ (with respect to $R$) is the kernel of the multiplication map $A\otimes_RA\to A$. 
It carries a canonical \emph{universal derivation}, defined as the $R$-linear derivation $\du{}\colon A\to \Omega^1_{R} A$, $a\mapsto\du{a}\defeq a\otimes 1-1\otimes a$. Cuntz--Quillen \cite{CQ95} proved that the pair $(\Omega^1_{R}A,\du{})$ satisfies the following universal property: for any $R$-linear derivation $\theta\colon A\to E$, there exists a unique $A$-bimodule morphism $i_\theta\colon\Omega^1_{ R}A\to E$ such that $\theta=i_\theta\circ\du{}$. 
In other words, there exists a canonical isomorphism of $A$-bimodules
\begin{equation}\label{sub:ncdiffforms.corep}
\Der_R(A,E)\lra{\cong}\Hom_{A^\e}(\Omega^1_{R}A,E),\quad \theta\longmapsto i_\theta,
\end{equation}
that is natural in the $A$-bimodule $E$, and whose inverse is given by
$\varphi\mapsto\varphi\circ\du{}$.
If $R=\kk$, we simply write $\Omega^1 A$.

It is clear that if $A$ is finitely generated over $R$, then $\Omega^1_RA$ is finitely generated as an $A$-bimodule. Sometimes one needs a stronger condition, introduced by Cuntz--Quillen~\cite{CQ95}. 

\begin{definition}[{\cite[Definition 2.1.1]{CBEG07}}]
\label{def:smoothness}
An $R$-algebra $A$ is called \emph{smooth over $R$} if it is finitely generated as an $R$-algebra and $\Omega^1_RA$ is projective as an $A$-bimodule. 
\end{definition}

Path algebras of quivers are prototypical examples of smooth algebras, as are the more general tensor algebras $A=T_BM$ when $M$ is a finitely generated projective bimodule over an $R$-algebra $B$ that is smooth over $R$ (see \cite[Proposition 5.3(3)]{CQ95}).

The \emph{algebra of noncommutative relative differential forms} on $A$ is the tensor algebra 
\begin{equation}\label{eq:relative-nc-all-diff-forms.1}
\Omega^\bullet_RA\defeq T_A(\Omega^1_RA)=\bigoplus_{i\in\mathbb{N}}\big(\Omega^1_RA\big)^{\otimes_A i}
\end{equation}
of the $A$-bimodule $\Omega^1_RA$.
This is a differential graded algebra with $\Omega^0_R A=A$ and differential
\begin{equation}\label{eq:relative-nc-all-diff-forms.2}
\du{}\colon \Omega^\bullet_RA\lto\Omega^{\bullet+1}_RA
\end{equation}
of degree 1 given by extension of the universal derivation $\du{}\colon A\to\Omega^1_R A$ to an $R$-linear differential, that is, an $R$-linear derivation such that
\begin{equation}\label{eq:relative-nc-all-diff-forms.3}
\du{}^2=0.
\end{equation}
Note that if the algebra $A$ is finitely generated over $R$, then so is $\Omega^\bullet_R A$, since $\Omega^1_R A$ is finitely generated as an $A$-bimodule in this case.
The differential graded algebra $(\Omega^\bullet_RA,\du{})$ is universal with respect to $R$-algebra morphisms $u\colon A\to\Gamma^0$ into the degree 0 subalgebra of differential graded $R$-algebras $(\Gamma^\bullet,\du_\Gamma{})$ (see~\cite[Proposition~2.1]{CQ95}).

Since the differential graded algebras $(\Omega^\bullet_RA,\du{})$ do not give rise to interesting cohomology theories (see, e.g.,~\cite[(2.5.1)]{CBEG07}), following Karoubi, we define the \emph{Karoubi--de Rham complex} as the differential graded vector space $(\DR{\bullet}{A},\du{}_{\operatorname{DR}})$, consisting of the graded vector space 
\begin{equation}\label{eq:def-Karoubi-de-Rham}
\DR{\bullet}{A}\defeq \Omega^\bullet_RA/[\Omega^\bullet_RA,\Omega^\bullet_RA],
\end{equation}
where $[\Omega^\bullet_RA,\Omega^\bullet_RA]$ denotes the $\kk$-linear span of all graded commutators, and the differential
\begin{equation}\label{eq:def-Karoubi-de-Rham-differential}
\du{}=\du{}_{\operatorname{DR}}\colon \DR{\bullet}{ A}\to \DR{\bullet+1}{A}
\end{equation}
induced by~\eqref{eq:relative-nc-all-diff-forms.2} (if there is no possible confusion, we simply denote $\du{}_{\operatorname{DR}}$ by $\du{}$).

Taking $E=A^{\otimes 2}_{\out}$ in \eqref{sub:ncdiffforms.corep}, we obtain a canonical isomorphism of $A$-bimodules
\begin{equation}\label{eq:DDer-isom.1}
i\colon\D_R A\lra{\cong} (\Omega^1_R A)^\vee,\quad X\longmapsto i_X,
\end{equation}
where $(\Omega^1_R A)^\vee=\Hom_{A^\e}(\Omega^1_R A, A^{\otimes 2}_{\out})$ (see \eqref{eq:dual-def}), and
\begin{equation}\label{eq:double-derivation.1}
\DDer_RA\defeq\Der_R \!\big(A, A^{\otimes 2}_{\out}\big)
\end{equation}
is the $A$-bimodule of \emph{$R$-linear double derivations} of $A$.
Its $A$-bimodule structure is induced by the inner $A$-bimodule structure of $A^{\otimes 2}$, that is, for all $X\in\D_RA$, $a,b,c\in A$, 
\begin{equation}\label{eq:double-derivation.A-bimod-str.1}
(b\cdot X\cdot c)(a)=b*(Xa)*c=(X'a)c\otimes b(X''a),
\end{equation}
where we use the inner multiplication~\eqref{eq:inner-multiplication.1} and Sweedler's notation
\begin{equation}\label{eq:Sweedler-double-derivations.1}
Xa=X'a\otimes X''a\in A^{\otimes 2},
\end{equation}
for all $a\in A$, to omit the summation index. 
Note also that $Xr=0$ for $r\in R$.
%

The isomorphism~\eqref{eq:DDer-isom.1} determines another useful $A$-bimodule morphism (cf.~\cite[\S 5.3]{CBEG07})
\begin{equation}\label{eq:canonical-map.1}
\mathtt{bidual}\defeq i^\vee\circ\ev\colon\Omega^1_RA\lto(\DDer_RA)^\vee,\quad\alpha\longmapsto\alpha^\vee, 
\end{equation}
defined as the composite of the evaluation map
\begin{equation}\label{eq:canonical-map.2}
\ev\colon\Omega^1_RA\lto(\Omega^1_RA)^{\vee\vee}
\end{equation}
and the dual of the $A$-bimodule morphism~\eqref{eq:DDer-isom.1}.
Here, $(\Omega^1_RA)^{\vee\vee}$ is the double dual $((\Omega^1_RA)^\vee)^\vee$, the evaluation map is given for all $\alpha\in\Omega^1_RA$ and $\delta\in(\Omega^1_RA)^\vee$ by 
\[
\ev(\alpha)(\delta)\defeq\delta(\alpha)^\sigma
\]
(the action of $\sigma\in\mathbb{S}_2$ is needed to obtain an $A$-bimodule morphism), and the dual of~\eqref{eq:DDer-isom.1} is the $A$-bimodule isomorphism
\begin{equation}\label{eq:canonical-map.4}
i^\vee\colon(\Omega^1_RA)^{\vee\vee}\lra{\cong}(\DDer_RA)^\vee
\end{equation}
given for all $\lambda\in(\Omega^1_RA)^{\vee\vee}$ and $X\in\DDer_RA$ by
\[
(i^\vee\lambda)(X)=\lambda(i_X).
\]
It follows that the composite~\eqref{eq:canonical-map.1} is given for all $\alpha\in\Omega^1_RA$ and $X\in\DDer_RA$ by
\begin{equation}\label{eq:canonical-map.3}
\alpha^\vee(X)\defeq\mathtt{bidual}(\alpha)(X)
=(i_X\alpha)^\sigma.
\end{equation}
Moreover,~\eqref{eq:canonical-map.1} is an isomorphism if $A$ is smooth over $R$, because so is~\eqref{eq:canonical-map.2} in this case.


\subsection{Double Poisson algebras}
\label{sec:double-Poisson}
Let $R$ be an algebra and $A$ an $R$-algebra. 
An \emph{$R$-linear double bracket} \cite{VdB08} on $A$ (or simply a \emph{double bracket} if $R=\kk$) is a bilinear map
\begin{equation}\label{eq:double-bracket.1}
\lr{-,-}\colon A\times A\lto A^{\otimes 2},\quad (a,b)\longmapsto\lr{a,b},
\end{equation}
that satisfies the following identities for all $a,b,c\in A$ and $r\in R$:
\begin{subequations}
\label{eq:double-bracket}
\begin{align}
\lr{a,r}&=0,\label{eq:double-bracket.0}
\\
\lr{a,b}&=-\lr{b,a}^\sigma,\label{eq:double-bracket.a}
\\
\lr{a,bc}&=b\lr{a,c}+\lr{a,b}c, \label{eq:double-bracket.b}
\end{align}
\end{subequations}
The identities \eqref{eq:double-bracket.a} and \eqref{eq:double-bracket.b} are called \emph{skewsymmetry} and the \emph{Leibniz rule}, respectively.
Note that~\eqref{eq:double-bracket.0} and~\eqref{eq:double-bracket.b} mean that $\lr{a,-}$ is an $R$-linear double derivation in its second entry (with respect to the outer bimodule structure on $A^{\otimes 2}$). Also, \eqref{eq:double-bracket.a} and \eqref{eq:double-bracket.b} imply that $\lr{ab,c}=a*\lr{b,c}+\lr{a,c}*b$; in other words, an $R$-linear double bracket is an $R$-linear derivation in its first entry for the inner bimodule structure on $A^{\otimes 2}$. 

The analogue of the Jacobi identity for double brackets is formulated using the following extensions of the double bracket $\lr{-,-}$. Given $a\in A$, $b=b_1\otimes\cdots\otimes b_n \in A^{\otimes n}$, we define
\begin{equation}
\begin{aligned}
\lr{a,b}_L&\defeq \lr{a,b_1}\otimes b_2\otimes \cdots\otimes b_n\in A^{\otimes(n+1)},
\\
\lr{a,b}_R&\defeq b_1\otimes\cdots\otimes b_{n-1}\otimes\lr{a,b_n}\in A^{\otimes (n+1)}.
\label{extension-left}
\end{aligned}
\end{equation}

\begin{definition}[{\cite[Definition 2.3.2]{VdB08}}]
\label{def:double-Poisson algebra}
A double bracket $\lr{-,-}$ is called \emph{Poisson} if
\begin{equation}\label{double-Jacobi-Poisson}
\lr{a,\lr{b,c}}_L+\tau_{(123)}\lr{b,\lr{c,a}}_L+\tau_{(132)}\lr{c,\lr{a,b}}_L=0.
\end{equation}
An ($R$-linear) \emph{double Poisson algebra} is a pair $(A,\lr{-,-})$ consisting of an $R$-algebra $A$ and an ($R$-linear) double Poisson bracket $\lr{-,-}$ on $A$.
\end{definition}

Equation \eqref{double-Jacobi-Poisson} is called the \emph{double Jacobi identity}.
This identity was reformulated by De Sole--Kac--Valeri~\cite[Remark 2.2]{DSKV15} in a way that is particularly convenient for double Poisson vertex algebras. In this formulation, the double bracket is extended in the first entry using the `inner' tensor product~\eqref{jumping-notation}, \emph{i.e.}, 
given $a\in A$, $b=b^{\prime}\otimes b^{\pprime}\in A\otimes A$, we set
\begin{equation}\label{notation-Kac}
\begin{aligned}
\lr{b^\prime\otimes b^{\prime\prime},a}_L&\defeq \lr{b^\prime,a}\otimes_1 b^{\prime\prime}=\lr{b^\prime,a}^{\prime}\otimes b^{\prime	\prime}\otimes\lr{b^\prime,a}^{\prime\prime},
\\
\lr{b^\prime\otimes b^{\prime\prime},a}_R&\defeq b^\prime\otimes_1\lr{b^{\prime\prime},a}=\lr{b^{\prime\prime},a}^\prime\otimes b^{\prime}\otimes \lr{b^{\prime\prime},a}^{\prime\prime},
\end{aligned}
\end{equation}
where in the right-hand sides we are using Sweedler's notation~\eqref{eq:Sweedler-general-convention}, that is, for all $a,b\in A$,
\begin{equation}\label{eq:Sweedler-double-bracket}
\lr{a,b}=\lr{a,b}^\prime\otimes\lr{a,b}^{\pprime}.
\end{equation}
Then, by \eqref{notation-Kac} and \eqref{eq:double-bracket.a}, the double Jacobi identity \eqref{double-Jacobi-Poisson} is equivalent to the identity
\begin{equation}
\lr{a,\lr{b,c}}_L=\lr{\lr{a,b},c}_L+\lr{b,\lr{a,c}}_R.
\label{Jacobi-Kac}
\end{equation}

\subsection{The double Schouten--Nijenhuis bracket}
\label{sec:double-SN}

Following~\cite[\S 3.2]{VdB08}, we will now define appropriate noncommutative analogues of the Lie bracket of vector fields and, more generally, the Schouten--Nijenhuis bracket of polyvector fields, when the vector fields on a manifold are replaced by the double derivations, that is, elements of~\eqref{eq:double-derivation.1}. 
Let $R$ be an algebra and $A$ an $R$-algebra that is finitely generated over $R$.
To define the analogue of the Lie bracket for double derivations, we will use the actions
\begin{subequations}\label{eq:R-linear-double-SN-bracket.4}
\begin{align}
\tau_{(23)*} \colon\Der_R( A, A^{\otimes 3}_\out)\lra{\cong}&
\Der_R( A, A^{\otimes 2}_\out\otimes  A),\quad X\longmapsto\tau_{(23)}\circ X,
\label{eq:R-linear-double-SN-bracket.4.a}
\\
\tau_{(12)*} \colon\Der_R( A, A^{\otimes 3}_\out)\lra{\cong}&
\Der_R( A, A\otimes  A^{\otimes 2}_\out),\quad 
X\longmapsto \tau_{(12)}\circ X, \label{eq:R-linear-double-SN-bracket.4.b}
\end{align}
\end{subequations}
of the permutations $(23), (12)\in\mathbb{S}_3$, given by composing with the maps
\begin{equation}\label{eq:graded-permutations.S_3}
\begin{aligned}
\tau_{(12)}\colon A^{\otimes 3}\lto A^{\otimes 3},\quad a\otimes b\otimes c\longmapsto b\otimes a \otimes c,
\\
\tau_{(23)}\colon A^{\otimes 3}\lto A^{\otimes 3},\quad a\otimes b\otimes c\longmapsto a\otimes c\otimes b,
\end{aligned}
\end{equation}
and the evaluation maps
\begin{subequations}\label{eq:graded-eval-map.1}
\begin{align}\label{eq:graded-eval-map.1.a}
\varepsilon_l\colon(\DDer_R A)\otimes A&\lra{\cong}\Der_R( A, A^{\otimes 2}_\out\otimes A),
\\\label{eq:graded-eval-map.1.b}
\varepsilon_r\colon A\otimes(\DDer_R A)&\lra{\cong}\Der_R( A, A\otimes A^{\otimes 2}_\out),
\end{align}
\end{subequations}
given for all $a,b\in A$, $X\in\DDer_R A$ by 
\begin{subequations}\label{eq:graded-eval-map.2}
\begin{align}
\varepsilon_l(X\otimes a)(b)&=(Xb)\otimes a, \label{eq:graded-eval-map.2.a}
\\
\varepsilon_r(a \otimes X)(b)&= a\otimes (Xb). \label{eq:graded-eval-map.2.b}
\end{align}
\end{subequations}
Note that~\eqref{eq:R-linear-double-SN-bracket.4} are isomorphisms because so are~\eqref{eq:graded-permutations.S_3}, and~\eqref{eq:graded-eval-map.1} are isomorphisms because $A$ is finitely generated over $R$.
By direct calculation~\cite[Proposition 3.2.1]{VdB08}, one can show that the following formulae define $R$-linear triple derivations for all $X,Y\in\DDer_RA$:
\begin{subequations}\label{eq:Schouten-Nijenhuis.3}
\begin{gather}\label{eq:Schouten-Nijenhuis.3.a}
\db{X,Y}_l^{\sim}\defeq(X\otimes\Id_A)\circ Y - (\Id_A\otimes Y)\circ X\colon A\lto A^{\otimes 3}_\out,
\\\label{eq:Schouten-Nijenhuis.3.b}
\db{X,Y}_r^{\sim}\defeq(\Id_A\otimes X)\circ Y - (Y\otimes\Id_A)\circ X\colon A\lto A^{\otimes 3}_\out.
\end{gather}\end{subequations}
By inspection, $\db{X,Y}_l^{\sim}=-\db{Y,X}_r^{\sim}$.
Then the double Schouten--Nijenhuis bracket 
\begin{equation}\label{eq:R-linear-double-SN-bracket.1}
\db{\?,\?}\colon\DDer_R A\times\DDer_R A\lto(\DDer_R A)\otimes A\oplus A\otimes(\DDer_R A)
\end{equation}
of $R$-linear double derivations on $A$ is given by 
\begin{equation}\label{eq:R-linear-double-SN-bracket.2}
\db{X,Y}\defeq\db{X,Y}_l+\db{X,Y}_r, 
\end{equation}
for all $X,Y\in\DDer_R A$, where
\begin{subequations}\label{eq:R-linear-double-SN-bracket.3}
\begin{gather}\label{eq:R-linear-double-SN-bracket.3.a}
\db{X,Y}_l=\varepsilon_l^{-1}\tau_{(23)*}(\db{X,Y}_l^{\sim})\in\DDer_R A\otimes A,
\\\label{eq:R-linear-double-SN-bracket.3.b}
\db{X,Y}_r=\varepsilon_r^{-1}\tau_{(12)*}(\db{X,Y}_r^{\sim})\in A\otimes\DDer_R A.
\end{gather}\end{subequations}

As in~\eqref{eq:Sweedler-double-bracket}, it is convenient to use Sweedler's notation, so we write 
\begin{subequations}\label{eq:Sweedler.R-linear-double-SN-bracket.1}
\begin{gather}\label{eq:Sweedler.R-linear-double-SN-bracket.1.a}
\db{X,Y}_l=\db{X,Y}^{\prime}_l\otimes \db{X,Y}^{\pprime}_l\in\DDer_R A\otimes A,
\\\label{eq:Sweedler.R-linear-double-SN-bracket.1.b}
\db{X,Y}_r=\db{X,Y}^{\prime}_r\otimes \db{X,Y}^{\pprime}_r\in A\otimes\DDer_R A,
\end{gather}\end{subequations}
with $\db{X,Y}^{\prime}_l,\db{X,Y}^{\pprime}_r\in\D_R A$, $\db{X,Y}^{\pprime}_l,\db{X,Y}^{\prime}_r\in A$.
Then the triple derivations \eqref{eq:Schouten-Nijenhuis.3} are given for all $a\in A$ in terms of \eqref{eq:R-linear-double-SN-bracket.3}, using the inner tensor products~\eqref{jumping-notation}, by 
\begin{subequations}\label{eq:graded-bracket.1}
\begin{gather}\label{eq:graded-bracket.1.a}
\db{X,Y}_l^\sim(a)=\tau_{(23)*}^{-1}\Big(\db{X,Y}_l'(a)\otimes\db{X,Y}_l''\Big)=\db{X,Y}_l'(a)\otimes_1\db{X,Y}_l'',
\\\label{eq:graded-bracket.1.b}
\db{X,Y}_r^\sim(a)=\tau_{(12)*}^{-1}\Big(\db{X,Y}_r'\otimes\db{X,Y}_r''(a)\Big)=\db{X,Y}_r'\otimes_1\db{X,Y}_r''(a).
\end{gather}\end{subequations}

The map~\eqref{eq:R-linear-double-SN-bracket.1} admits an extension to a `double Gerstenhaber bracket' on the tensor algebra $T_A(\D_R A)$, such that for all $a,b\in A$ and $X\in\D_R A$, 
\begin{equation}\label{eq:def-parte-trivial-SN}
\lr{a,b}=0,\quad \lr{X,a}=X a,
\quad \lr{a,X}=-(Xa)^\sigma,
\end{equation}
that will play the role of the Schouten--Nijenhuis bracket of polyvector fields on a manifold. 
We introduce now some notation about graded algebras to describe this structure as a graded version of a double Poisson bracket~\eqref{eq:double-bracket.1} (see~\cite[\S2.7]{VdB08}, \cite[\S 5.1]{BCER12}, \cite[\S3]{FH}).
A non-zero element $v$ of a graded vector space $V=\bigoplus_{k\in\mathbb{Z}}V^k$  has \emph{degree} $\wt{v}=k$ if $v\in V^k$.
A non-zero linear map $f\colon V\to W$ between graded vector spaces has \emph{degree} $\wt{f}=\ell$ if $f(V^k)\subset W^{k+\ell}$ for all $k\in\ZZ$.
When we write $\wt{v}$ or $\wt{f}$, we implicitly assume $v$ or $f$ is homogeneous.
We systematically use the Koszul sign rule, so the action 
\begin{equation}\label{eq:graded-permtutation-action.1}
\sigma_s\colon V_1\otimes\cdots\otimes V_n\lto V_{s^{-1}(1)}\otimes\cdots\otimes V_{s^{-1}(n)}
\end{equation}
of a permutation $s\in\mathbb{S}_n$ on the tensor product of some graded vector spaces $V_1,\ldots,V_n$ is given by
\begin{equation}\label{eq:graded-permtutation-action.2}
\sigma_s(v)\defeq(-1)^{\mathfrak{t}(v)} v_{s^{-1}(1)}\otimes\cdots\otimes v_{s^{-1}(n)},
\end{equation}
for all $v=v_1\otimes\cdots\otimes v_n$, with $v_i\in V_i$, where the sign is given by the Koszul sign rule, \emph{i.e.}, 
\begin{equation}\label{eq:Koszul-sign-rule.1}
\mathfrak{t}(v)=\sum_{\substack{i<j\\s^{-1}(i)>s^{-1}(j)}}\wt{v_{s^{-1}(i)}}\wt{v_{s^{-1}(j)}}.
\end{equation}
A \emph{graded $R$-algebra} is a graded algebra $\cA=\bigoplus_{k\in\mathbb{Z}}\cA^k$ equipped with a structure morphism $R\to\cA$ of degree 0, where $R$ is regarded as a graded algebra concentrated in degree 0. 
Given a graded $\cA$-bimodule $E$ and $\ell\in\mathbb{Z}$, an \emph{$R$-linear derivation $\theta\colon\cA\to E$ of degree $\ell$} is a linear map of degree $\ell$ such that $\theta(\alpha\beta)=(\theta\alpha)\beta+(-1)^{\ell\wt{\alpha}}\alpha(\theta\beta)$ for all $\alpha,\beta\in\cA$, that vanishes on the image of $R$. 
The graded vector space of $R$-linear derivations $\theta\colon\cA\to E$ is 
\begin{equation}\label{eq:graded-derivations.1}
\Der^\bullet_R(\cA,E)\defeq\bigoplus_{\ell\in\ZZ}\Der^\ell_R(\cA,E), 
\end{equation}
where $\Der^\ell_R(\cA,E)$ is the vector space of $R$-linear derivations $\theta\colon\cA\to E$ of degree $\ell$.
In particular, when $E=\cA^{\otimes 2}_\out$, 
we obtain the graded $\cA$-bimodule
\begin{equation}\label{eq:graded-double-derivations.1}
\DDer^\bullet_R\cA=\bigoplus_{\ell\in\ZZ}\DDer^\ell_R\cA, 
\end{equation}
where $\DDer^\ell_R\cA\defeq\Der^\ell_R(\cA,\cA^{\otimes 2}_\out)$ is the space of \emph{$R$-linear double derivations of degree $\ell$}. 

An \emph{$R$-linear double Poisson bracket of degree $\ell$} on a graded $R$-algebra $\cA$ is a graded map
\[
\lr{-,-}\colon\cA\times\cA\lto\cA\otimes\cA
\]
of degree $\ell$ that satisfies the following identities for all $r\in R$ and $\alpha,\beta,\gamma\in\cA$:
\begin{subequations}
\label{eq:def:double-Gerstenhaber-algebra.1.2}
\begin{align}
\label{eq:def:double-Gerstenhaber-algebra.1.2.0}
\db{\alpha,r}&=0,
\\
\label{eq:def:double-Gerstenhaber-algebra.1.2.a}
\db{\alpha,\beta}&=-(-1)^{(\wt{\alpha}+\ell)(\wt{\beta}+\ell)}\db{\beta,\alpha}^\sigma,
\\
\label{eq:def:double-Gerstenhaber-algebra.1.2.b}
\db{\alpha,\beta\gamma}&=\db{\alpha,\beta}\gamma+(-1)^{(\wt{\alpha}+\ell)\wt{\beta}}\beta\db{\alpha,\gamma},
\\
\begin{split}
\label{eq:def:double-Gerstenhaber-algebra.1.2.c}
\db{\alpha,\lr{\beta,\gamma}}_L&
+(-1)^{(\wt{\alpha}+\ell)(\wt{\beta}+\wt{\gamma})}\sigma_{(123)}\db{\beta, \db{\gamma,\alpha}}_L
\\&
+(-1)^{(\wt{\gamma}+\ell)(\wt{\alpha}+\wt{\beta})}\sigma_{(132)}\db{\gamma,\db{\alpha,\beta}}_L=0.
\end{split}
\end{align}
\end{subequations}
Note that~\eqref{eq:def:double-Gerstenhaber-algebra.1.2.0} and~\eqref{eq:def:double-Gerstenhaber-algebra.1.2.b} mean that $\lr{\alpha,-}\in\DDer^{\wt{\alpha}+\ell}_R\cA$, for all $\alpha\in\cA$. 

An ($R$-linear) \emph{double Poisson algebra of degree $\ell$} is a pair $(\cA,\lr{-,-})$ consisting of a graded $R$-algebra $\cA$ and an ($R$-linear) double Poisson bracket $\lr{-,-}$ of degree $\ell$ on $\cA$.
In particular, an \emph{double Gerstenhaber algebra} is a double Poisson algebras of degree $-1$. 

While Definition~\ref{def:double-Poisson algebra} is recovered from these notions when $\ell=0$, double Gerstenhaber algebras play a major role in our work due to the following result, that provides an interpretation of $T_A(\D_R A)$ as a noncommutative version of the algebra of polyvector fields. 
 
\begin{theorem}[{\cite[Theorem 3.2.2]{VdB08}}]\label{thm:VdB-Gerstenhaber-poly-der}
The definitions in~\eqref{eq:R-linear-double-SN-bracket.1} and~\eqref{eq:def-parte-trivial-SN} define a unique structure of $R$-linear double Gerstenhaber algebra on the graded tensor algebra $T_A(\D_R A)$, where $\D_R A$ is placed in degree $1$.
\end{theorem}

The $R$-linear double Gerstenhaber bracket $\lr{\?,\?}$ given by Theorem~\ref{thm:VdB-Gerstenhaber-poly-der} is called the \emph{double Schouten--Nijenhuis bracket} on $T_A(\D_R A)$.
For future calculations, we record now the skewsymmetry identity, the Leibniz rule and the double Jacobi identity for this bracket (they follow from Theorem \ref{thm:VdB-Gerstenhaber-poly-der}; see \cite[eqs. (3.4)--(3.8)]{VdB08}). 
For all $a\in A$, $X,Y,Z\in\D_RA$,
\begin{subequations}
\label{eq:double-SN-explicit}
\begin{align}
\lr{X,Y}&=-\lr{Y,X}^\sigma,
\label{eq:double-Lie-albd.axioms.double-bracket.skewsymm.1}
\\
\lr{X,aY}&=a\lr{X,Y}+(Xa)\cdot Y, 
\label{eq:double-Lie-albd.axioms.double-Leibniz.1.a}
\\
\lr{X,Ya}&=\lr{X,Y}a+Y\cdot (Xa),
\label{eq:double-Lie-albd.axioms.double-Leibniz.1.b}
\\
\lr{a,\lr{X,Y}}_L&+\sigma_{(123)}\lr{X,\lr{Y,a}}_L+\sigma_{(132)}\lr{Y,\lr{a,X}}_L=0,
\label{eq:double-Lie-albd.axioms.double-Jacobi.1-simple}
\\
\lr{X,\lr{Y,Z}}_L&+\sigma_{(123)}\lr{Y,\lr{Z,X}}_L+\sigma_{(132)}\lr{Z,\lr{X,Y}}_l=0.
\label{eq:double-Lie-albd.axioms.double-Jacobi.1}
\end{align}
\end{subequations}
Note that while~\eqref{eq:double-Lie-albd.axioms.double-bracket.skewsymm.1}--\eqref{eq:double-Lie-albd.axioms.double-Leibniz.1.b} take place in $\D_R A\otimes A\oplus A\otimes\D_R A$, \eqref{eq:double-Lie-albd.axioms.double-Jacobi.1-simple} is in $A^{\otimes 3}$, and~\eqref{eq:double-Lie-albd.axioms.double-Jacobi.1} is in $\D_R A\otimes A^{\otimes 2}\oplus A\otimes\D_R A\otimes A\oplus A^{\otimes 2}\otimes \D_R A$.
Moreover,~\eqref{eq:double-Lie-albd.axioms.double-bracket.skewsymm.1} implies
\begin{equation}\label{eq:double-Lie-albd.axioms.double-bracket.skewsymm.2}
\lb{X,Y}_r=-\lb{Y,X}_l^\sigma,
\end{equation}
so using Sweedler's notation~\eqref{eq:Sweedler.R-linear-double-SN-bracket.1}, we can take 
\begin{equation}\label{eq:double-Lie-albd.axioms.double-bracket.skewsymm.3}
\lb{X,Y}_r'=-\lb{Y,X}_l'',\quad\lb{X,Y}_r''=\lb{Y,X}_l'.
\end{equation}

\subsection{Pairings}
\label{sec:pairings}

Let $A$ be an algebra and $E$ an $A$-bimodule. Following \cite[\S 3.1]{VdB08a} (see also \cite[\S A.3]{Ke11}),
a map
\begin{equation}\label{eq:pairings-def}
\ii{-,-}\colon E\times E \lto A\otimes A
\end{equation}
is called a \emph{pairing} on $E$ if for all $e\in E$, the maps
\begin{subequations}\label{eq:pairing.1}
\begin{align}\label{eq:pairing.1.a}
\ii{e,-}\colon E&\lto A^{\otimes 2}_{\out},
\\\label{eq:pairing.1.b}
\ii{-,e}\colon E&\lto A^{\otimes 2}_{\inn},
\end{align}\end{subequations}
are $A$-bimodule morphisms. 
By~\eqref{eq:dual-def}, this pairing determines an $A$-bimodule morphism
\begin{equation}
\label{eq:anchor-double-lie-algebroid-double-CD.3.a}
(\?)^\flat\colon E\lto E^\vee,\quad e\longmapsto\ii{e,\?},
\end{equation}
The pairing is called \emph{non-degenerate} if the $A$-bimodule $E$ is a finitely generated and projective, and $(\?)^\flat$ is an isomorphism.

Using Sweedler's notation to omit the summation symbol, we will write
\[
\ii{e_1,e_2}=\ii{e_1,e_2}'\otimes \ii{e_1,e_2}'', \quad\text{with }\ii{e_1,e_2}',\ii{e_1,e_2}''\in A.
\]

A pairing is called \emph{symmetric} if
\begin{equation}\label{eq:pairing-symm}
\ii{e_1,e_2}=\ii{e_2,e_1}^\sigma,
\end{equation} 
for all $e_1,e_2\in E$ or, equivalently,
\begin{equation}\label{eq:Sweedler-pairing-symm}
\ii{e_1,e_2}'\otimes  \ii{e_1,e_2}''= \ii{e_2,e_1}''\otimes  \ii{e_2,e_1}'.
\end{equation} 

Finally, based on \eqref{extension-left} and \eqref{notation-Kac}, we can extend the pairing $\ii{-,-}$ to $E\otimes A\oplus A\otimes E$. By \eqref{jumping-notation}, we define the maps
\begin{equation} 
\label{eq:extension-pairing-first-second}
\begin{split}
\ii{e_1,-}_L\colon E\otimes A \lto A^{\otimes 3},\quad & e_2\otimes a \longmapsto \ii{e_1,e_2}\otimes a,
\\
\ii{e_1,-}_R\colon A\otimes E \lto A^{\otimes 3},\quad & a\otimes e_2 \longmapsto a\otimes  \ii{e_1,e_2},
\\
\ii{-,e_2}_L\colon E\otimes A \lto A^{\otimes 3},\quad & e_1\otimes a \longmapsto \ii{e_1,e_2}\otimes_1 a=\ii{e_1,e_2}'\otimes a\otimes \ii{e_1,e_2}'',
\\
\ii{-,e_2}_R\colon A\otimes E \lto A^{\otimes 3},\quad &a\otimes  e_1\longmapsto a\otimes_1 \ii{e_1,e_2}=\ii{e_1,e_2}'\otimes a\otimes \ii{e_1,e_2}''.
\end{split}
\end{equation}
Furthermore, we require that
\begin{equation}\label{eq:ext-pairing-zero}
\ii{e_1,a\otimes e_2}_L=\ii{e_1,e_2\otimes a}_R=\ii{a\otimes e_1,e_2}_L=\ii{e_1\otimes a,e_2}_R=0,
\end{equation}
for all $a\in A$ and $e_1,e_2\in E$.
\section{Differential calculus over double derivations}
\label{sec:diff-calculus}

In this section, we construct a noncommutative version of the usual differential calculus given by the exterior derivative, the Lie derivatives and the contraction operators acting on differential forms on a manifold, where differential forms and vector fields are replaced by noncommutative differential forms and double derivations.
This calculus is used to prove noncommutative analogues of the standard Cartan identities of differential geometry (Theorem~\ref{thm:Cartan-identities.1}), completing the list of identities already obtained by Crawley-Boevey--Etingof--Ginzburg~\cite{CBEG07} and Van den Bergh~\cite{VdB08}. 
The new Cartan identity~\eqref{eq:thm:Cartan-identities.1.f} provides a key ingredient in our construction of exact Courant--Dorfman algebras in~\secref{sec:exact-DC-alg},
and is used to give a full answer to~\cite[Remark 2.7.3]{CBEG07} (see Corollary~\ref{cor:answer-Crawley-Boevey-Etingof-Ginzburg-remark.1}) and to prove a new Cartan identity in the reduced differential calculus (see~\eqref{eq:prop:reduced-Cartan-identities.1.e}). 

\subsection{Contraction operators and Lie derivatives}
\label{sub:Cartan-operators.1}

Let $R$ be an algebra and $A$ an $R$-algebra.
Our first aim is to construct the contractions and Lie derivatives of a noncommutative relative differential form on $A$ with respect to double and triple derivations. 
The constructions for double derivations was of the subject of~\cite[\S\S 2.6, 2.7]{CBEG07}. 
To deal with double and triple derivations simultaneously, we fix an integer $\ell\geq 1$ throughout~\secref{sub:Cartan-operators.1} and consider \emph{$R$-linear $\ell$-fold derivations on $A$}, that is, $R$-linear derivations
\begin{equation}\label{eq:l-fold-derivation.1}
X\colon A\lto A^{\otimes \ell}_\out,
\end{equation}
where $A^{\otimes \ell}_\out$ is the $A$-bimodule with underlying vector space $A^{\otimes \ell}$ and multiplication given for all $a',a''\in A$ and $v=a_1\otimes\cdots\otimes a_\ell\in A^{\otimes\ell}$ (with $a_1,\ldots,a_\ell\in A$) by
\[
a' v a''\defeq(a'a_1)\otimes a_2\otimes\cdots\otimes a_{\ell-1}\otimes(a_\ell a'').
\]
Thus double derivations and triple derivations correspond to $\ell=2$ and $\ell=3$, respectively.
%

We start with the isomorphism~\eqref{sub:ncdiffforms.corep} when $E=A^{\otimes\ell}_\out$, that is,
\begin{equation}\label{eq:DDDer-isom.1}
\Der_R(A,A^{\otimes\ell}_\out)\lra{\cong}\Hom_{A^\e}(\Omega^1_{R}A,A^{\otimes\ell}_\out),\quad X\longmapsto i_X.
\end{equation}
This gives a bijection between $R$-linear $\ell$-fold derivations~\eqref{eq:l-fold-derivation.1} and $A$-bimodule morphisms 
\begin{equation}\label{eq:DDDer-isom.2}
i_X\colon\Omega^1_RA\lto A^{\otimes\ell}_\out,
\end{equation}
uniquely specified in terms of the universal derivation $\du{}\colon A\to\Omega^1_RA$ by the condition 
\begin{equation}\label{eq:DDDer-isom.3}
X=i_X\circ\du{}.
\end{equation}
Since $i_X$ is an $A$-bimodule morphism, 
there exists a unique derivation of degree $-1$ 
\begin{equation}\label{eq:contraction.1}
i_X\colon\Omega^\bullet_RA\lto(\Omega^\bullet_RA)^{\otimes \ell}_\out
\end{equation}
on the tensor algebra $\Omega^\bullet_RA\defeq T_A(\Omega^1_RA)$ (see~\eqref{eq:relative-nc-all-diff-forms.1}) that vanishes on $A$ and restricts to~\eqref{eq:DDDer-isom.2}. 
In fact, the vanishing condition means the derivation~\eqref{eq:contraction.1} is $A$-linear, that is, 
\begin{equation}\label{eq:contraction.2}
i_X\in\Der^{-1}_A\big(\Omega^\bullet_RA,(\Omega^\bullet_RA)^{\otimes \ell}_\out\big),
\end{equation}
for all $X\in\Der_R(A,A^{\otimes\ell}_\out)$.
Therefore we have extended~\eqref{eq:DDDer-isom.1} to an $A$-bimodule morphism 
\begin{equation}\label{eq:contraction.3}
i\colon\Der_R(A,A^{\otimes\ell}_\out)\lto\Der^{-1}_A\big(\Omega^\bullet_RA,(\Omega^\bullet_RA)^{\otimes\ell}_\out\big),\quad X\longmapsto i_X,
\end{equation}
where $i_X$ is the unique derivation of degree $-1$ such that for all $a\in A$,$\!\!$
\begin{subequations}\label{eq:contraction.4}
\begin{align}\label{eq:contraction.4.a}
i_Xa&=0,
\\\label{eq:contraction.4.b}
i_X(\du{a})&=Xa.
\end{align}\end{subequations}
The explicit formula for $i_X$ in~\cite[(2.6.2)]{CBEG07} (for $\ell=2$) extends for arbitrary $\ell\geq 1$, that is, for all $\alpha_1,\ldots,\alpha_n\in\Omega^1_RA$, 
\begin{equation}\label{eq:contraction.7}
i_X(\alpha_1\cdots\alpha_n)=\sum_{k=1}^n(-1)^{k-1}\alpha_1\cdots\alpha_{k-1}(i_X\alpha_k)\alpha_{k+1}\cdots\alpha_n.
\end{equation}

Next, we consider the tensor algebra $T_\kk(\Omega^\bullet_RA)$ of the vector space underlying $\Omega^\bullet_RA$ as a differential graded $R$-algebra, with multiplication and grading respectively given by
\begin{subequations}\label{eq:tensor-algebra-over-base-field.1}
\begin{align}\label{eq:tensor-algebra-over-base-field.1.a}
\alpha\beta&\defeq\alpha_1\otimes\cdots\otimes\alpha_{m-1}\otimes(\alpha_m\beta_1)\otimes\beta_2\otimes\cdots\otimes\beta_n,
\\\label{eq:tensor-algebra-over-base-field.1.b}
\wt{\alpha}&\defeq\wt{\alpha_1}+\cdots+\wt{\alpha_m},
\end{align}
\end{subequations}
for all $\alpha=\alpha_1\otimes\cdots\otimes\alpha_m, \beta=\beta_1\otimes\cdots\otimes\beta_n\in T_\kk(\Omega^\bullet_RA)$ with $\alpha_i,\beta_j\in\Omega^\bullet_RA$, and derivation
\begin{equation}\label{eq:extended-differential.1}
\du{}\colon T_\kk(\Omega^\bullet_RA)\lto T_\kk(\Omega^\bullet_RA)
\end{equation}
of degree $1$ obtained extending the universal differential~\eqref{eq:relative-nc-all-diff-forms.2} using the graded Leibniz rule 
\begin{equation}\label{eq:extended-differential.2}
\du{(\alpha\beta)}=(\du{\alpha})\beta+(-1)^{\wt{\alpha}}\alpha(\du{\beta}),  
\end{equation}
for all $\alpha,\beta\in T_\kk(\Omega^\bullet_RA)$ (cf.~\cite[\S 2.7]{CBEG07}). 
Note that~\eqref{eq:extended-differential.1} is an $R$-linear differential, that is, it is an $R$-linear derivation of degree $1$ (because so is the map~\eqref{eq:relative-nc-all-diff-forms.2}) and satisfies 
\begin{equation}\label{eq:extended-differential.3}
\du{}^2=0
\end{equation}
(by~\eqref{eq:relative-nc-all-diff-forms.3}). 
Furthermore, given $X\in\Der_R(A,A^{\otimes\ell}_\out)$, we extend the $A$-linear derivation~\eqref{eq:contraction.1} of degree $-1$ to another $A$-linear derivation of degree $-1$ (cf.~\cite[\S 2.7]{CBEG07})
\begin{equation}\label{eq:contraction.5}
i_X\colon T_\kk(\Omega^\bullet_RA)\lto T_\kk(\Omega^\bullet_RA),
\end{equation}
so that it satisfies the following Leibniz rule for all $\alpha,\beta\in T_\kk(\Omega^\bullet_RA)$:
\begin{equation}\label{eq:contraction.6}
i_X(\alpha\beta)=(i_X\alpha)\beta+(-1)^{\wt{\alpha}}\alpha(i_X\beta).
\end{equation}
By analogy with geometry, the maps~\eqref{eq:DDDer-isom.2}, \eqref{eq:contraction.1}, \eqref{eq:contraction.5} are called the \emph{contraction} with $X$.

Using the extended maps~\eqref{eq:extended-differential.1} and~\eqref{eq:contraction.5}, for any $X\in\Der_R(A,A^{\otimes\ell}_\out)$, we define now
\begin{equation}\label{eq:Lie-derivative.1}
L_X\colon T_\kk(\Omega^\bullet_RA)\lto T_\kk(\Omega^\bullet_RA)
\end{equation}
by the Cartan homotopic formula 
\begin{equation}\label{eq:Lie-derivative.2}
L_X=\du{}\circ i_X+i_X\circ\du{}.
\end{equation}
One can show by a short direct calculation based on the Leibniz rules~\eqref{eq:extended-differential.2} and~\eqref{eq:contraction.6} that the operator given by~\eqref{eq:Lie-derivative.2} is an $R$-linear derivation of degree $0$, that is, 
\begin{equation}\label{eq:Lie-derivative.3}
L_X\in\Der^0_R T_\kk(\Omega^\bullet_RA).
\end{equation}
Furthermore,~\eqref{eq:extended-differential.3} and~\eqref{eq:Lie-derivative.2} trivially imply (cf.~\cite[p. 283]{CBEG07})
\begin{equation}\label{eq:Lie-derivative.7}
\du{}\circ L_X=L_X\circ\du{}.
\end{equation}
We note now that~\eqref{eq:extended-differential.1} restricts to a map $\du{}\colon(\Omega^\bullet_RA)^{\otimes \ell}\to (\Omega^\bullet_RA)^{\otimes \ell}$, so it follows from the definition~\eqref{eq:Lie-derivative.2} that~\eqref{eq:Lie-derivative.1} restricts to another map 
\begin{equation}\label{eq:Lie-derivative.4}
L_X\colon \Omega^\bullet_RA\lto (\Omega^\bullet_RA)^{\otimes \ell}_\out.
\end{equation}
Then~\eqref{eq:Lie-derivative.3} implies that the map~\eqref{eq:Lie-derivative.4} is an $R$-linear $\ell$-fold derivation of degree $0$, \emph{i.e.},
\begin{equation}\label{eq:Lie-derivative.5-rep}
L_X\in\Der^0_R\big(\Omega^\bullet_RA,(\Omega^\bullet_RA)^{\otimes \ell}_\out\big).
\end{equation}
Being a derivation of degree $0$, it restricts further, in degrees $0$ and $1$, to maps
\begin{subequations}\label{eq:Lie-derivative.5}
\begin{alignat}{3}\label{eq:Lie-derivative.5.a}
&L_X\colon&\; A&\lto A^{\otimes \ell},\quad &a&\longmapsto i_X(\du{a}),
\\\label{eq:Lie-derivative.5.b}
&L_X\colon&\; \Omega^1_RA&\lto \bigoplus_{0\leq k\leq\ell-1}A^{\otimes k}\otimes\Omega^1_RA\otimes A^{\otimes\ell-k-1},
\quad &\alpha&\longmapsto \du{(i_X\alpha)}+i_X(\du{\alpha})
\end{alignat}\end{subequations}
(we have used~\eqref{eq:contraction.4.a} in~\eqref{eq:Lie-derivative.5.a}).
By analogy with geometry, the maps \eqref{eq:Lie-derivative.1}, \eqref{eq:Lie-derivative.4}, \eqref{eq:Lie-derivative.5} are called the \emph{Lie derivative} along $X$. 
Since the algebra $\Omega^\bullet_RA$ is generated over $A$ by $\Omega^1_RA$ and the $A$-bimodule $\Omega^1_RA$ is generated over $A$ by the image of the universal derivation $\du{}\colon A\to\Omega^1_RA$, these maps are uniquely specified by their restriction to generators $a\in A$ and $\du{a}\in\Omega^1_RA$. By~\eqref{eq:contraction.4}, \eqref{eq:extended-differential.3} and~\eqref{eq:Lie-derivative.5.a}, this restriction is given for all $a\in A$ by
\begin{subequations}\label{eq:Lie-derivative.6}
\begin{align}\label{eq:Lie-derivative.6.a}
L_Xa&=i_X(\du{a})=Xa,
\\\label{eq:Lie-derivative.6.b}
L_X(\du{a})&=\du{(Xa)}.
\end{align}\end{subequations}
In the important case of double derivations, \emph{i.e.}, $\ell=2$,~\eqref{eq:Lie-derivative.5} become
\begin{subequations}\label{eq:Lie-derivative.8.}
\begin{align}\label{eq:Lie-derivative.8.a}
L_X\colon A&\lto A^{\otimes 2},
\\\label{eq:Lie-derivative.8.b}
L_X\colon\Omega^1_RA&\lto\Omega^1_RA\otimes A\oplus A\otimes\Omega^1_RA.
\end{align}\end{subequations}
%


We collect now some notation, used in~\secref{sec:exact-DC-alg}, for the Lie derivative and the contraction of low-degree noncommutative differential forms with respect to double derivations (the reader may prefer to skip this part initially and refer back as necessary). 
Since the Lie derivative and the contraction map have degrees $0$ and $-1$, respectively, their actions on 1- and 2-forms admit decompositions
\begin{subequations}\label{eq:Sweedler-on-Cartan-calculus.1}
\begin{align}\label{eq:Sweedler-on-Cartan-calculus.1.a}
L_X\alpha&=L^l_X\alpha+L^r_X\alpha\in\Omega^1_RA\otimes A\oplus A\otimes\Omega^1_RA,
\\\label{eq:Sweedler-on-Cartan-calculus.1.b}
i_X\omega&=i^l_X\omega+i^r_X\omega\in \Omega^1_RA\otimes A\oplus A\otimes\Omega^1_RA,
\end{align}\end{subequations}
for all  $X\in\D_R A$, $\alpha\in\Omega^1_RA$ and $\omega\in\Omega^2_RA$, where
\begin{equation}\label{eq:Sweedler-on-Cartan-calculus.2}
L^l_X\alpha,i^l_X\omega\in\Omega^1_RA\otimes A,
\quad
L^r_X\alpha,i^r_X\omega\in A\otimes\Omega^1_RA.
\end{equation}
Using Sweedler's notation to omit summation symbols, we will write 
\begin{equation}\label{eq:Sweedler-on-Cartan-calculus.3}\begin{gathered}
i_X\alpha=i'_X\alpha\otimes i''_X\alpha\in A\otimes A,
\\
L^l_X\alpha=L^{l'}_X\alpha\otimes L^{l''}_X\alpha\in\Omega^1_RA\otimes A, \qquad 
L^r_X\alpha=L^{r'}_X\alpha\otimes L^{r''}_X\alpha\in A\otimes\Omega^1_RA,
\\
i^l_X\omega=i^{l'}_X\omega\otimes i^{l''}_X\omega\in\Omega^1_RA\otimes A,\qquad 
i^r_X\omega=i^{r'}_X\omega\otimes i^{r''}_X\omega\in A\otimes\Omega^1_RA,
\end{gathered}\end{equation}
where
\begin{equation}\label{eq:Sweedler-on-Cartan-calculus.4}
i'_X\alpha,i''_X\alpha,i^{l''}_X\omega,i^{r'}_X\omega,L^{l''}_X\alpha,L^{r'}_X\alpha\in A,
\qquad
i^{l'}_X\omega,i^{r''}_X\omega,L^{l'}_X\alpha,L^{r'}_X\alpha\in\Omega^1_RA.
\end{equation}

\subsection{The graded double Schouten--Nijenhuis brackets}
\label{sub:graded-Schouten-Nijenhuis.1}

Let $\Omega$ be a graded $R$-algebra that is finitely generated over $R$.
Then the following maps are isomorphisms (cf.~\eqref{eq:graded-eval-map.1})
\begin{subequations}\label{eq:graded-Schouten-Nijenhuis.1.1}
\begin{align}\label{eq:graded-Schouten-Nijenhuis.1.1.a}
\varepsilon_l\colon\DDer^\bullet_R\Omega{}\otimes\Omega&\lra{\cong}\Der^\bullet_R(\Omega,\Omega{}^{\otimes 2}_\out\otimes\Omega),
\\\label{eq:graded-Schouten-Nijenhuis.1.1.b}
\varepsilon_r\colon\Omega\otimes\DDer^\bullet_R\Omega{}&\lra{\cong}\Der^\bullet_R(\Omega,\Omega\otimes\Omega{}^{\otimes 2}_\out),
\end{align}
\end{subequations}
where, applying the Koszul sign rule, for all $\alpha,\beta\in\Omega$, $\xi\in\DDer^\bullet_R\Omega{}$, we define
\begin{subequations}\label{eq:graded-Schouten-Nijenhuis.1.2}
\begin{align}\label{eq:graded-Schouten-Nijenhuis.1.2.a}
\varepsilon_l(\xi\otimes\alpha)(\beta)&\defeq(-1)^{\wt{\alpha}\wt{\beta}}(\xi\beta)\otimes\alpha, 
\\\label{eq:graded-Schouten-Nijenhuis.1.2.b}
\varepsilon_r(\alpha\otimes\xi)(\beta)&\defeq\alpha\otimes (\xi\beta).
\end{align}
\end{subequations}
Furthermore, the action $\sigma_s\colon\Omega{}^{\otimes 3}\to\Omega{}^{\otimes 3}$ given by~\eqref{eq:graded-permtutation-action.2} for $s\in\mathbb{S}_3$ induces isomorphisms 
\begin{subequations}\label{eq:graded-Schouten-Nijenhuis.1.3}
\begin{align}\label{eq:graded-Schouten-Nijenhuis.1.3.a}
\sigma_{(23)*}\colon\Der^\bullet_R(\Omega,\Omega{}^{\otimes 3}_\out)\lra{\cong}&
\Der^\bullet_R(\Omega,\Omega{}^{\otimes 2}_\out\otimes\Omega),\;\;\; \xi\longmapsto\sigma_{(23)}\circ\xi,
\\\label{eq:graded-Schouten-Nijenhuis.1.3.b}
\sigma_{(12)*}\colon\Der^\bullet_R(\Omega,\Omega{}^{\otimes 3}_\out)\lra{\cong}&
\Der^\bullet_R(\Omega,\Omega\otimes\Omega{}^{\otimes 2}_\out),\;\;\;
\xi\longmapsto \sigma_{(12)}\circ\xi
\end{align}
\end{subequations}
(cf.~\eqref{eq:R-linear-double-SN-bracket.4}). 
One can also show, by a direct calculation as in~\eqref{eq:Schouten-Nijenhuis.3}, that the formulae
\begin{subequations}\label{eq:graded-Schouten-Nijenhuis.1.7}
\begin{gather}\label{eq:graded-Schouten-Nijenhuis.1.7.a}
\db{\xi,\eta}_l^{\sim}\defeq(\xi\otimes\Id_{\Omega})\circ \eta-(-1)^{\wt{\xi}\wt{\eta}}(\Id_{\Omega}\otimes \eta)\circ \xi\colon \Omega\lto\Omega{}^{\otimes 3}_\out, 
\\\label{eq:graded-Schouten-Nijenhuis.1.7.b}
\db{\xi,\eta}_r^{\sim}\defeq(\Id_{\Omega}\otimes \xi)\circ \eta-(-1)^{\wt{\xi}\wt{\eta}}(\eta\otimes\Id_{\Omega})\circ \xi\colon \Omega\lto\Omega{}^{\otimes 3}_\out, 
\end{gather}\end{subequations}
define $R$-linear derivations of degree $\wt{\xi}+\wt{\eta}$, for all $\xi,\eta\in\DDer^\bullet_R\Omega{}$, that is, 
\begin{equation}\label{eq:graded-Schouten-Nijenhuis.1.9}
\db{\xi,\eta}_l^{\sim},\db{\xi,\eta}_r^{\sim}\in\Der_R^{\wt{\xi}+\wt{\eta}}(\Omega,\Omega^{\otimes 3}_\out).
\end{equation}
As in~\eqref{eq:R-linear-double-SN-bracket.1}, combining~\eqref{eq:graded-Schouten-Nijenhuis.1.1} and~\eqref{eq:graded-Schouten-Nijenhuis.1.3}, we define the double Schouten--Nijenhuis bracket
\begin{equation}\label{eq:graded-Schouten-Nijenhuis.1.4}
\db{\?,\?}\colon\DDer^\bullet_R\Omega{}\times\DDer^\bullet_R\Omega{}\lto\DDer^\bullet_R\Omega{}\otimes\Omega\oplus\Omega\otimes\DDer^\bullet_R\Omega{}
\end{equation}
of graded $R$-linear derivations on $\Omega$, given by 
\begin{equation}\label{eq:graded-Schouten-Nijenhuis.1.5}
\db{\xi,\eta}\defeq\db{\xi,\eta}_l+\db{\xi,\eta}_r, 
\end{equation}
for all $\xi,\eta\in\DDer^\bullet_R\Omega{}$, where
\begin{subequations}\label{eq:graded-Schouten-Nijenhuis.1.6}
\begin{gather}\label{eq:graded-Schouten-Nijenhuis.1.6.a}
\db{\xi,\eta}_l=\varepsilon_l^{-1}\sigma_{(23)*}(\db{\xi,\eta}_l^{\sim})\in\DDer^\bullet_R\Omega{}\otimes \Omega,
\\\label{eq:graded-Schouten-Nijenhuis.1.6.b}
\db{\xi,\eta}_r=\varepsilon_r^{-1}\sigma_{(12)*}(\db{\xi,\eta}_r^{\sim})\in \Omega\otimes\DDer^\bullet_R\Omega{}.
\end{gather}\end{subequations}

%
%

The following result (used in Corollary~\ref{cor:answer-Crawley-Boevey-Etingof-Ginzburg-remark.1}) provides a link between the (graded) double Schouten--Nijenhuis bracket and the (graded) commutator. 
Following standard notation, given two derivations $\xi,\eta\colon T_\kk\Omega\to T_\kk\Omega$ of degrees $\wt{\xi},\wt{\eta}\in\ZZ$, respectively, of the tensor algebra $T_\kk\Omega$ of the vector space underlying $\Omega$ (cf.~\eqref{eq:tensor-algebra-over-base-field.1}), their \emph{graded commutator} is 
\begin{equation}\label{eq:graded-commutator.1}
[\xi,\eta]\defeq\xi\circ\eta-(-1)^{\wt{\xi}\wt{\eta}}\eta\circ\xi\in\Der_\kk^{\wt{\xi}+\wt{\eta}}T_\kk\Omega.
\end{equation}

\begin{lemma}\label{lem:graded-commutator.1}
For all $\xi,\eta\in\DDer^\bullet_R\Omega$, 
\begin{equation}\label{eq:graded-commutator.1.1}
[\xi,\eta]=\db{\xi,\eta}_l^\sim+\db{\xi,\eta}_r^\sim.
\end{equation}
\end{lemma}

\begin{proof}
Let $\alpha\in\Omega$. Using Sweedler's notation~\eqref{eq:Sweedler-double-derivations.1} and the Leibniz rule, we obtain 
\begin{align*}
[\xi,\eta](\alpha)&=\xi(\eta(\alpha))-(-1)^{\wt{\xi}\wt{\eta}}\eta(\xi(\alpha))
=\xi(\eta'(\alpha))\otimes\eta''(\alpha)+(-1)^{\wt{\xi}\wt{\eta'(\alpha)}}\eta'(\alpha)\otimes\xi(\eta''(\alpha))
\\&\qquad
-(-1)^{\wt{\xi}\wt{\eta}}\big(\eta(\xi'(\alpha))\otimes\xi''(\alpha)+(-1)^{\wt{\xi'(\alpha)}\wt{\eta}}\xi'(\alpha)\otimes\eta(\xi''(\alpha))\big)
\\&
=\big(\xi(\eta'(\alpha))\otimes\eta''(\alpha)-(-1)^{\wt{\xi}\wt{\eta}+\wt{\xi'(\alpha)}\wt{\eta}}\xi'(\alpha)\otimes\eta(\xi''(\alpha))\big)
\\&\qquad
+\big((-1)^{\wt{\xi}\wt{\eta'(\alpha)}}\eta'(\alpha)\otimes\xi(\eta''(\alpha))-(-1)^{\wt{\xi}\wt{\eta}}\eta(\xi'(\alpha))\otimes\xi''(\alpha)\big)
\\&
=\db{\xi,\eta}_l^\sim(\alpha)+\db{\xi,\eta}_r^\sim(\alpha).
\qedhere\end{align*}
\end{proof}

\subsection{The Cartan identities for double derivations}
\label{sub:Cartan-id}

As in~\secref{sec:double-SN}, we suppose $A$ is finitely generated over $R$.
Then the maps~\eqref{eq:R-linear-double-SN-bracket.4} and~\eqref{eq:graded-eval-map.1} determine isomorphisms (cf.~\eqref{eq:R-linear-double-SN-bracket.3})
\begin{subequations}\label{eq:triple-derivations.1}
\begin{align}\label{eq:triple-derivations.1.a}
\varepsilon_l^{-1}\circ\tau_{(23)*}\colon&
\Der_R(A,A^{\otimes 3}_\out)\lra{\cong}\DDer_RA\otimes A,
\\\label{eq:triple-derivations.1.b}
\varepsilon_r^{-1}\circ\tau_{(12)*}\colon&
\Der_R(A,A^{\otimes 3}_\out)\lra{\cong}A\otimes\DDer_RA,
\end{align}\end{subequations}
so the $R$-linear triple derivations (that will often be denoted with a tilde)
\begin{equation}\label{eq:triple-derivations.2}
\widetilde{X}\colon A\lto A^{\otimes 3}_\out  
\end{equation}
are in bijection with either of the following elements: 
\begin{subequations}\label{eq:triple-derivations.3}
\begin{align}\label{eq:triple-derivations.3.a}
X_l&\defeq(\varepsilon_l^{-1}\circ\tau_{(23)*})(\widetilde{X})\in\DDer_RA\otimes A,
\\\label{eq:triple-derivations.3.b}
X_r&\defeq(\varepsilon_r^{-1}\circ\tau_{(12)*})(\widetilde{X})\in A\otimes\DDer_RA.
\end{align}\end{subequations}
Using~\eqref{eq:triple-derivations.3}, the $R$-linear triple derivations~\eqref{eq:contraction.2} and~\eqref{eq:Lie-derivative.5-rep} (for $\ell=3$) will be denoted
\begin{subequations}\label{eq:triple-contraction.4}
\begin{gather}\label{eq:triple-contraction.4.a}
i^{l\sim}_{X_l}\defeq i^{r\sim}_{X_r}\defeq i_{\widetilde{X}}\in\Der^{-1}_R(\Omega{},\Omega{}^{\otimes 3}_\out),
\\\label{eq:triple-contraction.4.b}
L^{l\sim}_{X_l}\defeq L^{r\sim}_{X_r}\defeq L_{\widetilde{X}}\in\Der^0_R(\Omega{},\Omega{}^{\otimes 3}_\out), 
\end{gather}\end{subequations}
where $\Omega\defeq\Omega^\bullet_RA$.
Combining~\eqref{eq:graded-Schouten-Nijenhuis.1.1} and~\eqref{eq:graded-Schouten-Nijenhuis.1.3}, we get graded versions of~\eqref{eq:triple-derivations.1}, that is, 
\begin{subequations}\label{eq:triple-derivations.5}
\begin{align}\label{eq:triple-derivations.5.a}
\varepsilon_l^{-1}\circ\sigma_{(23)*}\colon&
\Der^\bullet_R(\Omega,\Omega{}^{\otimes 3}_\out)\lra{\cong}\DDer^\bullet_R\Omega{}\otimes\Omega,
\\\label{eq:triple-derivations.5.b}
\varepsilon_r^{-1}\circ\sigma_{(12)*}\colon&
\Der^\bullet_R(\Omega,\Omega{}^{\otimes 3}_\out)\lra{\cong}\Omega\otimes\DDer^\bullet_R\Omega{},
\end{align}\end{subequations}
because the graded $R$-algebra $\Omega=\Omega^\bullet_RA$ is finitely generated over $R$ (see~\secref{sub:Cartan-operators.1}), so we can convert the $R$-linear triple derivations~\eqref{eq:triple-contraction.4} into
\begin{subequations}\label{eq:triple-contraction.6}
\begin{align}\label{eq:triple-contraction.6.a}
i^l_{X_l}\defeq&(\varepsilon_l^{-1}\circ\sigma_{(23)*})(i^{l\sim}_{X_l})\in\DDer^\bullet_R\Omega{}\otimes\Omega,
\\\label{eq:triple-contraction.6.b}
i^r_{X_r}\defeq&(\varepsilon_r^{-1}\circ\sigma_{(12)*})(i^{r\sim}_{X_r})\in\Omega\otimes\DDer^\bullet_R\Omega{},
\\\label{eq:triple-contraction.6.c}
L^l_{X_l}\defeq&(\varepsilon_l^{-1}\circ\sigma_{(23)*})(L^{l\sim}_{X_l})\in\DDer^\bullet_R\Omega{}\otimes\Omega,
\\\label{eq:triple-contraction.6.d}
L^r_{X_r}\defeq&(\varepsilon_r^{-1}\circ\sigma_{(12)*})(L^{r\sim}_{X_r})\in\Omega\otimes\DDer^\bullet_R\Omega{}.
\end{align}\end{subequations}
Using~\eqref{eq:triple-contraction.6}, we now define
\begin{subequations}\label{eq:triple-derivations.7}
\begin{align}\label{eq:triple-derivations.7.a}
i_X&\defeq i^l_{X_l}+i^r_{X_r}\in\DDer^\bullet_R\Omega{}\otimes\Omega\oplus\Omega\otimes\DDer^\bullet_R\Omega{},
\\\label{eq:triple-derivations.7.b}
L_X&\defeq L^l_{X_l}+L^r_{X_r}\in\DDer^\bullet_R\Omega{}\otimes\Omega\oplus\Omega\otimes\DDer^\bullet_R\Omega{},
\end{align}\end{subequations}
for all
\begin{equation}\label{eq:triple-derivations.8}
X=X_l+X_r\in\DDer_RA\otimes A\oplus A\otimes\DDer_RA,
\end{equation}
with $X_l\in\DDer_RA\otimes A$, $X_r\in A\otimes\DDer_RA$.
Given $X=X_l+X_r$ as in~\eqref{eq:triple-derivations.8}, we also define
\begin{subequations}\label{eq:triple-derivations.10}
\begin{align}\label{eq:triple-derivations.10.a}
i^\sim_X\defeq& i^{l\sim}_{X_l}+i^{r\sim}_{X_r}\in\Der^{-1}_R(\Omega,\Omega^{\otimes 3}_\out),
\\\label{eq:triple-derivations.10.b}
L^\sim_X\defeq& L^{l\sim}_{X_l}+L^{r\sim}_{X_r}\in\Der^0_R(\Omega,\Omega^{\otimes 3}_\out).
\end{align}\end{subequations}

As in~\eqref{eq:Sweedler-double-bracket}, it is convenient to use Sweedler's notation
\begin{subequations}\label{eq:Sweedler.Cartan-id.1}
\begin{gather}\label{eq:Sweedler.R-linear-double-SN-bracket.1.a}
X_l=X'_l\otimes X''_l,
\\\label{eq:Sweedler.R-linear-double-SN-bracket.1.b}
X_r=X'_r\otimes X''_r,
\end{gather}\end{subequations}
with $X'_l,X''_r\in\DDer_RA$, $X''_l,X'_r\in A$, so applying it to the double Schouten--Nijenhuis bracket~\eqref{eq:R-linear-double-SN-bracket.1}, we recover~\eqref{eq:Sweedler.R-linear-double-SN-bracket.1}.
%
Furthermore, defining $\widetilde{X}\in\Der_R(A,A^{\otimes 3}_\out)$ in terms of $X_l$ or $X_r$ by the isomorphisms~\eqref{eq:triple-derivations.1} as in~\eqref{eq:triple-derivations.3}, it is clear that for all $a\in A$ (cf.~\eqref{eq:graded-bracket.1}) 
\begin{subequations}\label{eq:triple-derivations.9}
\begin{gather}\label{eq:triple-derivations.9.a}
\widetilde{X}(a)=\tau_{(23)}^{-1}\Big(X_l'(a)\otimes X_l''\Big)=X_l'(a)\otimes_1X_l'',
\\\label{eq:triple-derivations.9.b}
\widetilde{X}(a)=\tau_{(12)}^{-1}\Big(X_r'\otimes X_r''(a)\Big)=X_r'\otimes_1X_r''(a).
\end{gather}\end{subequations}

\begin{lemma}\label{lemma:Cartan-identities.1}
For all $\alpha\in\Omega$, $X_l\in\DDer_RA\otimes A$, $X_r\in A\otimes\DDer_RA$, we have
\begin{subequations}\label{eq:lemma:Cartan-identities.1.1}
\begin{align}\label{eq:lemma:Cartan-identities.1.1.a}
i^{l\sim}_{X_l}\alpha&=i_{X_l'}\alpha\otimes_1X_l'',
\\\label{eq:lemma:Cartan-identities.1.1.b}
i^{r\sim}_{X_r}\alpha&=X_r'\otimes_1i_{X_r''}\alpha,
\\\label{eq:lemma:Cartan-identities.1.1.c}
L_{X_l}^{l\sim}\alpha&=L_{X'_l}\alpha\otimes_1X''_l-(-1)^{\wt{\alpha}}i_{X'_l}\alpha\otimes_1\du{X''_l},
\\\label{eq:lemma:Cartan-identities.1.1.d}
L_{X_r}^{r\sim}\alpha&=X'_r\otimes_1L_{X''_r}\alpha+\du{X'_r}\otimes_1i_{X''_r}\alpha.
\end{align}\end{subequations}
Therefore the following identities hold for all $X=X_l+X_r\in\DDer_RA\otimes A\oplus A\otimes\DDer_RA$:
\begin{subequations}\label{eq:lemma:Cartan-identities.1.2}
\begin{align}\label{eq:lemma:Cartan-identities.1.2.a}
i_X&=i_{X_l'}\otimes X_l''+X_r'\otimes i_{X_r''},
\\\label{eq:lemma:Cartan-identities.1.2.b}
L_X&=L_{X_l'}\otimes X_l''+X_r'\otimes L_{X_r''}-i_{X'_l}\otimes \du{X''_l}+\du{X'_r}\otimes i_{X''_r}.
\end{align}\end{subequations}
\end{lemma}

\begin{proof}
To prove the identities~\eqref{eq:lemma:Cartan-identities.1.1.a} and~\eqref{eq:lemma:Cartan-identities.1.1.b}, we define $\widetilde{X}\in\Der_R(A,A^{\otimes 3})$ in terms of $X_l$ or $X_r$ by the isomorphisms~\eqref{eq:triple-derivations.1}, as in~\eqref{eq:triple-derivations.3}. 
Since the left-hand sides and the right-hand sides of these identities are derivations of degree $-1$ in $\alpha\in\DDer_RA$, it is sufficient to prove them on generators $\alpha\in A$ and $\alpha=\du{a}$, with $a\in A$. When $\alpha\in A$,~\eqref{eq:lemma:Cartan-identities.1.1.a} and~\eqref{eq:lemma:Cartan-identities.1.1.b} follow because the left-hand sides and the right-hand sides vanish by~\eqref{eq:contraction.4.a}. When $\alpha=\du{a}$, they follow because applying~\eqref{eq:contraction.4.b}, \eqref{eq:triple-contraction.4.a} and~\eqref{eq:triple-derivations.9}, we get
\begin{align*}
i^{l\sim}_{X_l}(\du{a})&
=i_{\widetilde{X}}(\du{a})
=\widetilde{X}(a)=X_l'(a)\otimes_1X_l''=i_{X_l'}(\du{a})\otimes_1X_l'',
\\
i^{r\sim}_{X_r}(\du{a})&
=i_{\widetilde{X}}(\du{a})
=\widetilde{X}(a)=X_r'\otimes_1X_r''(a)=X_r'\otimes_1i_{X_r''}(\du{a}).
\end{align*}
The identities~\eqref{eq:lemma:Cartan-identities.1.1.c} and~\eqref{eq:lemma:Cartan-identities.1.1.d} follow applying the Leibniz rule~\eqref{eq:extended-differential.2}, the Cartan homotopic formula~\eqref{eq:Lie-derivative.2}, \eqref{eq:lemma:Cartan-identities.1.1.a} and~\eqref{eq:lemma:Cartan-identities.1.1.b}:
\begin{align*}
\begin{split}
L^{l\sim}_{X_l}\alpha&=L_{\widetilde{X}}\alpha=\du{(i_{\widetilde{X}}\alpha)}+i_{\widetilde{X}}(\du{\alpha})
=\du{(i_{X_l'}\alpha\otimes_1X_l'')}+i_{X'_l}(\du{\alpha})\otimes_1X''_l
\\&
=\du{(i_{X'_l}\alpha)}\otimes_1X''_l-(-1)^{\wt{\alpha}}i_{X'_l}\alpha\otimes_1\du{X''_l}+i_{X'_l}(\du{\alpha})\otimes_1X''_l
\\&
=L_{X'_l}\alpha\otimes_1X''_l-(-1)^{\wt{\alpha}}i_{X'_l}\alpha\otimes_1\du{X''_l},
\end{split}
\\
\begin{split}
L^{r\sim}_{X_r}\alpha&=L_{\widetilde{X}}\alpha=\du{(i_{\widetilde{X}}\alpha)}+i_{\widetilde{X}}(\du{\alpha})
=\du{(X_r'\otimes_1i_{X_r''}\alpha)}+X_r'\otimes_1i_{X_r''}(\du{\alpha})
\\&
=\du{X'_r}\otimes_1i_{X''_r}\alpha+X'_r\otimes_1\du{(i_{X''_r}\alpha)}+X'_r\otimes_1i_{X''_r}(\du{\alpha})
\\&
=X'_r\otimes_1L_{X''_r}\alpha+\du{X'_r}\otimes_1i_{X''_r}\alpha.
\end{split}\end{align*}

To prove~\eqref{eq:lemma:Cartan-identities.1.2}, note that these equations are in $\DDer_R\Omega{}\otimes\Omega \oplus \Omega\otimes\DDer_R\Omega{}$, so they are equivalent to the equations obtained projecting them onto the two direct summands, \emph{i.e.},
\begin{subequations}\label{eq:pf:lemma:Cartan-identities.1.1}
\begin{align}\label{eq:pf:lemma:Cartan-identities.1.1.a}
i^l_{X_l}&=i_{X_l'}\otimes X_l'',
\\\label{eq:pf:lemma:Cartan-identities.1.1.b}
i^r_{X_r}&=X_r'\otimes i_{X_r''},
\\\label{eq:pf:lemma:Cartan-identities.1.1.c}
L_{X_l}^l&=L_{X'_l}\otimes X''_l-i_{X'_l}\otimes \du{X''_l},
\\\label{eq:pf:lemma:Cartan-identities.1.1.d}
L_{X_r}^r&=X'_r\otimes L_{X''_r}+\du{X'_r}\otimes i_{X''_r}.
\end{align}
\end{subequations}
The identities~\eqref{eq:pf:lemma:Cartan-identities.1.1.a} and~\eqref{eq:pf:lemma:Cartan-identities.1.1.b} are obviously equivalent to~\eqref{eq:lemma:Cartan-identities.1.1.a} and~\eqref{eq:lemma:Cartan-identities.1.1.b}, respectively. To prove~\eqref{eq:pf:lemma:Cartan-identities.1.1.c} and~\eqref{eq:pf:lemma:Cartan-identities.1.1.d}, we define $T_l\defeq L_{X'_l}\otimes X''_l-i_{X'_l}\otimes \du{X''_l}$ and $T_r\defeq X'_r\otimes L_{X''_r}+\du{X'_r}\otimes i_{X''_r}$.
By~\eqref{eq:graded-Schouten-Nijenhuis.1.2}, we have
\begin{align*}
\varepsilon_l(T_l)(\alpha)&=L_{X_l'}\alpha\otimes X''_l-(-1)^{\wt{\alpha}}i_{X'_l}\alpha\otimes \du{X''_l}, 
\\
\varepsilon_r(T_r)(\alpha)&=X'_r\otimes L_{X''_r}\alpha+\du{X'_r}\otimes i_{X''_r}\alpha, 
\end{align*}
for all $\alpha\in\Omega$, 
so applying~\eqref{eq:graded-Schouten-Nijenhuis.1.3} to these identities and using~\eqref{eq:lemma:Cartan-identities.1.1.c} and \eqref{eq:lemma:Cartan-identities.1.1.d}, we get
\begin{align*}
\sigma_{(23)*}^{-1}\varepsilon_l(T_l)(\alpha)&=L_{X_l'}\alpha\otimes_1 X''_l-(-1)^{\wt{\alpha}}i_{X'_l}\alpha\otimes_1\du{X''_l}=L_{X_l}^{l\sim}\alpha,
\\
\sigma_{(12)*}^{-1}\varepsilon_r(T_r)(\alpha)&=X'_r\otimes_1 L_{X''_r}\alpha+\du{X'_r}\otimes_1i_{X''_r}\alpha=L_{X_r}^{r\sim}\alpha,
\end{align*}
and hence the required identities~\eqref{eq:pf:lemma:Cartan-identities.1.1.c} and \eqref{eq:pf:lemma:Cartan-identities.1.1.d}, by~\eqref{eq:triple-contraction.6.c} and \eqref{eq:triple-contraction.6.d}. 
\end{proof}

The following identities will be needed in the proof of Theorem~\ref{thm:standard-DCA}.

\begin{proposition}\label{prop:Cartan-identities.1}
The following identities hold for all $X,Y\in\DDer_RA$:
\begin{subequations}\label{eq:prop:Cartan-identities.1.1}
\begin{align}\label{eq:prop:Cartan-identities.1.1.a}
\db{i_X,i_Y}_l^\sim&=0,
\\\label{eq:prop:Cartan-identities.1.1.b}
\db{i_X,i_Y}_r^\sim&=0,
\\\label{eq:prop:Cartan-identities.1.1.c}
\db{i_X,L_Y}_l^\sim&=i^{l\sim}_{\lb{X,Y}_l},
\\\label{eq:prop:Cartan-identities.1.1.d}
\db{i_X,L_Y}_r^\sim&=i^{r\sim}_{\lb{X,Y}_r},
\\\label{eq:prop:Cartan-identities.1.1.e}
\db{L_X,i_Y}_l^\sim&=\db{i_X,L_Y}_l^\sim,
\\\label{eq:prop:Cartan-identities.1.1.f}
\db{L_X,i_Y}_r^\sim&=\db{i_X,L_Y}_r^\sim,
\\\label{eq:prop:Cartan-identities.1.1.g}
\db{L_X,L_Y}_l^\sim&=L^{l\sim}_{\db{X,Y}_l},
\\\label{eq:prop:Cartan-identities.1.1.h}
\db{L_X,L_Y}_r^\sim&=L^{r\sim}_{\db{X,Y}_r}.
\end{align}\end{subequations}
\end{proposition}

\begin{proof}
Since the left and the right hand sides of the identities~\eqref{eq:prop:Cartan-identities.1.1} are (triple) derivations of the same degree by~\eqref{eq:graded-Schouten-Nijenhuis.1.9} (when they are not zero), it suffices to show that they hold when these derivations are evaluated on generators $\alpha=a\in A$ and $\alpha=\du{a}\in\Omega^1$. 

We start applying~\eqref{eq:contraction.4} to check the identities~\eqref{eq:prop:Cartan-identities.1.1.a} and~\eqref{eq:prop:Cartan-identities.1.1.b}:
\begin{align*}
\db{i_X,i_Y}_l^\sim&(a)=-\db{i_Y,i_X}_l^\sim(a)=(i_X\otimes\Id_{\Omega})i_Ya+(\Id_\Omega\otimes i_Y)i_Xa=0,
\\
\db{i_X,i_Y}_l^\sim&(\du{a})=-\db{i_Y,i_X}_l^\sim(\du{a})=(i_X\otimes\Id_{\Omega})i_Y\!\du{a}+(\Id_\Omega\otimes i_Y)i_X\!\du{a}
\\&
=(i_X\otimes\Id_{\Omega})(Ya)+(\Id_\Omega\otimes i_Y)(Xa)
=i_X(Y'a)\otimes Y''a+X'a\otimes i_Y(X''a)=0, 
\end{align*}
where we have used Sweedler's notation~\eqref{eq:Sweedler-double-derivations.1}. 
To prove~\eqref{eq:prop:Cartan-identities.1.1.c}, we use again~\eqref{eq:contraction.4}:
\begin{align*}
\db{i_X,L_Y}_l^\sim(a)&=(i_X\otimes\Id_{\Omega})L_Ya-(\Id_\Omega\otimes L_Y)i_Xa
\\&
=(i_X\otimes\Id_{\Omega})(Ya)=i_X(Y'a)\otimes Y''a=0
=i_{\lb{X,Y}^\sim_l}a=i^{l\sim}_{\lb{X,Y}_l}a,
\\
\db{i_X,L_Y}_l^\sim(\du{a})&=(i_X\otimes\Id_{\Omega})L_Y(\du{a})-(\Id_\Omega\otimes L_Y)i_X(\du{a})
\\&
=(i_X\otimes\Id_{\Omega})\du{(Ya)}-(\Id_\Omega\otimes L_Y)(Xa)
\\&
=(i_X\otimes\Id_{\Omega})(\du{(Y'a)}\otimes Y''a+Y'a\otimes\du{(Y''a)})-(\Id_\Omega\otimes L_Y)(Xa)
\\&
=i_X\du{(Y'a)}\otimes Y''a+i_X(Y'a)\otimes\du{(Y''a)}-X'a\otimes L_Y(X''a)
\\&
=X(Y'a)\otimes Y''a+0-X'a\otimes Y(X''a)=\db{X,Y}_l^\sim(a)
\\&
=i_{\db{X,Y}_l^\sim}\du{a}=i^{l\sim}_{\lb{X,Y}_l}(\du{a}). 
\end{align*}
The proof of~\eqref{eq:prop:Cartan-identities.1.1.d} is similar and will be omitted.

The identity~\eqref{eq:prop:Cartan-identities.1.1.e} follows now from~\eqref{eq:lemma:Cartan-identities.1.1.a}, \eqref{eq:lemma:Cartan-identities.1.1.b}, \eqref{eq:prop:Cartan-identities.1.1.c} and~\eqref{eq:prop:Cartan-identities.1.1.d}:
\begin{align*}
\db{L_X,i_Y}^\sim_l&=-\db{i_Y,L_X}^\sim_r=-i^{r\sim}_{\lb{Y,X}_r}=-i_{\lb{Y,X}^\sim_r}=i_{\lb{X,Y}^\sim_l}
=i^{l\sim}_{\lb{X,Y}_l}=\db{i_X,L_Y}_l^\sim.
\end{align*}
The proof of~\eqref{eq:prop:Cartan-identities.1.1.f} is similar and will be omitted.
To prove~\eqref{eq:prop:Cartan-identities.1.1.g}, we use~\eqref{eq:contraction.4} again:
\begin{align*}&
\db{L_X,L_Y}_l^\sim(a)=(L_X\otimes\Id_{\Omega})L_Ya-(\Id_{\Omega}\otimes L_Y)L_Xa
\\&\;
=(L_X\otimes\Id_{\Omega})(Ya)-(\Id_{\Omega}\otimes L_Y)(Xa)
=L_X(Y'a)\otimes Y''a-(X'a)\otimes L_Y(X''a)
\\&\;
=X(Y'a)\otimes Y''a-(X'a)\otimes Y(X''a)
=\lb{X,Y}^\sim_l(a)=L^{l\sim}_{\db{X,Y}_l}(a),
\\&
\db{L_X,L_Y}_l^\sim(\du{a})=(L_X\otimes\Id_{\Omega})L_Y(\du{a})-(\Id_{\Omega}\otimes L_Y)L_X(\du{a})
\\&\;
=(L_X\otimes\Id_{\Omega})\du{(Ya)}-(\Id_{\Omega}\otimes L_Y)\du{(Xa)}
=(L_X\otimes\Id_{\Omega})(\du{(Y'a)}\otimes Y''a+Y'a\otimes\du{(Y''a)})
\\&\quad
-(\Id_{\Omega}\otimes L_Y)(\du{(X'a)}\otimes X''a+X'a\otimes\du{(X''a)})
\\&\;
=L_X\du{(Y'a)}\otimes Y''a+L_X(Y'a)\otimes\du{(Y''a)}
-(\du{(X'a)}\otimes L_Y(X''a)+X'a\otimes L_Y\du{(X''a)})
\\&\;
=\du{(X(Y'a))}\otimes Y''a+X(Y'a)\otimes\du{(Y''a)}
-(\du{(X'a)}\otimes Y(X''a)+X'a\otimes\du{(Y(X''a))})
\\&\;
=\du{\(X(Y'a)\otimes Y''a-X'a\otimes Y(X''a))\)}
=\du{\(\lb{X,Y}^\sim_l(a)\)}
=L_{\db{X,Y}^\sim_l}(\du{a})
=L^{l\sim}_{\db{X,Y}_l}(\du{a}).
\end{align*}
The proof of~\eqref{eq:prop:Cartan-identities.1.1.h} is similar.
Alternatively, it follows from~\eqref{eq:prop:Cartan-identities.1.1.g}:
\[
\db{L_X,L_Y}_r^\sim=-\db{L_Y,L_X}_l^\sim=-L^{l\sim}_{\db{Y,X}_l}=-L_{\db{Y,X}^\sim_l}=L_{\db{X,Y}^\sim_r}=L^{r\sim}_{\db{X,Y}_r}.
\qedhere
\]
\end{proof}

We are now in a position to prove a suitable version of the Cartan identities for double derivations.
They will be the key ingredients of the proof of Theorem~\ref{thm:standard-DCA}. 

\begin{theorem}\label{thm:Cartan-identities.1}
Let $A$ be an $R$-algebra that is finitely generated over $R$. Then for all $X,Y\in\D_R A$, we have
\begin{subequations}\label{eq:thm:Cartan-identities.1}
\begin{align}\label{eq:thm:Cartan-identities.1.a}
\du{}^2&=0,
\\\label{eq:thm:Cartan-identities.1.b}
L_X&=[\du{},i_X],
\\\label{eq:thm:Cartan-identities.1.c}
[\du{},L_X]&=0,
\\\label{eq:thm:Cartan-identities.1.d}
\db{i_X,i_Y}&=0,
\\\label{eq:thm:Cartan-identities.1.e}
\db{i_X,L_Y}&=\db{L_X,i_Y}=i_{\lb{X,Y}},
\\\label{eq:thm:Cartan-identities.1.f}
\db{L_X,L_Y}&=L_{\lb{X,Y}}.
\end{align}\end{subequations}
\end{theorem}

\begin{proof}
The identities~\eqref{eq:thm:Cartan-identities.1.a}, \eqref{eq:thm:Cartan-identities.1.b} and \eqref{eq:thm:Cartan-identities.1.c} are~\eqref{eq:extended-differential.3}, \eqref{eq:Lie-derivative.2} and \eqref{eq:Lie-derivative.7}, respectively. 

The identity~\eqref{eq:thm:Cartan-identities.1.d} follows trivially from~\eqref{eq:graded-Schouten-Nijenhuis.1.5}, \eqref{eq:graded-Schouten-Nijenhuis.1.6}, \eqref{eq:prop:Cartan-identities.1.1.a} and~\eqref{eq:prop:Cartan-identities.1.1.b}:
\[
\lb{i_X,i_Y}=\lb{i_X,i_Y}_l+\db{i_X,i_Y}_r
=\varepsilon_l^{-1}\sigma_{(23)*}(\db{i_X,i_Y}_l^{\sim})+\varepsilon_r^{-1}\sigma_{(12)*}(\db{i_X,i_Y}_r^{\sim})=0.
\]

Next, we prove~\eqref{eq:thm:Cartan-identities.1.e}. 
By~\eqref{eq:graded-Schouten-Nijenhuis.1.6}, \eqref{eq:triple-contraction.6.a},  \eqref{eq:prop:Cartan-identities.1.1.c} and~\eqref{eq:prop:Cartan-identities.1.1.d}, we obtain
\begin{equation}\label{eq:proof:prop:Cartan-identities.1.5}
\db{L_X,i_Y}_l=\lb{i_X,L_Y}_l=i^l_{\lb{X,Y}_l},\quad\db{L_X,i_Y}_r=\lb{i_X,L_Y}_r=i^r_{\lb{X,Y}_r},
\end{equation}
and hence the required identity~\eqref{eq:thm:Cartan-identities.1.e}, by~\eqref{eq:R-linear-double-SN-bracket.2}, \eqref{eq:graded-Schouten-Nijenhuis.1.5} and~\eqref{eq:triple-derivations.7.a}:
\begin{align*}
\lb{i_X,L_Y}&=\lb{i_X,L_Y}_l+\lb{i_X,L_Y}_r=i^l_{\lb{X,Y}_l}+i^r_{\lb{X,Y}_r}=i_{\lb{X,Y}},
\\
\lb{L_X,i_Y}&=\lb{L_X,i_Y}_l+\lb{L_X,i_Y}_r=\lb{i_X,L_Y}_l+\lb{i_X,L_Y}_r=\lb{i_X,L_Y}.
\end{align*}

Finally, we prove~\eqref{eq:thm:Cartan-identities.1.f}. 
The left-hand side is the double Schouten--Nijenhuis bracket
\begin{equation}\label{eq:proof:prop:Cartan-identities.1.1}
\db{L_X,L_Y}=\db{L_X,L_Y}_l+\db{L_X,L_Y}_r,
\end{equation}
(see~\eqref{eq:graded-Schouten-Nijenhuis.1.5}). The two terms of the right-hand side of~\eqref{eq:proof:prop:Cartan-identities.1.1} are given by~\eqref{eq:graded-Schouten-Nijenhuis.1.6}, that is,
\begin{equation}\label{eq:proof:prop:Cartan-identities.1.2}\begin{split}
\db{L_X,L_Y}_l=\varepsilon_l^{-1}\sigma_{(23)*}\big(\db{L_X,L_Y}_l^\sim\big)\in\DDer_R\Omega\otimes\Omega,
\\
\db{L_X,L_Y}_r=\varepsilon_r^{-1}\sigma_{(12)*}\big(\db{L_X,L_Y}_r^\sim\big)\in\Omega\otimes\DDer_R\Omega.
\end{split}\end{equation}
Similarly, by~\eqref{eq:R-linear-double-SN-bracket.2} and~\eqref{eq:triple-derivations.7.b}, the right-hand side of~\eqref{eq:thm:Cartan-identities.1.f} is 
\begin{equation}\label{eq:proof:prop:Cartan-identities.1.3}
L_{\lb{X,Y}}=L^l_{\lb{X,Y}_l}+L^r_{\lb{X,Y}_r}.
\end{equation}
The two terms of the right-hand side of~\eqref{eq:proof:prop:Cartan-identities.1.3} are given by~\eqref{eq:triple-contraction.6.c} and~\eqref{eq:triple-contraction.6.d}, that is, 
\begin{equation}\label{eq:proof:prop:Cartan-identities.1.4}\begin{split}
L^l_{\lb{X,Y}_l}&=\varepsilon_l^{-1}\sigma_{(23)*}\big(L^{l\sim}_{\lb{X,Y}_l}\big)\in\DDer_R\Omega\otimes\Omega,
\\
L^r_{\lb{X,Y}_r}&=\varepsilon_r^{-1}\sigma_{(12)*}\big(L^{r\sim }_{\lb{X,Y}_r}\big)\in\Omega\otimes\DDer_R\Omega.
\end{split}\end{equation}
Now,~\eqref{eq:R-linear-double-SN-bracket.2}, \eqref{eq:prop:Cartan-identities.1.1.g}, \eqref{eq:prop:Cartan-identities.1.1.h},~\eqref{eq:proof:prop:Cartan-identities.1.2} and~\eqref{eq:proof:prop:Cartan-identities.1.4} imply 
\[
\db{L_X,L_Y}_l=L^{l}_{\lb{X,Y}_l},\quad \db{L_X,L_Y}_r=L^{r}_{\lb{X,Y}_r},
\]
and hence the required identity~\eqref{eq:thm:Cartan-identities.1.f}, by~\eqref{eq:proof:prop:Cartan-identities.1.1} and~\eqref{eq:proof:prop:Cartan-identities.1.3}.
\end{proof}

\begin{remark}\label{rem:Cartan-identities.1}
The identities~\eqref{eq:thm:Cartan-identities.1.a}--\eqref{eq:thm:Cartan-identities.1.c} were already in \cite[\S\S 2.6, 2.7]{CBEG07}, whilst~\eqref{eq:thm:Cartan-identities.1.d} and~\eqref{eq:thm:Cartan-identities.1.e} appeared in \cite[Appendix A, (A.2), (A.3)]{VdB08}, formulated as follows:
\begin{subequations}\label{eq:rem:Cartan-identities.1.1}
\begin{align}\label{eq:rem:Cartan-identities.1.1.a}
\lb{i_X,i_Y}_l&=\lb{i_X,i_Y}_r=0,
\\\label{eq:rem:Cartan-identities.1.1.b}
\lb{i_X,L_Y}_l&=i_{\lb{X,Y}'_l}\otimes\lb{X,Y}''_l,
\\\label{eq:rem:Cartan-identities.1.1.c}
\lb{i_X,L_Y}_r&=\lb{X,Y}'_r\otimes i_{\lb{X,Y}''_r}.
\end{align}\end{subequations}
The identities~\eqref{eq:rem:Cartan-identities.1.1.a} can be obtained projecting~\eqref{eq:thm:Cartan-identities.1.d} onto the two direct summands of $\DDer^\bullet_R\Omega{}\otimes\Omega\oplus\Omega\otimes\DDer^\bullet_R\Omega{}$, and~\eqref{eq:rem:Cartan-identities.1.1.b} and~\eqref{eq:rem:Cartan-identities.1.1.c} follow projecting~\eqref{eq:thm:Cartan-identities.1.e} onto the same direct summands, obtaining~\eqref{eq:proof:prop:Cartan-identities.1.5}, and using~\eqref{eq:lemma:Cartan-identities.1.2.a} (or equivalently,~\eqref{eq:pf:lemma:Cartan-identities.1.1.a} and~\eqref{eq:pf:lemma:Cartan-identities.1.1.b}). 

As far as we know, the identity~\eqref{eq:thm:Cartan-identities.1.f} (or equivalently, \eqref{eq:prop:Cartan-identities.1.1.g}--\eqref{eq:prop:Cartan-identities.1.1.h}) is new and was the only missing one needed to have a complete analogy with the usual Cartan differential calculus in geometry. 
Projecting~\eqref{eq:thm:Cartan-identities.1.f} onto the two direct summands of $\DDer^\bullet_R\Omega{}\otimes\Omega\oplus\Omega\otimes\DDer^\bullet_R\Omega{}$ and using~\eqref{eq:lemma:Cartan-identities.1.2.b}, it has a reformulation similar to~\eqref{eq:rem:Cartan-identities.1.1}, namely
\begin{subequations}\label{eq:rem:Cartan-identities.1.2}
\begin{gather}\label{eq:rem:Cartan-identities.1.2.a}
\lb{L_X,L_Y}_l=L_{\lb{X,Y}'_l}\otimes\lb{X,Y}''_l-i_{\lb{X,Y}'_l}\otimes\du{\lb{X,Y}''_l},
\\\label{eq:rem:Cartan-identities.1.2.b}
\lb{L_X,L_Y}_r=\lb{X,Y}'_r\otimes L_{\lb{X,Y}''_r}+\du{\lb{X,Y}'_r}\otimes i_{\lb{X,Y}''_r}.
\end{gather}\end{subequations}
Observe that the second terms on the right-hand sides in this formulation of the Cartan identity~\eqref{eq:thm:Cartan-identities.1.f} have no commutative counterparts.

The Cartan identity~\eqref{eq:thm:Cartan-identities.1.f} provides a complete answer to~\cite[Remark 2.7.3]{CBEG07}, given by the identity~\eqref{eq:cor:answer-Crawley-Boevey-Etingof-Ginzburg-remark.1.1.c} of the following corollary.
Note that a different proof of~\eqref{eq:cor:answer-Crawley-Boevey-Etingof-Ginzburg-remark.1.1.a} was given in~\cite[Lemma~2.6.3]{CBEG07}, but it seems~\eqref{eq:cor:answer-Crawley-Boevey-Etingof-Ginzburg-remark.1.1.b} and~\eqref{eq:cor:answer-Crawley-Boevey-Etingof-Ginzburg-remark.1.1.c} are new.
%
These identities are formulated in terms of the graded commutators~\eqref{eq:graded-commutator.1} and the triple derivations~\eqref{eq:triple-derivations.10}.
\end{remark}

\begin{corollary}\label{cor:answer-Crawley-Boevey-Etingof-Ginzburg-remark.1}
Suppose $A$ is finitely generated over $R$. Then for all $X,Y\in\D_R A$, 
\begin{subequations}\label{eq:cor:answer-Crawley-Boevey-Etingof-Ginzburg-remark.1.1}
\begin{align}\label{eq:cor:answer-Crawley-Boevey-Etingof-Ginzburg-remark.1.1.a}
[i_X,i_Y]&=0,
\\\label{eq:cor:answer-Crawley-Boevey-Etingof-Ginzburg-remark.1.1.b}
[i_X,L_Y]&=[L_X,i_Y]=i^\sim_{\lb{X,Y}},
\\\label{eq:cor:answer-Crawley-Boevey-Etingof-Ginzburg-remark.1.1.c}
[L_X,L_Y]&=L^\sim_{\lb{X,Y}}.
\end{align}\end{subequations}
\end{corollary}

\begin{proof}
These identities follow from~\eqref{eq:graded-commutator.1.1}, \eqref{eq:triple-derivations.10} and \eqref{eq:prop:Cartan-identities.1.1}:
\begin{align*}
[i_X,i_Y]&=\db{i_X,i_Y}_l^\sim+\db{i_X,i_Y}_r^\sim=0,
\\
[i_X,L_Y]&=\db{i_X,L_Y}_l^\sim+\db{i_X,L_Y}_r^\sim
=i^{l\sim}_{\db{X,Y}_l}\alpha+i^{r\sim}_{\db{X,Y}_r}\alpha
=i^\sim_{\lb{X,Y}}\alpha,
\\
[L_X,i_Y]&=\db{L_X,i_Y}_l^\sim+\db{L_X,i_Y}_r^\sim
=i^{l\sim}_{\db{X,Y}_l}\alpha+i^{r\sim}_{\db{X,Y}_r}\alpha
=i^\sim_{\lb{X,Y}}\alpha,
\\
[L_X,L_Y]&=\db{L_X,L_Y}_l^\sim+\db{L_X,L_Y}_r^\sim
=L^{l\sim}_{\db{X,Y}_l}\alpha+L^{r\sim}_{\db{X,Y}_r}\alpha=L^\sim_{\db{X,Y}}\alpha.
\;\;\qedhere\end{align*}
\end{proof}

\subsection{Reduced differential calculus over double derivations}
\label{sub:reduced-diff-calculus}

Here we obtain commutation relations involving the reduced versions introduced in \cite[\S2.8]{CBEG07} of the contractions and the Lie derivatives along double derivations. They will be widely used in \secref{sec:twisted}.

As in the rest of~\secref{sec:diff-calculus}, for a given $R$-algebra $A$, we consider the differential graded algebra $(\Omega^\bullet_R A,\du{})$ and the $A$-bimodule $\D_R A$.
%
%
Recall the Karoubi--de Rham complex $(\dR_R^\bullet A,\du{}_{\dR})$ consists of the differential graded vector space defined in \eqref{eq:def-Karoubi-de-Rham} and the de Rham differential $\du{}_{\dR}$ in~\eqref{eq:def-Karoubi-de-Rham-differential}. 
%
%
Following~\cite[A.1]{VdB08}, given $\alpha=\alpha_1\otimes\alpha_2\in (\Omega^{\bullet}_R A)^{\otimes 2}$, with $\alpha_i\in\Omega^\bullet_RA$, we define
\[
{}^\circ \alpha\defeq (-1)^{\wt{\alpha_1}\wt{\alpha_2}}\alpha_2\alpha_1\in\Omega^{\wt{\alpha_1}+\wt{\alpha_2}}_RA
\]
(see~\cite[\S 2.8]{CBEG07} for an alternative notation), and given a linear map $F\colon \Omega^\bullet_R A\to(\Omega^\bullet_R A)^{\otimes 2}$, 
\[
{}^\circ F\colon \dR^\bullet_R A\longrightarrow\Omega^\bullet_R A,\quad \alpha\longmapsto {}^{\circ}\big(F(\alpha)\big).
\]
Applying this operation to the contraction operator $i_X\in\Der^{-1}_A(\Omega^\bullet_RA,(\Omega^\bullet_RA)^{\otimes 2}_\out)$ in \eqref{eq:contraction.2} and the Lie derivative $L_X\in\Der^0_R(\Omega^\bullet_RA,(\Omega^\bullet_RA)^{\otimes 2}_\out)$ in \eqref{eq:Lie-derivative.1}, with $\ell=2$ and $X\in\DDer_RA$, we define the \emph{reduced contraction} and the \emph{reduced Lie derivative} respectively by 
\begin{subequations}
\begin{align}\label{eq:def-contraction-reduced}
\iota_X\defeq {}^\circ\big(i_X\big) \colon& \dR_R^\bullet A\longrightarrow \Omega^{\bullet-1}_R A.
\\\label{eq:def-Lie-deriv-reducida}
\mathcal{L}_X\defeq {}^\circ\big(L_X\big)\colon& \dR_R^\bullet A\longrightarrow \Omega^\bullet_RA.
\end{align}\end{subequations}
By~\cite[Lemma 2.8.6(ii)]{CBEG07}, if $a,a'\in A$, $X\in\D_R A$, $\omega\in\dR_R^\bullet \Omega$, then $\iota_{a\cdot X\cdot a'}\omega=a(\iota_X\omega)a'$, and, by~\eqref{eq:contraction.4.a}, $\iota_Xa=0$.

The following reduced versions of the Cartan identities for double derivations will be used in the proof of Theorem~\ref{thm:twist-CD-algebra-double}.
As in Theorem~\ref{thm:Cartan-identities.1}, the identities~\eqref{eq:reduced-homotopic-Cartan}--\eqref{eq:identidades-cruciales-reducida.b} already appeared in~\cite{CBEG07,VdB08}, but~\eqref{eq:prop:reduced-Cartan-identities.1.e} is new (cf. Remark~\ref{rem:Cartan-identities.1}).
As in~\eqref{eq:thm:Cartan-identities.1.a} and~\eqref{eq:thm:Cartan-identities.1.b}, the identities~\eqref{eq:reduced-homotopic-Cartan} and~\eqref{eq:prop:reduced-Cartan-identities.1.b} actually hold even if $A$ is not finitely generated over $R$. 

\begin{proposition}\label{prop:reduced-Cartan-identities.1}
Suppose $A$ is finitely generated over $R$. Then for all $X,Y\in\D_R A$, 
\begin{subequations}\label{eq:prop:reduced-Cartan-identities.1}
\begin{align}\label{eq:reduced-homotopic-Cartan}
\mathcal{L}_X&=\du{}\circ\iota_X+\iota_X\circ\du{}_{\dR},
\\\label{eq:prop:reduced-Cartan-identities.1.b}
\du{}\circ\mathcal{L}_X&=\mathcal{L}_X\circ\du{}_{\dR},
\\\label{eq:identidades-cruciales-reducida.a}
i_X\iota_Y+\sigma_{(12)}i_Y\iota_X&=0,
\\\label{eq:identidades-cruciales-reducida.b}
i_X\mathcal{L}_Y-\sigma_{(12)}L_Y\iota_X&=\iota_{\lr{X,Y}'_l}\otimes\lr{X,Y}''_l+\lr{X,Y}'_r\otimes \iota_{\lr{X,Y}''_r},
\\\label{eq:prop:reduced-Cartan-identities.1.e}
\begin{split}
L_X\mathcal{L}_Y-\sigma_{(12)} L_Y\mathcal{L}_X&=\mathcal{L}_{\lr{X,Y}'_l}\otimes\lr{X,Y}''_l+\lr{X,Y}'_r\otimes \mathcal{L}_{\lr{X,Y}''_r}
\\
&\qquad -\iota_{\lr{X,Y}'_l}\otimes\du{\lr{X,Y}''_l}+\du{\lr{X,Y}'_r}\otimes  \iota_{\lr{X,Y}''_r}.
\end{split}\end{align}\end{subequations}
\end{proposition}

\begin{proof}
The reduced Cartan homotopic formula~\eqref{eq:reduced-homotopic-Cartan} and the identity~\eqref{eq:prop:reduced-Cartan-identities.1.b} follow from the corresponding non-reduced identities \eqref{eq:thm:Cartan-identities.1.b} and~\eqref{eq:thm:Cartan-identities.1.b}, respectively (see \cite[Lemma 2.8.8(i)]{CBEG07}).
To prove the remaining identities, we will use Van den Bergh's formula \cite[\S A.1]{VdB08}
\begin{equation}\label{eq:formula-fundamental-VdB-reducida}
\xi\circ {}^\circ \eta-(-1)^{\wt{\xi}\wt{\eta}}\sigma_{(12)}\circ \eta\circ{}^\circ \xi ={}^{\circ, l}\!\lr{\xi,\eta}_l+ {}^{\circ, r}\!\lr{\xi,\eta}_r,
\end{equation}
valid for all $\xi,\eta\in\D^\bullet_R(\Omega^\bullet_RA)$ (cf. Lemma~\ref{lem:graded-commutator.1}), where $\lr{-,-}_l$ and $\lr{-,-}_r$ are the components~\eqref{eq:graded-Schouten-Nijenhuis.1.6} of the (graded) double Schouten--Nijenhuis bracket, 
and
\[
  {}^{\circ, l}(\epsilon'\otimes\epsilon'')\defeq {}^\circ \epsilon'\otimes \epsilon'',\qquad
  {}^{\circ, r}(\tilde{\epsilon}'\otimes\tilde{\epsilon}'')\defeq \tilde{\epsilon}'\otimes {}^\circ \tilde{\epsilon}''.
\]
Now, the identities~\eqref{eq:identidades-cruciales-reducida.a} and~\eqref{eq:identidades-cruciales-reducida.b} follow applying~\eqref{eq:formula-fundamental-VdB-reducida} to \eqref{eq:thm:Cartan-identities.1.d} and \eqref{eq:thm:Cartan-identities.1.e}, respectively (see~\cite[eqs. (A.5) and (A.6)]{VdB08}).
Finally,~\eqref{eq:prop:reduced-Cartan-identities.1.e} follows from~\eqref{eq:formula-fundamental-VdB-reducida} applied to $\xi=L_X$, $\eta=L_Y$, where the right-hand side is calculated using~\eqref{eq:rem:Cartan-identities.1.2} and the  identities
\begin{align*}
{}^{\circ,l}\big(i_{\lr{X,Y}'_l}\otimes\du{\lr{X,Y}''_l}\big)&=\iota_{\lr{X,Y}'_l}\otimes\du{\lr{X,Y}''_l},
\\
{}^{\circ, r}\big(\du{\lr{X,Y}'_r}\otimes i_{\lr{X,Y}''_r}\big)&=\du{\lr{X,Y}'_r\otimes\iota_{\lr{X,Y}''_r}}.
\qedhere
\end{align*}
\end{proof}


\section{Double Poisson vertex algebras}
\label{sec:DPVA}

In \cite{DSKV15}, De Sole, Kac and Valeri developed a theory of integrable systems on noncommutative associative algebras and, remarkably, they introduced double Poisson vertex algebras, which we present in this section. 

A \emph{differential algebra} \cite[\S 3.1]{DSKV15} $(\mathcal{V},\partial)$ is a pair consisting of an associative algebra $\mathcal{V}$, and a derivation $\partial\in\Der\mathcal{V}$. 
Given a polynomial $P(\lambda)=\sum a_1\otimes\dots \otimes a_n\lambda^{k_1}_1\cdots\lambda^{k_{n-1}}_{n-1}\in \mathcal{V}^{\otimes n}[\lambda_1,\dots,\lambda_{n-1}]$ and $f\in\mathcal{V}$, we define
\begin{align*}
(e^{\partial\partial_{\lambda_i}}f)&*_iP(\lambda_1,\dots,\lambda_i,\dots\lambda_{n-1})
\\
&=\sum\sum^{k_i}_{j=0}
\binom{k_i}{j}
a_1\otimes\cdots\otimes a_i\otimes (\partial^j f)a_{i+1}\otimes \cdots\otimes a_n\,\lambda^{k_1}_1\cdots\lambda^{k_i-j}_i\cdots \lambda^{k_{n-1}}_{n-1}.
\end{align*}
That is, $e^{\partial\partial_{\lambda_i}}$ replaces $\lambda_i$ by $\lambda_i+\partial$, and the parentheses point out that $\partial$ is applied to $f$. Moreover, we denote
\begin{align*}
P(\lambda_1,\dots,\lambda_i&+\partial,\dots,\lambda_{n-1})_{\to}*_{n-i}f
\\
&=\sum\sum^{k_i}_{j=0}\binom{k_i}{j}a_1\otimes \cdots\otimes a_i(\partial^jf)\otimes \cdots\otimes a_n\,\lambda^{k_1}_1\cdots\lambda^{k_i-j}_i\cdots\lambda^{k_{n-1}}_{n-1}.
\end{align*}
In other words, the arrow means that $\partial$ is applied to $f$.

\begin{definition}[{\cite[p. 1058]{DSKV15}}]
A \emph{double $\lambda$-bracket} on a differential algebra $(\mathcal{V},\partial)$ is a linear map
\begin{equation}
\lr{-{}_\lambda-}\colon\mathcal{V}\otimes\mathcal{V}\lto (\mathcal{V}\otimes \mathcal{V})[\lambda]
\label{eq:doublelambda-bracket-def}
\end{equation} 
satisfying, for all $a,b,c\in\mathcal{V}$, the sesquilinearity conditions:
\begin{subequations}\label{eq:sesquilinearity-vertex}
\begin{align}
\label{eq:sesquilinearity-vertex.a}
\lr{\partial a{}_\lambda b}&=-\lambda\lr{a{}_\lambda b},
\\
\label{eq:sesquilinearity-vertex.b}
\lr{a{}_\lambda \partial b}&=(\lambda+\partial)\lr{a{}_\lambda b},
\end{align}
\end{subequations}
and the Leibniz rules:
\begin{subequations}\label{eq:Leibniz-vertex}
\begin{align}
\label{eq:Leibniz-vertex.a}
\lr{a{}_\lambda bc}&=b\lr{a{}_\lambda c}+\lr{a{}_\lambda b}c,
\\
\label{eq:Leibniz-vertex.b}
\lr{ab{}_\lambda c}&=\lr{a{}_{\lambda+\partial}c}_{\to}*_1 b+(e^{\partial \partial_\lambda}a)*_1\lr{b{}_\lambda c}.
\end{align}
\end{subequations}
\label{def:double-Poisson-vertex-alg}
\end{definition}

Recall that in \eqref{notation-Kac} we introduced suitable extensions
of the double bracket $\lr{-,-}$ to define the double Jacobi identity \eqref{Jacobi-Kac}. 
Similarly, to define a suitable Jacobi identity for double $\lambda$-brackets, we need to extend \eqref{eq:doublelambda-bracket-def} as follows (see \cite[eq. (4.3)]{DSKV15}). 
For all $a,b,c\in\mathcal{V}$,
\begin{equation}
\begin{aligned}
\lr{a{}_\lambda(b\otimes c)}_L&\defeq\lr{a{}_\lambda b}\otimes c,
\\
\lr{a{}_\lambda(b\otimes c)}_R&\defeq b\otimes\lr{a{}_\lambda c},
\\
\lr{(a\otimes b){}_\lambda c}_L&\defeq\lr{a{}_{\lambda+\partial}c}_\to\otimes_1 b,
\end{aligned}
\label{eq:extension-vertex-double-bracket}
\end{equation}
where the notation $\otimes_1$ was introduced in \eqref{jumping-notation}.

\begin{definition}[{\cite[Definition 3.2]{DSKV15}}]\label{DPVA}
A \emph{double Poisson vertex algebra} $(\mathcal{V},\partial,\lr{-{}_\lambda-})$ is a differential algebra $(\mathcal{V},\partial)$ endowed with a double $\lambda$-bracket  $\lr{-{}_\lambda-}\colon \mathcal{V}\otimes\mathcal{V}\to(\mathcal{V}\otimes\mathcal{V})[\lambda]$, satisfying the following axioms:
 \begin{enumerate}
 \item [\textup{(i)}]
Skewsymmetry:
\begin{equation}
\lr{a{}_\lambda b}=-\lr{b{}_{-\lambda-\partial}a}^\sigma,
\label{vertex-skew}
\end{equation}
where $-\lambda-\partial$ in the right-hand side is moved to the left, acting on the coefficients.
\item [\textup{(ii)}]
Jacobi identity:
\begin{equation}
\lr{a{}_\lambda\lr{b{}_\mu c}}_L=\lr{b{}_\mu\lr{a{}_\lambda c}}_R+\lr{\lr{a_{\lambda} b}{}_{\lambda+\mu}c}_L.
\label{Jacobi-vertex}
\end{equation}
\end{enumerate}
\end{definition}

It is easy to check that if $\partial=0$, and $\lambda=\mu=0$, the axioms of double Poisson vertex algebras become the axioms of double Poisson algebras. So, as emphasized in \cite{DSKV15}, a double Poisson vertex algebra is a cross between a double Poisson algebra and a Poisson vertex algebra. Furthermore, it is straightforward to check that given a double Poisson vertex algebra, the skewsymmetry  \eqref{vertex-skew} and the left Leibniz rule \eqref{eq:Leibniz-vertex.a} imply the right Leibniz rule \eqref{eq:Leibniz-vertex.b} (we thank Daniele Valeri for pointing this out). Therefore, \eqref{eq:Leibniz-vertex.b} is redundant in Definition \ref{DPVA} and so we will omit it below.

All double Poisson vertex algebras appearing in this paper will have the following additional structure:

\begin{definition}\label{graded-DPVA}
A \emph{graded double Poisson vertex algebra of degree $-1$} is a double Poisson vertex algebra $(\mathcal{V},\partial,\lr{-{}_\lambda-})$ equipped with an $\NN$-grading $\mathcal{V}=\bigoplus_{n\geq 0}\mathcal{V}_n$, called \emph{weight}, such that $\mathcal{V}$ is graded as an algebra, 
the derivation $\partial$ has weight $1$ and the double $\lambda$-bracket $\lr{-{}_\lambda-}\colon\mathcal{V}\otimes\mathcal{V}\to (\mathcal{V}\otimes \mathcal{V})[\lambda]$ has weight $-1$, where we assign weight 1 to the formal variable $\lambda$.
\end{definition}

Following \cite[Proposition 3.20]{DSKV15}, given $a,b\in\mathcal{V}$, we can write the double $\lambda$-bracket as 
\begin{equation}\label{eq:notation-vector-bracket-sumatorio}
\lr{a{}_{\lambda}b}=\sum_{\mathit{n}\in\mathbb{N}}\big((a_{\mathit{n}}b)^{\prime}\otimes (a_{\mathit{n}}b)^{\pprime}\big)\lambda^{\mathit{n}}.
\end{equation}
Then the condition in Definition~\ref{graded-DPVA} that $\lr{-{}_\lambda-}$ has weight $-1$ can be formulated, using Sweedler's notation~\eqref{eq:notation-vector-bracket-sumatorio}, as the condition that for all $i,j,\mathit{n}\in\NN$, $a\in\mathcal{V}_i$, $b\in\mathcal{V}_j$,
\[
(a_{\mathit{n}}b)^{\prime}\otimes (a_{\mathit{n}}b)^{\pprime}\in\(\mathcal{V}\otimes\mathcal{V}\)_{i+j-\mathit{n}-1}.
\]


\section{Double Courant--Dorfman algebras}
\label{sec:CD-algebras}

\allowdisplaybreaks
Roytenberg \cite{Roy09} introduced an algebraic analogue of Courant algebroids called \emph{Courant--Dorfman algebras}; their relationship is analogous to that of Lie--Rinehart algebras to Lie algebroids.
In this section, based on \cite{ACF17} (see also \cite{Fer16}), we define a noncommutative version of Courant--Dorfman algebras called \emph{double Courant--Dorfman algebras} by adapting Roytenberg's definition to the noncommutative setting reviewed in \S\ref{sub:basics}.
We start by recalling Roytenberg's definition of Courant--Dorfman algebras~\cite[Definitions 2.1 and 2.3]{Roy09}.

\begin{definition}\label{def:CD-algebra-def-comm}
A \emph{Courant--Dorfman algebra} consists of the following data:
\begin{enumerate}
\item[(1)]
a commutative $\kk$-algebra $\cO$,
\item[(2)]
an $\cO$-module $\cE$,
\item[(3)]
a symmetric bilinear form $\langle-,-\rangle\colon\cE\otimes_\cO\cE\to\cO$,
\item[(4)]
a derivation $\partial\colon\cO\to\cE$,
\item [(5)]
a linear map $[-,-]\colon\cE\otimes\cE\to\cE$, called the \emph{Dorfman bracket}.
\end{enumerate}
These data are required to satisfy the following conditions, for all $\comma,\commb\in \cO$, $s,\comme,\commf,\commg\in \cE$:
\begin{subequations}
\label{eq:CD-comm}
\begin{align}\label{eq:CD-comm.a}
[\comme,\comma\commf]&=\comma[\comme,\commf]+\langle\comme,\partial\comma\rangle \commf, 
\\\label{eq:CD-comm.b}
\langle \comme,\partial\langle\commf,\commg\rangle\rangle&= \langle [\comme,\commf],\commg\rangle+\langle\commf,[\comme,\commg]\rangle,
\\\label{eq:CD-comm.c}
[\comme,\commf]+[\commf,\comme]&=\partial\langle \comme,\commf\rangle,
\\\label{eq:CD-comm.d}
[\comme,[\commf,\commg]]&=[[\comme,\commf],\commg]+[\commf,[\comme,\commg]],
\\\label{eq:CD-comm.e}
[\partial\comma,s]&=0,
\\\label{eq:CD-comm.f}
\langle\partial\comma,\partial\commb\rangle&=0.
\end{align}
\end{subequations}
A \emph{Courant algebroid} is a non-degenerate Courant--Dorfman algebra, that is, a Courant--Dorfman algebra such that the following map is an isomorphism (where $\cE^*\defeq\Hom_{\cO}(\cE,\cO)$):
\[
\cE\lto\cE^*,\quad\comme\longmapsto\langle\comme,\?\rangle.
\]
\end{definition}

It is easy to check that~\eqref{eq:CD-comm.a} and~\eqref{eq:CD-comm.c} imply
\begin{equation}\label{eq:CD-comm.g}
[\comma\comme,\commf]=\comma[\comme,\commf]+\langle\comme,\commf\rangle\partial \comma-\langle\partial\comma,\commf\rangle\comme.
\end{equation}

As Roytenberg pointed out, a vector space $\cE$ equipped with a bracket $[-,-]$ satisfying \eqref{eq:CD-comm.d} is called a ($\kk$-)Leibniz algebra. Moreover, if the pairing $\langle-,-\rangle$ is nondegenerate, the action of $\cE$ on $\cO$ completely determines the derivation $\partial$, so \eqref{eq:CD-comm.e} and \eqref{eq:CD-comm.f} are redundant, 
and one recovers the definition of Courant algebroids of \cite{Roy00} (see \cite{Roy09} and references therein).

Our noncommutative version of Courant--Dorfman algebra in Definition~\ref{CD-Definition}
is obtained replacing the data in (1)--(5) of Definition~\ref{eq:CD-comm} by an associative algebra $A$, an $A$-bimodule $E$ and appropriate linear maps
\[
\ii{-,-}\colon E\otimes E\lto A\otimes A, \quad\partial\colon E\lto A,\quad \cc{-,-}\colon E\otimes E\lto E\otimes A\oplus A\otimes E.
\]
Considering the tensor algebra $T_A E$ of $E$ over $A$, with $E$ placed in degree 1, 
we impose the Leibniz rule~\eqref{eq:double-bracket.b} to extend the operator $\cc{-,-}$ to a bilinear map of degree $-1$, 
\begin{equation}\label{eq:cc-tensoralg.1}
\cc{-,-}\colon T_AE\otimes T_AE\lto T_AE\otimes T_AE, 
\end{equation}
adding the following conditions for all $e\in E$ and $a,b\in A$ (see~\cite[\S 7.1]{ACF17}):
\begin{equation}\label{eq:cc-tensoralg.2}
\cc{e,a}=\ii{e,\partial a},\quad \cc{a,e}=-\ii{\partial a,e},\quad \cc{a,b}=0.
\end{equation}
Using the identities~\eqref{eq:cc-tensoralg.2}, it seems natural to extend $\cc{e,-},\cc{-,f}\colon E\to E\otimes A\oplus A\otimes E$ (for $a\in A$, $e,f\in E$) to linear maps (cf.~\eqref{extension-left}, \eqref{notation-Kac}, \eqref{eq:extension-vertex-double-bracket})
\begin{equation}
\label{eq:extension-CD-bracket}
\hspace*{-1.3ex}
\begin{aligned}&
\cc{e,-}_L\colon E\otimes A\lto E\otimes A\otimes A\oplus A\otimes E\otimes A,&&  f\otimes a\longmapsto \cc{e,f}\otimes a,
\\&
\cc{e,-}_L\colon A\otimes E\lto A\otimes A\otimes E,&& a\otimes f\longmapsto \ii{e,\partial a}\otimes f,
\\&
\cc{e,-}_R\colon E\otimes A\lto E\otimes A\otimes A,&& f\otimes a\longmapsto f\otimes\ii{e,\partial a},
\\&
\cc{e,-}_R\colon A\otimes E\lto A\otimes E\otimes A\oplus A\otimes A\otimes E,&& a\otimes f\longmapsto a\otimes \cc{e,f},
\\&
\cc{-,f}_L\colon E\otimes A\lto E\otimes A\otimes A\, + c.p.,&&  e\otimes a\longmapsto \cc{e,f}\otimes_1a+\ii{e,f}\otimes_1\partial a,
\\&
\cc{-,f}_L\colon A\otimes E\lto A\otimes E\otimes A,&& a\otimes e\longmapsto -\ii{\partial a,f}\otimes_1 e,
\\&
\cc{-,f}_R\colon E\otimes A\lto A\otimes E\otimes A,&&  e\otimes a\longmapsto -e\otimes_1\ii{\partial a,f},
\\&
\cc{-,f}_R\colon A\otimes E\lto E\otimes A\otimes A\, + c.p.,&& a\otimes e\longmapsto a\otimes_1\cc{e,f}+\partial a\otimes_1\ii{e,f},
\end{aligned}
\end{equation}
where $E\otimes A\otimes A\,+ c.p.$ is a shorthand for $E\otimes A\otimes A\oplus A\otimes E\otimes A\oplus A\otimes A\otimes E$, and the notation $\otimes_1$ was introduced in \eqref{jumping-notation}.
Note that some of the extensions~\eqref{eq:extension-CD-bracket} depend on $\ii{-,-}$ and $\partial$.

\begin{definition}\label{CD-Definition}
A \emph{double Courant--Dorfman algebra}
 is a 5-tuple $(A,E,\ii{-,-},\partial,\cc{-,-})$ consisting of the following data:
\begin{enumerate}
\item[(1)]
an algebra $A$,
\item[(2)]
an $A$-bimodule $E$,
\item[(3)]
a symmetric pairing $\ii{-,-}\,\colon E\otimes E\lto A\otimes A$,
\item[(4)]
a derivation $\partial\,\colon A\lto E$,
\item[(5)]
  a linear map $\cc{-,-}\colon E\otimes E\lto E\otimes A\oplus A\otimes E$, called the \emph{double Dorfman bracket}.
\end{enumerate}
These data must satisfy the following axioms, for all $a,b\in A$ and $e,f,g\in E$:
\begin{subequations}
\label{eq:CD}
\begin{align}
\cc{e,f}+\cc{f,e}^\sigma&=\partial\ii{e,f}, \label{eq:CD.a}
\\
\cc{\partial a, e}&=0,  \label{eq:CD.b}
\\
\ii{\partial a,\partial b}&=0,   \label{eq:CD.c}
\\
\cc{e,fa}&=\cc{e,f}a+f\ii{e,\partial a},  \label{eq:CD.d}
\\
\cc{e,af}&=a\cc{e,f}+\ii{e,\partial a}f,  \label{eq:CD.e}
\\
\cc{e,\cc{f,g}}_L&=\cc{\cc{e,f},g}_L+\cc{f,\cc{e,g}}_R,   \label{eq:CD.f}
\\
\ii{e,\partial\ii{f,g}}_L&=\ii{\cc{e,f},g}_L+\ii{f,\cc{e,g}}_R.    \label{eq:CD.g}
\end{align}
\end{subequations}
When $A$ is an $R$-algebra, a Courant--Dorfman algebra is called \emph{$R$-linear} if so are the derivation $\partial$ and the double Dorfman bracket $\cc{-,-}$.
Finally, a double Courant--Dorfman algebra is called \emph{non-degenerate} if so is the pairing. Such an algebra will also be called a \emph{double Courant algebroid}.
\end{definition}

In \eqref{eq:CD.a} and \eqref{eq:CD.g}, the derivation $\partial$ is extended to $A\otimes A$ by the Leibniz rule, so
\begin{equation}
\partial\ii{e_1,e_2}=\partial\ii{e_1,e_2}^\prime\otimes\ii{e_1,e_2}^{\prime\prime}+\ii{e_1,e_2}^\prime\otimes\partial \ii{e_1,e_2}^{\prime\prime},
\label{eq:partialLeibniz-rule}
\end{equation}
for all $e_1,e_2\in E$.
In \eqref{eq:CD.g} we use the extensions of $\ii{-,-}$ given by \eqref{eq:extension-pairing-first-second}, so this identity is in $A^{\otimes 3}$. Moreover, in \eqref{eq:CD.d} and \eqref{eq:CD.e} the products use the outer bimodule structure.

In future calculations, we will use the following decomposition of the double Dorfman bracket $\cc{-,-}$. 
Given $e,f\in E$, 
\begin{equation}
\label{eq:decomposition-Courant-Dorfman-def}
\cc{e,f}=\cc{e,f}_l+\cc{e,f}_r\in E\otimes A\oplus A\otimes E,
\end{equation}
that is, $\cc{e,f}_l\in E\otimes A$ and $\cc{e,f}_r\in A\otimes E$. 
Using Sweedler's notation, we can write
\begin{equation}
\label{eq:decomposition-Courant-Dorfman-def-1}
\cc{e,f}_l=\cc{e,f}^{\prime}_l\otimes \cc{e,f}^{\pprime}_l,\quad \cc{e,f}_r= \cc{e,f}^{\prime}_r\otimes \cc{e,f}^{\pprime}_r,
\end{equation}
with
$\cc{e,f}^{\prime}_l,\cc{e,f}^{\pprime}_r\in E$ and $\cc{e,f}^{\prime}_r,\cc{e,f}^{\pprime}_l\in A$.

The following formulae will be used frequently throughout this paper.

\begin{lemma}\label{lem:CD-identities}
Let $(A,E,\ii{-,-},\partial,\cc{-,-})$ be a double Courant--Dorfman algebra. Then the following identities hold for all $a\in A$ and $e,f,g\in E$:
\begin{subequations}\label{eq:CD-identities.1}
\begin{align}\label{eq:CD-identities.1.a}
\partial\ii{e,\partial a}&=\cc{e,\partial a},
\\\label{eq:CD-identities.1.b}
\partial\ii{\partial a,e}&=\cc{e,\partial a}^\sigma,
\\\label{eq:CD-identities.1.c}
\cc{ae,f}&=a*\cc{e,f}+\partial a*\ii{e,f}-\ii{\partial a,f}*e,
\\\label{eq:CD-identities.1.d}
\cc{ea,f}&=\cc{e,f}*a+\ii{e,f}*\partial a-e*\ii{\partial a,f},
\\\label{eq:CD-identities.1.e}
\cc{e,\cc{f,g}}_R&=\cc{\cc{e,f},g}_R+\cc{f,\cc{e,g}}_L,
\\\label{eq:CD-identities.1.f}
\ii{f,\partial\ii{e,g}}_R&=\ii{\cc{f,e},g}_R+\ii{e,\cc{f,g}}_L,
\\\label{eq:CD-identities.1.g}
\ii{f,\partial\ii{e,g}}_R&=\ii{\partial\ii{e,f},g}_L-\ii{\cc{e,f},g}_L+\ii{e,\cc{f,g}}_L.
\end{align}\end{subequations}
\end{lemma}

\begin{proof}
The identities~\eqref{eq:CD-identities.1.a} and~\eqref{eq:CD-identities.1.b} follow from~\eqref{eq:CD.a} and~\eqref{eq:CD.b}:
\[
\partial\ii{e,\partial a}=\cc{e,\partial a}+\cc{\partial a,e}^\sigma=\cc{e,\partial a},
\quad
\partial\ii{\partial a,e}=\cc{\partial a,e}+\cc{e,\partial a}^\sigma=\cc{e,\partial a}^\sigma.
\]

To prove~\eqref{eq:CD-identities.1.c}, we apply~\eqref{eq:CD.a}, \eqref{eq:CD.e}, the identity $\ii{ae,f}=a*\ii{e,f}$ 
and~\eqref{eq:pairing-symm}:
\begin{align*}
\cc{ae,f}&=\partial\ii{ae,f}-\cc{f,ae}^\sigma
=\partial(a*\ii{e,f})-\(a\cc{f,e}+\ii{f,\partial a}e\)^\sigma
\\&
=a*\partial\ii{e,f}+\partial a*\ii{e,f}-\(a*\cc{f,e}^\sigma+\ii{f,\partial a}^\sigma*e\)
\\&
=a*\cc{e,f}+\partial a*\ii{e,f}-\ii{\partial a,f}*e.
\end{align*}
The identity~\eqref{eq:CD-identities.1.d} follows by a similar calculation, using~\eqref{eq:CD.d} rather than~\eqref{eq:CD.e}, and $\ii{ea,f}=\ii{e,f}*a$  rather than $\ii{ae,f}=a*\ii{e,f}$.

Next, we prove~\eqref{eq:CD-identities.1.e}. By \eqref{eq:extension-CD-bracket}, \eqref{eq:CD.a}, \eqref{eq:CD.f}, \eqref{eq:decomposition-Courant-Dorfman-def} and \eqref{eq:decomposition-Courant-Dorfman-def-1}, the left-hand side is
\begin{align}\notag
\cc{e,\cc{f,g}}_R&=\cc{f,g}^{\prime}_l\otimes\ii{e,\partial\cc{f,g}^{\pprime}_l}+\cc{f,g}^{\prime}_r\otimes\cc{e,\cc{f,g}^{\pprime}_r}
\\\notag&
=\sigma_{(123)}\big(\ii{e,\partial\cc{f,g}^{\pprime}_l}\otimes\cc{f,g}^{\prime}_l+\cc{e,\cc{f,g}^{\pprime}_r}\otimes\cc{f,g}^{\prime}_r\big)
\\\notag&
=\sigma_{(123)}\big(\cc{e,\cc{f,g}^\sigma}_L\big)
=\sigma_{(123)}\big(\cc{e,\partial\ii{g,f}}_L-\cc{e,\cc{g,f}}_L\big)
\\\label{eq:CD-identities.3}&
=\sigma_{(123)}\big(\cc{e,\partial\ii{g,f}}_L-\cc{\cc{e,g},f}_L-\cc{g,\cc{e,f}}_R\big)
\end{align}
(the third identity follows by applying $\sigma=\sigma_{(12)}$ to~\eqref{eq:decomposition-Courant-Dorfman-def-1} and then expanding the double bracket using~\eqref{eq:extension-CD-bracket}). 
By \eqref{eq:extension-pairing-first-second}, \eqref{eq:ext-pairing-zero}, \eqref{eq:extension-CD-bracket}, \eqref{eq:CD.g}, \eqref{eq:partialLeibniz-rule}, \eqref{eq:CD-identities.1.a}, in the first term of \eqref{eq:CD-identities.3}, we have
\begin{align}\notag
\cc{e,\partial\ii{g,f}}_L&=\cc{e,\partial\ii{g,f}^{\prime}\otimes\ii{g,f}^{\pprime}
+\ii{g,f}^{\prime}\otimes\partial\ii{g,f}^{\pprime}}_L
\\\notag&
=\cc{e,\partial\ii{g,f}^{\prime}}\otimes\ii{g,f}^{\pprime}+\ii{e,\partial\ii{g,f}^{\prime}}\otimes\partial\ii{g,f}^{\pprime}
\\\notag&
=\partial\ii{e,\partial\ii{g,f}^{\prime}}\otimes\ii{g,f}^{\pprime}+\ii{e,\partial\ii{g,f}^{\prime}}\otimes\partial\ii{g,f}^{\pprime}
\\\notag&
=\partial\big(\ii{e,\partial\ii{g,f}^{\prime}}\otimes\ii{g,f}^{\pprime}\big)
=\partial\big(\ii{e,\partial\ii{g,f}}_L\big)
\\\label{eq:CD-identities.4}&
=\partial\(\ii{\cc{e,g},f}_L\)+\partial\(\ii{g,\cc{e,f}}_R\),
\end{align}
where $\partial$ is extended to an operator on $A^{\otimes 3}$ by the Leibniz rule (cf.~\eqref{eq:partialLeibniz-rule}).
By \eqref{eq:pairing-symm}, \eqref{eq:extension-pairing-first-second}, \eqref{eq:ext-pairing-zero}, \eqref{eq:extension-CD-bracket} and \eqref{eq:CD.a}, the two terms in the right-hand side of \eqref{eq:CD-identities.4} are 
\begin{subequations}\label{eq:CD-identities.5}\begin{align}\notag
\partial\!\(\ii{\cc{e,g},f}_L\)&=\partial\!\(\!\ii{\cc{e,g}^{\prime}_l,f}\!\otimes_1\!\cc{e,g}^{\pprime}_l\)
=\partial\ii{\cc{e,g}^{\prime}_l,f}\!\otimes_1\!\cc{e,g}^{\pprime}_l+\ii{\cc{e,g}^{\prime}_l,f}\otimes_1\partial\cc{e,g}^{\pprime}_l
\\\notag&\hspace*{-3.1ex}
=\(\cc{\cc{e,g}^{\prime}_l,f}+\cc{f,\cc{e,g}^{\prime}_l}^\sigma\)\otimes_1\cc{e,g}^{\pprime}_l+\ii{\cc{e,g}^{\prime}_l,f}\otimes_1\partial\cc{e,g}^{\pprime}_l
\\\notag&\hspace*{-3.1ex}
=\(\cc{\cc{e,g}^{\prime}_l,f}\otimes_1\cc{e,g}^{\pprime}_l+\ii{\cc{e,g}^{\prime}_l,f}\otimes_1\partial\cc{e,g}^{\pprime}_l\)+\sigma_{(132)}\(\cc{f,\cc{e,g}^{\prime}_l}\otimes\cc{e,g}^{\pprime}_l\)
\\\label{eq:CD-identities.5.a}
&\hspace*{-3.1ex}
=\cc{\cc{e,g}_l,f}_L+\sigma_{(132)}\(\cc{f,\cc{e,g}_l}_L\),
\\\notag
\partial\!\(\ii{g,\cc{e,f}}_R\)&=\partial\!\(\cc{e,f}^{\prime}_r\!\otimes\!\ii{g,\cc{e,f}^{\pprime}_r}\)=\partial\cc{e,f}^{\prime}_r\!\otimes\!\ii{g,\cc{e,f}^{\pprime}_r}+\cc{e,f}^{\prime}_r\!\otimes\!\partial\ii{g,\cc{e,f}^{\pprime}_r}
\\\notag&\hspace*{-3.1ex}
=\partial\cc{e,f}^{\prime}_r\!\otimes\!\ii{g,\cc{e,f}^{\pprime}_r}+\cc{e,f}^{\prime}_r\!\otimes\(\cc{g,\cc{e,f}^{\pprime}_r}+\cc{\cc{e,f}^{\pprime}_r,g}^\sigma\)
\\\notag&\hspace*{-3.1ex}
=\cc{e,f}^{\prime}_r\otimes\cc{g,\cc{e,f}^{\pprime}_r}+\sigma_{(132)}\(\cc{e,f}^{\prime}_r\otimes_1\cc{\cc{e,f}^{\pprime}_r,g}+\partial\cc{e,f}^{\prime}_r\otimes_1\ii{\cc{e,f}^{\pprime}_r,g}\)
\\\label{eq:CD-identities.5.b}
&\hspace*{-3.1ex}
=\cc{g,\cc{e,f}_r}_R+\sigma_{(132)}\(\cc{\cc{e,f}_r,g}_R\).
\end{align}\end{subequations}
Combining \eqref{eq:decomposition-Courant-Dorfman-def}, \eqref{eq:CD-identities.3}, \eqref{eq:CD-identities.4}, \eqref{eq:CD-identities.5}, and using the fact that $\sigma_{(132)}=\sigma_{(123)}^{-1}$, we obtain
\begin{align}\label{eq:CD-identities.6}
\cc{e,\cc{f,g}}_R&
=\cc{\cc{e,f}_r,g}_R+\cc{f,\cc{e,g}_l}_L
-\sigma_{(123)}\(\cc{g,\cc{e,f}_l}_R+\cc{\cc{e,g}_r,f}_L\).
\end{align}
By~\eqref{eq:pairing-symm} and~\eqref{eq:extension-CD-bracket}, the last two terms in the right-hand side of~\eqref{eq:CD-identities.6} are
\begin{subequations}\label{eq:CD-identities.7}\begin{align}\notag
-\sigma_{(123)}\(\cc{g,\cc{e,f}_l}_R\)&
=-\sigma_{(123)}\big(\cc{e,f}^{\prime}_l\otimes\ii{g,\partial\cc{e,f}^{\pprime}_l}\big)
=-\cc{e,f}^{\prime}_l\otimes_1\ii{g,\partial\cc{e,f}^{\pprime}_l}^\sigma
\\\label{eq:CD-identities.7.a}&
=-\cc{e,f}^{\prime}_l\otimes_1\ii{\partial\cc{e,f}^{\pprime}_l,g}
=\cc{\cc{e,f}_l,g}_R,
\\\notag
-\sigma_{(123)}\(\cc{\cc{e,g}_r,f}_L\)&
=-\sigma_{(123)}\big(-\ii{\partial\cc{e,g}^{\prime}_r,f}\otimes_1\cc{e,g}^{\pprime}_r\big)
=\ii{\partial\cc{e,g}^{\prime}_r,f}^\sigma\otimes\cc{e,g}^{\pprime}_r
\\\label{eq:CD-identities.7.b}&
=\ii{f,\partial\cc{e,g}^{\prime}_r}\otimes\cc{e,g}^{\pprime}_r
=\cc{f,\cc{e,g}_r}_L.
\end{align}\end{subequations}
The required identity~\eqref{eq:CD-identities.1.e} follows now from~\eqref{eq:decomposition-Courant-Dorfman-def}, \eqref{eq:CD-identities.6} and \eqref{eq:CD-identities.7}.

To prove~\eqref{eq:CD-identities.1.f} and~\eqref{eq:CD-identities.1.g}, we first show that 
\begin{subequations}\label{eq:CD-identities.2}
\begin{align}\label{eq:CD-identities.2.a}
\ii{\cc{f,e}^\sigma,g}_L&=\ii{\cc{f,e},g}_R,
\\\label{eq:CD-identities.2.b}
\ii{\partial\ii{f,g},e}_L&=\ii{g,\cc{e,f}^\sigma}_R+\ii{f,\cc{e,g}^\sigma}_L.
\end{align}\end{subequations}
The identity~\eqref{eq:CD-identities.2.a} follows using~\eqref{eq:extension-pairing-first-second}, \eqref{eq:ext-pairing-zero}, \eqref{eq:decomposition-Courant-Dorfman-def} and \eqref{eq:decomposition-Courant-Dorfman-def-1}:
\begin{align*}
\ii{\cc{f,e}^\sigma,g}_L&=\ii{\cc{f,e}_r^\sigma,g}_L=\ii{\cc{f,e}^{\pprime}_r,g}\otimes_1\cc{f,e}^{\prime}_r
\\&
=\cc{f,e}^{\prime}_r\otimes_1\ii{\cc{f,e}^{\pprime}_r,g}
=\ii{\cc{f,e}_r,g}_R=\ii{\cc{f,e},g}_R.
\end{align*}
To obtain~\eqref{eq:CD-identities.2.b}, we apply $\sigma_{(132)}$ to \eqref{eq:CD.g} and use \eqref{eq:Sweedler-pairing-symm}, \eqref{eq:extension-pairing-first-second}, \eqref{eq:ext-pairing-zero}, \eqref{eq:partialLeibniz-rule} and \eqref{eq:decomposition-Courant-Dorfman-def-1}:
\begin{align*}&
\sigma_{(132)}\ii{e,\partial\ii{f,g}}_L=\sigma_{(132)}\big(\ii{e,\partial\ii{f,g}^{\prime}}\otimes\ii{f,g}^{\pprime}\big)
\\&\;\;
=\big(\ii{e,\partial\ii{f,g}^{\prime}}^{\pprime}\otimes\ii{e,\partial\ii{f,g}^{\prime}}^{\prime}\big)\otimes_1\ii{f,g}^{\pprime}
=\ii{\partial\ii{f,g}^{\prime},e}\otimes_1\ii{f,g}^{\pprime}
=\ii{\partial\ii{f,g},e}_L,
\\&
\sigma_{(132)}\ii{\cc{e,f},g}_L=\sigma_{(132)}\big(\ii{\cc{e,f}_l^{\prime},g}\otimes_1\cc{e,f}_l^{\pprime}\big)
=\cc{e,f}_l^{\pprime}\otimes\ii{\cc{e,f}_l^{\prime},g}^{\pprime}\otimes\ii{\cc{e,f}_l^{\prime},g}^{\prime}
\\&\;\;
=\cc{e,f}_l^{\pprime}\otimes\ii{g,\cc{e,f}_l^{\prime}}
=\ii{g,\cc{e,f}_l^\sigma}_R=\ii{g,\cc{e,f}^\sigma}_R,
\\&
\sigma_{(132)}\ii{f,\cc{e,g}}_R=\sigma_{(132)}\big(\cc{e,g}_r^{\prime}\otimes\ii{f,\cc{e,g}_r^{\pprime}}\big)
=\ii{f,\cc{e,g}_r^{\pprime}}\otimes\cc{e,g}_r^{\prime}=\ii{f,\cc{e,g}^\sigma}_L.
\end{align*}
Hence~\eqref{eq:CD-identities.2.b} is equivalent to~\eqref{eq:CD.g} (also using \eqref{eq:Sweedler-pairing-symm}), because the map $\sigma_{(132)}$ is additive.

The identity~\eqref{eq:CD-identities.1.f} now follows from \eqref{eq:CD.a}, \eqref{eq:CD.g} and \eqref{eq:CD-identities.2}:
\begin{align*}&
\ii{f,\partial\ii{e,g}}_R
=\ii{f,\cc{e,g}}_R+\ii{f,\cc{g,e}^\sigma}_R 
=\ii{e,\partial\ii{f,g}}_L-\ii{\cc{e,f},g}_L+\ii{f,\cc{g,e}^\sigma}_R 
\\&\;\;
=\ii{e,\cc{f,g}}_L+\ii{e,\cc{g,f}^\sigma}_L-\ii{\cc{e,f},g}_L+\ii{f,\cc{g,e}^\sigma}_R 
\\&\;\;
=\ii{e,\cc{f,g}}_L+\ii{\partial\ii{e,f},g}_L-\ii{\cc{e,f},g}_L 
=\ii{e,\cc{f,g}}_L+\ii{\cc{f,e}^\sigma,g}_L  
\\&\;\;
=\ii{e,\cc{f,g}}_L+\ii{\cc{f,e},g}_R.
\end{align*}

Finally,~\eqref{eq:CD-identities.1.g} follows because~\eqref{eq:CD.a}, \eqref{eq:CD.g} and \eqref{eq:CD-identities.2.b} imply
\begin{align*}&
\ii{f,\partial\ii{e,g}}_R
=\ii{f,\cc{e,g}}_R+\ii{f,\cc{g,e}^\sigma}_R
=\ii{e,\partial\ii{f,g}}_L-\ii{\cc{e,f},g}_L+\ii{f,\cc{g,e}^\sigma}_R
\\&\;\;
=\ii{e,\cc{f,g}}_L-\ii{\cc{e,f},g}_L+\ii{e,\cc{g,f}^\sigma}_L+\ii{f,\cc{g,e}^\sigma}_R
\\&\;\;
=\ii{e,\cc{f,g}}_L-\ii{\cc{e,f},g}_L+\ii{\partial\ii{e,f},g}_L.
&\qedhere
\end{align*}
\end{proof}

We define the \emph{anchor} and the \emph{coanchor} of an $R$-linear double Courant--Dorfman algebra $(A,E,\ii{-,-},\partial,\cc{-,-})$ as the $A$-bimodule morphisms (cf., e.g.,~\cite[\S 2.3]{Roy09})
\begin{subequations}
\begin{gather}\label{eq:anchor-CD-alg.1}
\rho_E\colon E\longrightarrow \D_R A,\quad e\longmapsto \ii{e,\partial\?},
\\\label{eq:coanchor-double-lie-algebroid-double-CD.1}
j_E=i_\partial\colon\Omega_R^1A\lto E,
\end{gather}\end{subequations}
respectively. Here,~\eqref{eq:coanchor-double-lie-algebroid-double-CD.1} is the image of $\partial\in\Der_R(A,E)$ under~\eqref{sub:ncdiffforms.corep}, so it is uniquely specified in terms of the universal $R$-linear derivation $\du{}\colon A\to\Omega_R^1A$ by the condition
\begin{equation}\label{eq:coanchor-double-lie-algebroid-double-CD.2}
\partial=j_E\circ\du{}.
\end{equation}
Then $\rho_E\circ j_E=0$, because~\eqref{eq:CD.c},~\eqref{eq:anchor-CD-alg.1} and~\eqref{eq:coanchor-double-lie-algebroid-double-CD.2} imply that for all $a,b\in A$,
\[
\big((\rho_E\circ j_E)(\du{a})\big)b=\ii{j_E(\du{a}),\partial b}=\ii{\partial a,\partial b}=0.
\]
In other words, we get an $A$-bimodule complex, called the \emph{$R$-linear tangent complex} of $E$:
\begin{equation}\label{eq:exact-double-Courant-algbd.ses.1}
0\to\Omega_R^1A\lra{j_E}E\lra{\rho_E}\DDer_RA\to 0.
\end{equation}


\begin{definition}\label{def:exact-double-Courant-algbd.1}
An $R$-linear double Courant--Dorfman algebra is called \emph{exact over $R$} if so is its $R$-linear tangent complex, that is,~\eqref{eq:exact-double-Courant-algbd.ses.1} is an $A$-bimodule short exact sequence.
\end{definition}

Definition~\ref{def:exact-double-Courant-algbd.1} is a noncommutative version of the corresponding notion~\cite[Definition 2.11]{Roy09} for Courant--Dorfman algebras over a commutative algebra.

\section{Exact double Courant--Dorfman algebras}
\label{sec:exact-DC-alg}

We apply now the Cartan calculus and the reduced Cartan calculus on double derivations described in Theorem \ref{thm:Cartan-identities.1} and Proposition~\ref{prop:reduced-Cartan-identities.1}, respectively, to construct double Courant--Dorfman algebras that are exact in the sense of Definition~\ref{def:exact-double-Courant-algbd.1}. 
They are noncommutative versions of the standard exact Courant algebroids and of their twists by closed 3-forms.

\subsection{The standard exact double Courant--Dorfman algebras}
\label{sub:standard-DC-alg.1}

The following construction is based on Courant's and Dorfman's original brackets (see~\cite[\S 2.3]{Cou90} and~\cite[Theorem 2.1]{Do93}).

\begin{theorem}\label{thm:standard-DCA}
Let $R$ be an algebra and $A$ a finitely generated $R$-algebra.
%
Then the 5-tuple
\[
\big(A,E,\ii{-,-},\partial,\cc{-,-}\big)
\]
given by the following data is an $R$-linear double Courant--Dorfman algebra (called the \emph{standard exact $R$-linear double Courant--Dorfman algebra over $A$}):
\begin{enumerate}
\item[\textup{(1)}]
the $R$-algebra $A$, 
\item[\textup{(2)}]
the $A$-bimodule
\begin{equation}\label{eq:thm:standard-DCA.1}
E\defeq \D_R A\oplus\Omega^1_RA,
\end{equation}
whose elements will be denoted as sums $X+\alpha$, with $X\in\D_R A$ and $\alpha\in\Omega^1_RA$, 
\item[\textup{(3)}]
the pairing $\ii{-,-}\colon E\otimes E\to A\otimes A$ given for all $X+\alpha,Y+\beta\in E$ by
\begin{equation}\label{eq:thm:standard-DCA.2}
\ii{X+\alpha,Y+\beta}\defeq i_X\beta+(i_Y\alpha)^\sigma,
\end{equation}
\item[\textup{(4)}]
the $R$-linear derivation $\partial\,\colon A\to E$ given for all $a\in A$ by
\begin{equation}\label{eq:thm:standard-DCA.3}
\partial a\defeq 0+\du{a},
\end{equation}
\item[\textup{(5)}]
the $R$-linear double Dorfman bracket $\cc{-,-}\colon E\otimes E\to E\otimes A\oplus A\otimes E$ given by
\begin{equation}\label{eq:thm:standard-DCA.4}
\cc{X+\alpha,Y+\beta}\defeq\db{X,Y}+L_X\beta-(i_Y\du{\alpha})^\sigma,
\end{equation}
for all $X+\alpha,Y+\beta\in E$. 
\end{enumerate}
This double Courant--Dorfman algebra is non-degenerate if $A$ is smooth over $R$.
\end{theorem}

\begin{proof}
We need to prove that the 5-tuple $(A,E,\ii{-,-},\partial,\cc{-,-})$ satisfies the axioms of Definition~\ref{CD-Definition}, and that the pairing $\ii{-,-}$ is non-degenerate when $A$ is smooth over $R$.

Note that the map $\ii{-,-}\colon E\otimes E\to A\otimes A$ given by~\eqref{eq:thm:standard-DCA.2} is indeed a pairing. This follows because the map $i\colon\DDer_RA\to(\Omega^1_RA)^\vee$ in~\eqref{eq:DDer-isom.1} is an $A$-bimodule morphism, \emph{i.e.}, 
\begin{equation}\label{eq:proof.thm:standard-DCA.1}
i_{a\cdot X\cdot b}\alpha=a*(i_X\alpha)*b,\quad i_X(a\alpha b)=a(i_X\alpha)b,
\end{equation}
for all $a,b\in A$, $X\in\D_R A $, $\alpha\in\Omega^1_RA$ (by~\eqref{eq:double-derivation.A-bimod-str.1}), and applying~\eqref{eq:cyclic-permutation.1} to~\eqref{eq:proof.thm:standard-DCA.1}, we get
\begin{equation}\label{eq:proof.thm:standard-DCA.2}
(i_{a\cdot X\cdot b}\alpha)^\sigma=a(i_X\alpha)^\sigma b,\quad (i_X(a\alpha b))^\sigma=a*(i_X\alpha)^\sigma*b,
\end{equation}
and hence the maps~\eqref{eq:pairing.1} are morphisms of $A$-bimodules for all $e=X+\alpha\in E$.
This pairing is symmetric, because
\begin{equation}\label{eq:proof.thm:standard-DCA.3}
\ii{Y+\beta,X+\alpha}^\sigma=\big(i_Y\alpha+(i_X\beta)^\sigma\big)^\sigma=i_X\beta+(i_Y\alpha)^\sigma=\ii{X+\alpha,Y+\beta}.
\end{equation}
Furthermore, it determines by restriction the $A$-bimodule isomorphism
\[
i\colon\DDer_RA\stackrel{\cong}{\lto}(\Omega^1_RA)^\vee,\quad  X\longmapsto \ii{X,\?}|_{\Omega^1_RA}=i_X
\]
in~\eqref{eq:DDer-isom.1}, and the $A$-bimodule morphism
\[
\mathtt{bidual}\colon\Omega^1_RA\lto(\DDer_RA)^\vee,\quad\alpha\longmapsto \ii{\alpha,\?}|_{\DDer_RA}=i(\alpha)^\sigma=\alpha^\vee,
\]
in~\eqref{eq:canonical-map.1} (by~\eqref{eq:canonical-map.3}), which is an isomorphism if $A$ is smooth over $R$, as already observed in~\secref{sec:nc-diff-forms-double-derivations}.
Moreover, $\ii{X,\?}$ vanishes on $\DDer_RA$ for all $X\in\D_R A$, and $\ii{\alpha,\?}$ vanishes on $\Omega^1_RA$ for all $\alpha\in\Omega^1_RA$. 
Since the pairing is symmetric, it follows then that it is non-degenerate when $A$ is smooth over $R$.

Regarding the map $\cc{\?,\?}$ given by~\eqref{eq:thm:standard-DCA.4}, observe that the right-hand side of~\eqref{eq:thm:standard-DCA.4} is in $E\otimes A\oplus A\otimes E$, because the first term is the double Schouten--Nijenhuis bracket (see~\eqref{eq:R-linear-double-SN-bracket.1})
\begin{equation}\label{eq:double-Schouten-Nijenhuis.1}
\db{X,Y}\in\D_RA\otimes A\oplus A\otimes\D_RA,
\end{equation}
the second term is defined using the Lie derivative $L_X\colon\Omega^1_RA\to\Omega^1_RA\otimes A\oplus A\otimes \Omega^1_RA$ (see~\eqref{eq:Lie-derivative.8.b}),
and the third term is defined using the universal differential $\du{}\colon\Omega^1_RA\to\Omega^2_RA$, the contraction map with an element $Y\in\D_RA$ (that is, $i_Y\colon\Omega^2_RA\to\Omega^1_RA\otimes A\oplus A\otimes\Omega^1_RA$ in~\eqref{eq:contraction.1} with $\ell=2$) and the permutation given by~\eqref{eq:cyclic-permutation.1}.

To prove the theorem, we need to show that the 5-tuple $(A,E,\ii{-,-},\partial,\cc{-,-})$ satisfies~\eqref{eq:CD}. 
We start proving~\eqref{eq:CD.a} using~\eqref{eq:double-Lie-albd.axioms.double-bracket.skewsymm.1} and the Cartan homotopic formula~\eqref{eq:Lie-derivative.2}: 
\begin{align*}
&\cc{X+\alpha,Y+\beta}+\cc{Y+\beta,X+\alpha}^\sigma
\\
&=\(\db{X,Y}+\db{Y,X}^\sigma\)
+\(\du{(i_X\beta)}+i_X(\du{\beta})\)+\big(\du{(i_Y\alpha)}+i_Y(\du{\alpha})\big)^\sigma
-i_Y(\du{\alpha})^\sigma-i_X(\du{\beta})
\\&
=\du{(i_X\beta)}+(\du{(i_Y\alpha)})^\sigma
=\du{\ii{X+\alpha,Y+\beta}}. 
\end{align*}
The identities~\eqref{eq:CD.b} and~\eqref{eq:CD.c} are obviously satisfied, so we omit their proofs. 

To prove~\eqref{eq:CD.d} and~\eqref{eq:CD.e}, we use the following identity, where $a\in A$ (see~\eqref{eq:Lie-derivative.6.a}):
\begin{equation}\label{eq:exact-CD.pairing-differential}
\ii{X+\alpha,\du{a}}=i_X(\du{a})=L_Xa.  
\end{equation}
Using this identity, \eqref{eq:Lie-derivative.5} (for $\ell=2$), \eqref{eq:Lie-derivative.6.a}, \eqref{eq:double-Lie-albd.axioms.double-Leibniz.1.a}, \eqref{eq:double-Lie-albd.axioms.double-Leibniz.1.b},\eqref{eq:proof.thm:standard-DCA.1} and~\eqref{eq:proof.thm:standard-DCA.2}, we obtain 
\begin{align*}
\cc{X+\alpha,(Y+\beta)a}&
=\db{X,Y\cdot a}+L_X(\beta a)-\(i_{Y\cdot a}\du{\alpha}\)^\sigma
\\&
=\(\db{X,Y}a+Y\cdot(L_Xa)\)+\((L_X\beta)a+\beta(L_Xa)\)-\((i_Y\du{\alpha})*a\)^\sigma
\\&
=\(\db{X,Y}+(L_X\beta)-(i_Y\du{\alpha})^\sigma\)a+(Y+\beta)(L_Xa)
\\&
=\cc{X+\alpha,Y+\beta}a+(Y+\beta)\ii{X+\alpha,\du{a}},
\\
\cc{X+\alpha,a(Y+\beta)}&
=\db{X,a\cdot Y}+L_X(a\beta)-(i_{a\cdot Y}\du{\alpha})^\sigma
\\&
=\(a\db{X,Y}+(L_Xa)\cdot Y\)+((L_Xa)\beta+aL_X\beta)-(a*(i_Y\du{\alpha}))^\sigma
\\&
=a\(\db{X,Y}+L_X\beta-(i_Y\du{\alpha})^\sigma)\)+(L_Xa)(Y+\beta)
\\&
=a\cc{X+\alpha,Y+\beta}+\ii{X+\alpha,\du{a}}(Y+\beta).
\end{align*}
Thus~\eqref{eq:CD.d} and~\eqref{eq:CD.e} are satisfied.

It remains to prove \eqref{eq:CD.f} and~\eqref{eq:CD.g}, that is, for all $X+\alpha, Y+\beta, Z+\gamma\in E$,
\begin{subequations}
\label{eq:standard-CD}
\begin{align}\label{eq:standard-CD.f}
\begin{split}
\cc{X+\alpha,\cc{Y+\beta,Z+\gamma}}_L&=\cc{\cc{X+\alpha,Y+\beta},Z+\gamma}_L
\\
&\qquad +\cc{Y+\beta,\cc{X+\alpha,Z+\gamma}}_R,  
\end{split}
\\\label{eq:standard-CD.g}\hspace*{-6ex}
\begin{split}
\ii{X+\alpha,\du{\ii{Y+\beta,Z+\gamma}}}_L&=\ii{\cc{X+\alpha,Y+\beta},Z+\gamma}_L
\\
&\qquad +\ii{Y+\beta,\cc{X+\alpha,Z+\gamma}}_R.
\end{split}
\end{align}
\end{subequations}
Since $\cc{\?,\?}$ and $\ii{\?,\?}$ are additive in each variable, to prove~\eqref{eq:standard-CD}, we can consider separately the cases $\alpha=0$ or $X=0$, and $\beta=0$ or $Y=0$, and $\gamma=0$ or $Z=0$, so~\eqref{eq:standard-CD.f} becomes 8 equations, and~\eqref{eq:standard-CD.g} becomes other 8 equations.
However the identity~\eqref{eq:standard-CD.f} is trivially satisfied if two of the elements $X,Y,Z$ are zero, because $\cc{\alpha,\beta}=0$ and $\cc{X,\alpha},\cc{\alpha,X}\in\Omega^1_RA$ for all $\alpha,\beta\in\Omega^1_RA$, $X\in\D_R A$, by \eqref{eq:thm:standard-DCA.4}. Hence to prove~\eqref{eq:standard-CD.f}, it suffices to show the following identities for all $X,Y,Z\in\D_RA$ and $\alpha,\beta,\gamma\in\Omega^1_RA$:
\begin{subequations}\label{eq:standard-CD.f.1}
\begin{align}\label{eq:standard-CD.f.1.a}
\cc{X,\cc{Y,Z}}_L&=\cc{\cc{X,Y},Z}_L+\cc{Y,\cc{X,Z}}_R,
\\\label{eq:standard-CD.f.1.b}
\cc{\cc{X,Y},\gamma}_L&=\cc{X,\cc{Y,\gamma}}_L-\cc{Y,\cc{X,\gamma}}_R,
\\\label{eq:standard-CD.f.1.c}
\cc{\alpha,\cc{Y,Z}}_L&=\cc{\cc{\alpha,Y},Z}_L+\cc{Y,\cc{\alpha,Z}}_R,  
\\\label{eq:standard-CD.f.1.d}
\cc{\beta,\cc{X,Z}}_R&=\cc{X,\cc{\beta,Z}}_L-\cc{\cc{X,\beta},Z}_L.
\end{align}
\end{subequations}
Similarly,~\eqref{eq:standard-CD.g} is trivially satisfied if two of the elements $X,Y,Z$ are zero or $\alpha=\beta=\gamma=0$, since $\ii{X,Y}=\ii{\alpha,\beta}=0$, $\cc{\alpha,\beta}=0$, $\cc{X,\alpha},\cc{\alpha,X}\in \Omega^1_R A\otimes A\oplus A\otimes\Omega^1_R A$ and $\cc{X,Y}\in\D_R A\otimes A\oplus A\otimes\D_R A$ for all $\alpha,\beta\in\Omega^1_R A$, $X,Y\in\D_R A$, by \eqref{eq:thm:standard-DCA.4}. Hence to prove~\eqref{eq:standard-CD.g}, it suffices to show the following identities for all $X,Y,Z\in\D_R A$ and $\alpha,\beta,\gamma\in\Omega^1_R A$: 
\begin{subequations}\label{eq:standard-CD.g.1}
\begin{align}\label{eq:standard-CD.g.1.a}
\ii{\cc{X,Y},\gamma}_L&=\ii{X,\du{\ii{Y,\gamma}}}_L-\ii{Y,\cc{X,\gamma}}_R,
\\\label{eq:standard-CD.g.1.b}
\ii{\beta,\cc{X,Z}}_R&=\ii{X,\du{\ii{\beta,Z}}}_L-\ii{\cc{X,\beta},Z}_L,
\\\label{eq:standard-CD.g.1.c}
\ii{\alpha,\du{\ii{Y,Z}}}_L&=\ii{Y,\cc{\alpha,Z}}_R+\ii{\cc{\alpha,Y},Z}_L,
\end{align}
\end{subequations}
To prove~\eqref{eq:standard-CD.f.1} and~\eqref{eq:standard-CD.g.1}, we first show the following identities for all $X,Y,Z\in\D_R A$, $\alpha,\beta\in\Omega^1_RA$, $\omega\in\Omega^2_RA$, where we use the decompositions~\eqref{eq:Sweedler-on-Cartan-calculus.1}, \eqref{eq:Sweedler-on-Cartan-calculus.3}, and set $\Omega\defeq\Omega^\bullet_RA$:
\begin{subequations}\label{eq:standard-CD.1}
\begin{align}\label{eq:standard-CD.1.a}
(i_Y\otimes\Id_\Omega)L_X\alpha&=(i_Y\otimes\Id_A)L_X^l\alpha,
\\\label{eq:standard-CD.1.b}
(\Id_\Omega\otimes i_Y)L_X\alpha&=(\Id_A\otimes i_Y)L_X^r\alpha,
\\\label{eq:standard-CD.1.c}
(i_Y\otimes\Id_\Omega)i_X\omega&=(i_Y\otimes\Id_A)i_X^l\omega
\\\label{eq:standard-CD.1.d}
(\Id_\Omega\otimes i_Y)i_X\omega&=(\Id_A\otimes i_Y)i_X^r\omega.
\\\label{eq:pf:prelim.standard-CD.f.1.c-and-d.2.f.2.b}
i_Z\du{i_Y^{l''}\!\du{\alpha}}=L_Zi_Y^{l''}\!\du{\alpha},
\;&\quad
i_Y\du{i_Z^{r'}\!\du{\alpha}}=L_Yi_Z^{r'}\!\du{\alpha},   
\\\label{eq:pf:prelim.standard-CD.f.1.c-and-d.2.f.1.a}
i_Z\!\du{L_X^{l'}\beta}\otimes L_X^{l''}\beta&=i_Z\!\du{i_X^{l'}\!\du{\beta}}\otimes i_Z^{l''}\!\du{\beta},
\\\label{eq:pf:prelim.standard-CD.f.1.c-and-d.2.f.1.b}
i_ZL_X^{l'}\beta\otimes\du{L_X^{l''}\beta}&=i_Z\du{i_X'\beta}\otimes\du{i_X''\beta}+i_Zi_X^{l'}\!\du{\beta}\otimes\du{i_X^{l''}\!\du{\beta}},
\\\label{eq:pf:prelim.standard-CD.f.1.c-and-d.2.f.1.c}
L_ZL_X^{r'}\beta\otimes L_X^{r''}\beta&=i_Z\!\du{i_X'\beta}\otimes\du{i_X''\beta}+i_Z\!\du{i_X^{r'}\!\du{\beta}}\otimes i_X^{r''}\!\du{\beta}.
\end{align}\end{subequations}
The identities~\eqref{eq:standard-CD.1.a}--\eqref{eq:standard-CD.1.d} follow from~\eqref{eq:Sweedler-on-Cartan-calculus.1}, as $i_Y(L^{r'}_X\alpha)=i_Y(L^{l''}_X\alpha)=i_Y(i^{r'}_X\omega)=i_Y(i^{l''}_X\omega)=0$ by~\eqref{eq:contraction.4.a} and~\eqref{eq:Sweedler-on-Cartan-calculus.4}.
The identities~\eqref{eq:pf:prelim.standard-CD.f.1.c-and-d.2.f.2.b} follow from \eqref{eq:relative-nc-all-diff-forms.3}, \eqref{eq:Lie-derivative.2}, \eqref{eq:contraction.4.a}, \eqref{eq:Sweedler-on-Cartan-calculus.3} and \eqref{eq:Sweedler-on-Cartan-calculus.4}.
The identities~\eqref{eq:pf:prelim.standard-CD.f.1.c-and-d.2.f.1.a}--\eqref{eq:pf:prelim.standard-CD.f.1.c-and-d.2.f.1.c} follow by applying  \eqref{eq:Lie-derivative.2}, \eqref{eq:contraction.4.a} and~\eqref{eq:Sweedler-on-Cartan-calculus.4} to $L_X^l\beta$ and $L_X^r\beta$, because 
\begin{align*}
L_X\beta=\du{(i_X'\beta\otimes i_X''\beta)}+i_X\du{\beta}=L_X^l\beta+L_X^r\beta,
\end{align*}
by \eqref{eq:Lie-derivative.2}, \eqref{eq:Sweedler-on-Cartan-calculus.1}, \eqref{eq:Sweedler-on-Cartan-calculus.2}, \eqref{eq:Sweedler-on-Cartan-calculus.3} and \eqref{eq:Sweedler-on-Cartan-calculus.4}, and so 
\begin{gather*}
L_X^l\beta=\du{i'_X\!\du{\beta}}\otimes i_X''\beta+i_X^{l'}\!\du{\beta}\otimes i_X^{l''}\!\du{\beta},
\\
L_X^r\beta=i_X'\beta\otimes\du{i''_X\!\du{\beta}}+i_X^{r'}\!\du{\beta}\otimes i_X^{r''}\!\du{\beta}.
\end{gather*}

We start now the proof of the identities~\eqref{eq:standard-CD.f.1}. 
The identity~\eqref{eq:standard-CD.f.1.a} follows because $\cc{X,Y}=\lr{X,Y}$ by \eqref{eq:thm:standard-DCA.4}, so~\eqref{eq:standard-CD.f.1.a} is equivalent to the double Jacobi identity~\eqref{eq:double-Lie-albd.axioms.double-Jacobi.1} for the double Schouten--Nijenhuis bracket $\lr{\?,\?}$, that holds by Theorem \ref{thm:VdB-Gerstenhaber-poly-der}.

To prove~\eqref{eq:standard-CD.f.1.b}, we calculate the three terms of this identity. 
By~\eqref{eq:thm:standard-DCA.4}, $\cc{X,Y}=\lb{X,Y}$, so by~\eqref{eq:R-linear-double-SN-bracket.2},~\eqref{eq:Sweedler.R-linear-double-SN-bracket.1}, \eqref{eq:lemma:Cartan-identities.1.1.c},~\eqref{eq:extension-CD-bracket},~\eqref{eq:thm:standard-DCA.2}--
\eqref{eq:thm:standard-DCA.4}, the left-hand side of~\eqref{eq:standard-CD.f.1.b} is 
\begin{equation}
\begin{aligned}
\cc{\cc{X,Y},\gamma}_L&=L_{\lb{X,Y}_l'}\gamma\otimes_1\lb{X,Y}_l''+\ii{\lb{X,Y}_l',\gamma}\otimes_1\partial{\lb{X,Y}_l''}
\\
&\qquad -\ii{\partial{\lb{X,Y}_r''},\gamma}\otimes_1\lb{X,Y}_r''
\\
&=L_{\lb{X,Y}_l'}\gamma\otimes_1\lb{X,Y}_l''+i_{\lb{X,Y}_l'}\gamma\otimes_1\du{\lb{X,Y}_l''}
=L^{l\sim}_{\lb{x,y}_l}\gamma.
\end{aligned}
\label{eq:pf:standard-CD.f.1.b.1}
\end{equation}
By~\eqref{eq:thm:standard-DCA.4}, $\cc{X,\gamma}=L_X\gamma$ and $\cc{Y,\gamma}=L_Y\gamma$, so applying~\eqref{eq:Sweedler-on-Cartan-calculus.1.a},~\eqref{eq:Sweedler-on-Cartan-calculus.3},~\eqref{eq:extension-CD-bracket} and~\eqref{eq:thm:standard-DCA.2}--
\eqref{eq:thm:standard-DCA.4}, the two terms of right-hand side of~\eqref{eq:standard-CD.f.1.b} are given by 
\begin{subequations}\label{eq:pf:standard-CD.f.1.b.2}
\begin{align}\notag
\cc{X,\cc{Y,\gamma}}_L&=\cc{X,L_Y^{l'}\gamma}\otimes L_Y^{l''}\gamma+\ii{X,\partial(L_Y^{r'}\gamma)}\otimes L_Y^{r''}\gamma
\\\notag&
=L_X(L_Y^{l'}\gamma)\otimes L_Y^{l''}\gamma+i_X(\du{(L_Y^{r'}\gamma}))\otimes L_Y^{r''}\gamma
\\\label{eq:pf:standard-CD.f.1.b.2.a}&
=(L_X\otimes\Id_A)L_Y^l\gamma+(L_X\otimes\Id_{\Omega^1})L_Y^r\gamma
=(L_X\otimes\Id_\Omega)L_Y\gamma,
\\\notag
\cc{Y,\cc{X,\gamma}}_R&=L_X^{l'}\gamma\otimes\ii{Y,\partial(L_X^{l''}\gamma)}+L_X^{r'}\gamma\otimes\cc{Y,L_X^{r''}\gamma}
\\\notag&
=L_X^{l'}\gamma\otimes i_Y(\du{(L_X^{l''}\gamma)})+L_X^{r'}\gamma\otimes L_Y(L_X^{r''}\gamma)
\\\label{eq:pf:standard-CD.f.1.b.2.b}&
=(\Id_{\Omega^1}\otimes L_Y)L_X^l\gamma+(\Id_A\otimes L_Y)L_X^r\gamma
=(\Id_{\Omega}\otimes L_Y)L_X\gamma,
\end{align}\end{subequations}
where we used $i_X(\du{(L_Y^{r'}\gamma)})=L_X(L_Y^{r'}\gamma)$ and $i_Y(\du{(L_X^{l''}\gamma)})=L_X(L_Y^{l''}\gamma)$ (see~\eqref{eq:Lie-derivative.6.a}) in~\eqref{eq:pf:standard-CD.f.1.b.2.a} and~\eqref{eq:pf:standard-CD.f.1.b.2.b}, respectively.
By~\eqref{eq:graded-Schouten-Nijenhuis.1.7.a} and~\eqref{eq:pf:standard-CD.f.1.b.2}, the right-hand side of~\eqref{eq:standard-CD.f.1.b} is 
\begin{equation}\label{eq:pf:standard-CD.f.1.b.3}
\cc{X,\cc{Y,\gamma}}_L-\cc{Y,\cc{X,\gamma}}_R
=(L_X\otimes\Id_\Omega)L_Y\gamma-(\Id_\Omega\otimes L_Y)L_X\gamma
=\db{L_X,L_Y}_l^\sim(\gamma).
\end{equation}
Therefore, it follows from~\eqref{eq:pf:standard-CD.f.1.b.1} and~\eqref{eq:pf:standard-CD.f.1.b.3} that~\eqref{eq:standard-CD.f.1.b} is equivalent to the identity
\[
L^{l\sim}_{\lb{X,Y}_l'}\gamma=\db{L_X,L_Y}_l^\sim(\gamma),
\]
which is satisfied by the Cartan identity~\eqref{eq:prop:Cartan-identities.1.1.g}.

To prove~\eqref{eq:standard-CD.f.1.c} and~\eqref{eq:standard-CD.f.1.d}, we first show that for all $X,Y,Z\in\D_R A$ and $\alpha\in\Omega^1_R A$, 
\begin{subequations}
\begin{gather}\label{eq:pf:standard-CD.f.1.c-and-d.1}\hspace*{-1ex}
\du{i_Xi_Y^{l'}\!\du{\alpha}}\otimes i_Y^{l''}\!\du{\alpha} + i_Xi_Y^{l'}\!\du{\alpha}\otimes\du{i_Y^{l''}\!\du{\alpha}}
+ \du{i_X^{r'}\!\du{\alpha}}\otimes i_Y\!\du{i_Y^{r''}\!\du{\alpha}} + i_X^{r'}\!\du{\alpha}\otimes\du{i_Yi_X^{r''}\!\du{\alpha}}=0, 
\\
\label{eq:pf:prelim.standard-CD.f.1.c-and-d.2.f.2.a}
L_Z\!\du{\alpha}
=\du{i_Z^{l'}\!\du{\alpha}}\otimes i_Z^{l''}\!\du{\alpha}-i_Z^{l'}\!\du{\alpha}\otimes\du{i_Z^{l''}\!\du{\alpha}}+\du{i_Z^{r'}\!\du{\alpha}}\otimes i_Z^{r''}\!\du{\alpha}+i_Z^{r'}\!\du{\alpha}\otimes\du{i_Z^{r''}\!\du{\alpha}}.
\end{gather}
\end{subequations}
The identity~\eqref{eq:pf:standard-CD.f.1.c-and-d.1} follows from $\du\lb{i_X,i_Y}_l^\sim\!(\du{\alpha})=0$ by~\eqref{eq:prop:Cartan-identities.1.1.a}, as the left-hand side is
\begin{align*}
\du\db{i_X,i_Y}_l^\sim\!(\du{\alpha})&=
\du{\big((i_X\otimes\!\Id_\Omega)i_Y^l\!\du{\alpha}+(\Id_\Omega\otimes\! i_Y)i_X^r\!\du{\alpha}\big)}
\\
&=\du{\big(i_Xi_Y^{l'}\!\du{\alpha}\otimes i_Y^{l''}\!\!\du{\alpha}+i_X^{r'}\!\du{\alpha}\otimes i_Y i_X^{r''}\!\!\du{\alpha}\big)}
\end{align*}
(by~\eqref{eq:Sweedler-on-Cartan-calculus.3},~\eqref{eq:graded-Schouten-Nijenhuis.1.7.a},~\eqref{eq:standard-CD.1.c} and~\eqref{eq:standard-CD.1.d}), which is equal to the left-hand side of~\eqref{eq:pf:standard-CD.f.1.c-and-d.1}. 
The identity~\eqref{eq:pf:prelim.standard-CD.f.1.c-and-d.2.f.2.a} follows from \eqref{eq:relative-nc-all-diff-forms.3}, \eqref{eq:Lie-derivative.2}, \eqref{eq:Sweedler-on-Cartan-calculus.3} and \eqref{eq:Sweedler-on-Cartan-calculus.4}. 
We will also use the following obvious identities, where $\alpha\in(\Omega^\bullet_R A)^{\otimes 2}$ and $\beta\in\Omega^\bullet_R A$: 
\begin{subequations}\label{eq:permutations-alpha-beta.1}
\begin{align}\label{eq:permutations-alpha-beta.1.a}
\beta\otimes\alpha&=\tau_{(123)}(\alpha\otimes\beta)=\tau_{(132)}(\beta\otimes\alpha),
\\\label{eq:permutations-alpha-beta.1.b}
\alpha^\sigma\otimes\beta&=\tau_{(123)}(\alpha\otimes_1\beta),
\quad
\beta\otimes\alpha^\sigma=\tau_{(123)}(\beta\otimes_1\alpha),
\\\label{eq:permutations-alpha-beta.1.c}
\alpha^\sigma\otimes_1\beta&=\tau_{(123)}(\beta\otimes\alpha)=\tau_{(132)}(\alpha\otimes\beta).
\end{align}\end{subequations}
Next, we calculate the terms involved in~\eqref{eq:standard-CD.f.1.c} and~\eqref{eq:standard-CD.f.1.d}, using \eqref{eq:Lie-derivative.2}, \eqref{eq:lemma:Cartan-identities.1.1.a}, \eqref{eq:lemma:Cartan-identities.1.1.b}, \eqref{eq:extension-CD-bracket}, \eqref{eq:decomposition-Courant-Dorfman-def}, \eqref{eq:decomposition-Courant-Dorfman-def-1}, \eqref{eq:thm:standard-DCA.2}--\eqref{eq:thm:standard-DCA.4} and~\eqref{eq:permutations-alpha-beta.1}:
\begin{subequations}\label{eq:pf:standard-CD.f.1.c-and-d.2}
\begin{align}\label{eq:pf:standard-CD.f.1.c-and-d.2.a}
\begin{split}
\cc{\alpha,\cc{Y,Z}}_L
&=\cc{\alpha,\lb{Y,Z}_l'}\otimes\lb{Y,Z}_l''+\ii{\alpha,\du{\lb{Y,Z}_r'}}\otimes\du{\lb{Y,Z}_r''}
\\&
=-\(i_{\lb{Y,Z}_l'}\du{\alpha}\)^\sigma\otimes\lb{Y,Z}_l''
=-\tau_{(123)}\(i_{\lb{Y,Z}_l'}\du{\alpha}\otimes_1\lb{Y,Z}_l''\)
\\&
=-\tau_{(123)}\(i^{\sim l}_{\lb{Y,Z}_l}\!\du{\alpha}\),
\end{split}
\\
\notag
\cc{\cc{\alpha,Y},Z}_L&=-\cc{(i_Y\!\du{\alpha})^\sigma,Z}_L
\\\notag&
=\ii{\du{i_Y^{l''}\!\du{\alpha}},Z}\otimes_1 i_Y^{l'}\!\du{\alpha}
-\cc{i_Y^{r''}\!\du{\alpha},Z}\otimes_1i_Y^{r'}\!\du{\alpha}
-\ii{i_Y^{r''}\!\du{\alpha},Z}\otimes_1\!\du{i_Y^{r'}\!\du{\alpha}}
\\\notag&
=(i_Z\!\du{i_Y^{l''}\!\du{\alpha}})^\sigma\otimes_1 i_Y^{l'}\!\du{\alpha}
+(i_Z\!\du{i_Y^{r''}\!\du{\alpha}})^\sigma\otimes_1i_Y^{r'}\!\du{\alpha}
-(i_Zi_Y^{r''}\!\du{\alpha})^\sigma\otimes_1\!\du{i_Y^{r'}\!\du{\alpha}}
\\\label{eq:pf:standard-CD.f.1.c-and-d.2.b}&
=-\tau_{(123)}(-i_Y^{l'}\!\du{\alpha}\otimes i_Z\!\du{i_Y^{l''}\!\du{\alpha}}\!-i_Y^{r'}\!\du{\alpha}\otimes i_Z\!\du{i_Y^{r''}\!\du{\alpha}} 
\!+\!\du{i_Y^{r'}\!\du{\alpha}}\otimes i_Zi_Y^{r''}\!\du{\alpha}),
\\\label{eq:pf:standard-CD.f.1.c-and-d.2.c}
\begin{split}
\cc{Y,\cc{\alpha,Z}}_R&=-\cc{Y,(i_Z\!\du{\alpha})^\sigma}_R
=-i_Z^{l''}\!\du{\alpha}\otimes\cc{Y,i_Z^{l'}\!\du{\alpha}}
-i_Z^{r''}\!\du{\alpha}\otimes\ii{Y,\du{i_Z^{r'}\!\du{\alpha}}}
\\&
=-i_Z^{l''}\!\du{\alpha}\otimes L_Y(i_Z^{l'}\!\du{\alpha})
-i_Z^{r''}\!\du{\alpha}\otimes i_Y\!\du{i_Z^{r'}\!\du{\alpha}}
\\&
=-\tau_{(123)}\(\du{i_Yi_Z^{l'}\!\du{\alpha}}\otimes i_Z^{l''}\!\du{\alpha} \!+ i_Y\!\du{i_Z^{l'}\!\du{\alpha}}\otimes i_Z^{l''}\!\du{\alpha}\! + i_Y\!\du{i_Z^{r'}\!\du{\alpha}}\otimes i_Z^{r''}\!\du{\alpha}\),
\end{split}
\\\label{eq:pf:standard-CD.f.1.c-and-d.2.d}
\begin{split}
\cc{\beta,\cc{X,Z}}_R&
=\lb{X,Z}_l'\otimes\ii{\beta,\du{\lb{X,Z}_l''}}+\lb{X,Z}_r'\otimes\cc{\beta,\lb{X,Z}_r''}
\\&
=-\lb{X,Z}_r'\otimes\(i_{\lb{X,Z}_r''}\!\du{\beta}\)^\sigma
=-\tau_{(132)}\(\lb{X,Z}_r'\otimes i_{\lb{X,Z}_r''}\!\du{\beta}\)
\\&
=-\tau_{(132)}\(i^{r\sim}_{\lb{X,Z}_r}\!\du{\beta}\),
\end{split}
\\\label{eq:pf:standard-CD.f.1.c-and-d.2.e}
\begin{split}
\cc{X,\cc{\beta,Z}}_L&=
\cc{X,(i_Z\!\du{\beta})^\sigma}_L
=-\cc{X,i_Z^{r''}\!\du{\beta}}\otimes i_Z^{r'}\!\du{\beta}+\ii{X,\du{i_Z^{l'}\!\du{\beta}}}\otimes i_Z^{l'}\!\du{\beta}
\\&
=-L_Xi_Z^{r''}\!\du{\beta}\otimes i_Z^{r'}\!\du{\beta}-i_X\!\du{i_Z^{l''}\!\du{\beta}}\otimes i_Z^{l'}\!\du{\beta}
\\&
=-\tau_{(132)}\!\(i_Z^{r'}\!\du{\beta}\otimes\du{i_Xi_Z^{r''}\!\du{\beta}}\!+i_Z^{r'}\!\du{\beta}\otimes i_X\!\du{i_Z^{r''}\!\du{\beta}} \!+ i_Z^{l'}\!\du{\beta}\otimes i_X\!\du{i_Z^{l''}\!\du{\beta}}\).
\end{split}
\end{align}\end{subequations}
To obtain the last term of~\eqref{eq:standard-CD.f.1.d}, we also need to apply~\eqref{eq:pf:prelim.standard-CD.f.1.c-and-d.2.f.1.a}--\eqref{eq:pf:prelim.standard-CD.f.1.c-and-d.2.f.1.c}: 
\begin{equation}\label{eq:pf:standard-CD.f.1.c-and-d.2.f}\begin{split}
\cc{\cc{X,\beta},Z}_L&
=\cc{L_X^{l'}\beta,Z}\otimes_1L_X^{l''}\beta+\ii{L_X^{l'}\beta,Z}\otimes_1\du{L_X^{l''}\beta}-\ii{\du{L_X^{r'}\beta},Z}\otimes_1L_X^{r''}\beta
\\&
=-(i_Z\!\du{L_X^{l'}\beta})^\sigma\otimes_1L_X^{l''}\beta+(i_ZL_X^{l'}\beta)^\sigma\otimes_1\du{L_X^{l''}\beta}-(L_ZL_X^{r'}\beta)^\sigma\otimes_1L_X^{r''}\beta
\\&
=-\tau_{(132)}\(i_Z\!\du{L_X^{l'}\beta}\otimes L_X^{l''}\beta-i_ZL_X^{l'}\beta\otimes\du{L_X^{l''}\beta}+L_ZL_X^{r'}\beta\otimes L_X^{r''}\beta\)
\\&
=-\tau_{(132)}\(i_Z\!\du{i_X^{l'}\!\du{\beta}}\otimes i_X^{l''}\!\du{\beta}-i_Zi_X^{l'}\!\du{\beta}\otimes\du{i_X^{l''}\!\du{\beta}}+i_Z\!\du{i_X^{r'}\!\du{\beta}}\otimes i_X^{r''}\!\du{\beta}\).
\end{split}\end{equation}
To conclude the proof of~\eqref{eq:standard-CD.f.1.c} and~\eqref{eq:standard-CD.f.1.d},
we first use~\eqref{eq:pf:prelim.standard-CD.f.1.c-and-d.2.f.2.b} and~\eqref{eq:pf:prelim.standard-CD.f.1.c-and-d.2.f.2.a}, obtaining
\begin{subequations}\label{eq:pf:prelim.standard-CD.f.1.c-and-d.2.f.3}
\begin{align}\label{eq:pf:prelim.standard-CD.f.1.c-and-d.2.f.3.a}
(i_Y\otimes\Id_\Omega)\circ L_Z\!\du{\alpha}&
=i_Y\!\du{i_Z^{l'}\!\du{\alpha}}\otimes i_Z^{l''}\!\du{\alpha}-i_Yi_Z^{l'}\!\du{\alpha}\otimes\du{i_Z^{l''}\!\du{\alpha}}+i_Y\!\du{i_Z^{r'}\!\du{\alpha}}\otimes i_Z^{r''}\!\du{\alpha},
\\
(\Id_\Omega\otimes L_Z)\circ i_Y\!\du{\alpha}&
=i_Y^{l'}\!\du{\alpha}\otimes i_Z\!\du{i_Y^{l''}\!\du{\alpha}}+i_Y^{r'}\!\du{\alpha}\otimes\du{i_Zi_Y^{r''}\!\du{\alpha}}+i_Y^{r'}\!\du{\alpha}\otimes i_Z\!\du{i_Y^{r''}\!\du{\alpha}}.
\end{align}\end{subequations}
By a similar calculation (or reading~\eqref{eq:pf:prelim.standard-CD.f.1.c-and-d.2.f.3} from right to left, exchanging the superscripts $l', l''$ with $r'', r'$, respectively, and replacing $Y$ and $\alpha$ by $X$ and $\beta$, respectively), we get
\begin{subequations}\label{eq:pf:prelim.standard-CD.f.1.c-and-d.2.f.4}
\begin{align}\label{eq:pf:prelim.standard-CD.f.1.c-and-d.2.f.4.a}
(\Id_\Omega\otimes i_X)\circ L_Z\!\du{\beta}&
=i_Z^{l'}\!\du{\beta}\otimes i_X\!\du{i_Z^{l''}\du{\beta}}+i_Z^{r'}\!\du{\beta}\otimes i_X\!\du{i_Z^{r''}\du{\beta}}-\du{i_Z^{r'}\!\du{\beta}}\otimes i_Xi_Z^{r''}\!\du{\beta},
\\
(L_Z\otimes\Id_\Omega)\circ i_X\!\du{\beta}&
=\du{i_Zi_X^{l'}\!\du{\beta}}\otimes i_X^{l''}\!\du{\beta}+i_Z\!\du{i_X^{l'}\!\du{\beta}}\otimes i_X^{l''}\!\du{\beta}+i_Z\!\du{i_X^{r'}\!\du{\beta}}\otimes i_X^{r''}\!\du{\beta}.
\end{align}\end{subequations}
By \eqref{eq:graded-Schouten-Nijenhuis.1.7}, \eqref{eq:pf:standard-CD.f.1.c-and-d.1} (with $X,Y,\alpha$ replaced by $Y,Z,\alpha$ to get~\eqref{eq:pf:standard-CD.f.1.c-and-d.3.a} and by $Z,X,\beta$ to get~\eqref{eq:pf:standard-CD.f.1.c-and-d.3.b}) and \eqref{eq:pf:standard-CD.f.1.c-and-d.2}--\eqref{eq:pf:prelim.standard-CD.f.1.c-and-d.2.f.4}, the right-hand sides of~\eqref{eq:standard-CD.f.1.c}, \eqref{eq:standard-CD.f.1.d} are
\begin{subequations}\label{eq:pf:standard-CD.f.1.c-and-d.3}
\begin{align}\label{eq:pf:standard-CD.f.1.c-and-d.3.a}
\begin{split}
\cc{\cc{\alpha,Y},Z}_L+\cc{Y,\cc{\alpha,Z}}_R&
=-\tau_{(123)}\((i_Y\otimes\Id_\Omega)\circ L_Z\!\du{\alpha}-(\Id_\Omega\otimes L_Z)\circ i_Y\!\du{\alpha}\)
\\&
=-\tau_{(123)}\(\db{i_Y,L_Z}_l^\sim(\du{\alpha})\),
\end{split}
\\\label{eq:pf:standard-CD.f.1.c-and-d.3.b}
\begin{split}
\cc{X,\cc{\beta,Z}}_L-\cc{\cc{X,\beta},Z}_L&
=-\tau_{(132)}\((\Id_\Omega\otimes i_X)\circ L_Z\!\du{\beta}-(L_Z\otimes\Id_\Omega)\circ i_X\!\du{\beta}\)
\\&
=-\tau_{(132)}\(\db{i_X,L_Z}_r^\sim(\du{\beta})\).
\end{split}\end{align}\end{subequations}
Therefore~\eqref{eq:standard-CD.f.1.c} and~\eqref{eq:standard-CD.f.1.d} follow from~\eqref{eq:prop:Cartan-identities.1.1.c},~\eqref{eq:prop:Cartan-identities.1.1.d}, \eqref{eq:pf:standard-CD.f.1.c-and-d.2.a}, \eqref{eq:pf:standard-CD.f.1.c-and-d.2.d} and \eqref{eq:pf:standard-CD.f.1.c-and-d.3}. 

We prove now~\eqref{eq:standard-CD.g.1}. 
We start computing the terms involved in these identities, using \eqref{eq:Sweedler.R-linear-double-SN-bracket.1}, \eqref{eq:double-Lie-albd.axioms.double-Jacobi.1}, \eqref{eq:extension-pairing-first-second}, \eqref{eq:ext-pairing-zero}, 
\eqref{eq:extended-differential.2}, \eqref{eq:contraction.6}, \eqref{eq:Lie-derivative.6.a}, \eqref{eq:Sweedler-on-Cartan-calculus.1}--\eqref{eq:Sweedler-on-Cartan-calculus.4}, \eqref{eq:lemma:Cartan-identities.1.1.a}, \eqref{eq:lemma:Cartan-identities.1.1.b}, \eqref{eq:thm:standard-DCA.2}--\eqref{eq:thm:standard-DCA.4} and \eqref{eq:standard-CD.1}:
\begin{subequations}\label{eq:pf:standard-CD.g.1}
\begin{align*}
\ii{\cc{X,Y},\gamma}_L&=\ii{\lb{X,Y},\gamma}_L
=\ii{\lb{X,Y}_l',\gamma}\otimes_1\lb{X,Y}_l''
\\&
=i_{\lb{X,Y}_l'}(\gamma)\otimes_1\lb{X,Y}_l''=i^{l\sim}_{\lb{X,Y}_l}\gamma,
\\
\ii{X,\du{\ii{Y,\gamma}}}_L&=\ii{X,\du{(i_Y\gamma)}}_L=\ii{X,\du{(i'_Y\gamma)}}\otimes i''_Y\gamma
=i_X\du{(i'_Y\gamma)}\otimes i''_Y\gamma
\\&
=L_X(i'_Y\gamma)\otimes i''_Y\gamma=((L_X\otimes\Id_\Omega)\circ i_Y)(\gamma), 
\\
\ii{Y,\cc{X,\gamma}}_R&=\ii{Y,L_X\gamma}_R=L^{r'}_X\gamma\otimes i_Y(L^{r''}_X\gamma)
\\&
=((\Id_A\otimes i_Y)\circ L^r_X)(\gamma)=(\Id_A\otimes i_Y)(L_X\gamma),
\\
\ii{\beta,\cc{X,Z}}_R&=\ii{\beta,\lb{X,Z}}_R=\lb{X,Z}_r'\otimes\ii{\beta,\lb{X,Z}_r''}=\lb{X,Z}_r'\otimes\(i_{\lb{X,Z}_r''}\beta\)^\sigma
\\&
=\tau_{(132)}\(\lb{X,Z}_r'\otimes_1 i_{\lb{X,Z}_r''}\beta\)=\tau_{(132)}\(i^{r\sim}_{\lb{X,Z}_r}\beta\),
\\
\ii{X,\du{\ii{\beta,Z}}}_L&=\ii{X,\du{(i_Z\beta)}^\sigma}_L
=\ii{X,\du{(i''_z\beta)}}\otimes i'_Z\beta=i_X(\du{(i''_Z\beta)})\otimes i'_Z\beta
\\&
=L_X(\du{(i''_Z\beta)})\otimes i'_Z\beta=\tau_{(132)}\big(i'_Z\beta\otimes L_X(i''_Z\beta)\big)
=\tau_{(132)}\big((\Id_A\otimes L_X)(i_Z\beta)\big),
\\
\ii{\cc{X,\beta},Z}_L&
=\ii{L_X\beta,Z}_L=\ii{L^{l'}_X\beta,Z}\otimes_1L^{l''}_X\beta=\(i_Z(L^{l'}_X\beta)\)^\sigma\otimes_1L^{l''}_X\beta
\\
&=\tau_{(132)}\!\big(i_Z(L^{l'}_X\beta)\otimes L^{l''}_X\beta\big)\!
=\tau_{(132)}\!\((i_Z\!\otimes \Id_A)(L^l_Z\beta)\)\!
\\
&=\tau_{(132)}\big((i_Z\otimes \Id_A)(L_Z\beta)\big),
\\
\ii{\alpha,\du{\ii{Y,Z}}}_L&=0,
\\
\ii{Y,\cc{\alpha,Z}}_R&
=-\ii{Y,(i_Z\du{\alpha})^\sigma}_R=-i^{l''}_Z\du{\alpha}\otimes\ii{Y,i^{l'}_Z\du{\alpha}}_R
=-i^{l''}_Z\du{\alpha}\otimes i_Y(i^{l'}_Z\du{\alpha})
\\&
=-\tau_{(123)}\big(i_Y(i^{l'}_Z\du{\alpha})\otimes i^{l''}_Z\du{\alpha}\big)
=-\tau_{(123)}\big((i_Y\otimes\Id_A)(i^l_Z(\du{\alpha}))\big)
\\&
=-\tau_{(123)}\big((i_Y\otimes\Id_\Omega)(i_Z(\du{\alpha}))\big),
\\
\ii{\cc{\alpha,Y},Z}_L&
=-\ii{(i_Y\du{\alpha})^\sigma,Z}_L=-\ii{i^{r''}_Y\du{\alpha},Z}\otimes_1i^{r'}_Y\du{\alpha}
=-\(i_Z(i^{r''}_Y\du{\alpha})\)^\sigma\otimes_1i^{r'}_Y\du{\alpha}
\\&
=-\tau_{(123)}\big(i^{r'}_Y\du{\alpha}\otimes i_Z(i^{r''}_Y\du{\alpha})\big)
=-\tau_{(123)}\big((\Id_A\otimes i_Z)i^r_Y(\du{\alpha})\big)
\\&
=-\tau_{(123)}\big((\Id_A\otimes i_Z)i_Y(\du{\alpha})\big).
\end{align*}
\end{subequations}
Using these calculations and~\eqref{eq:graded-Schouten-Nijenhuis.1.7}, we can now obtain the right-hand sides of~\eqref{eq:standard-CD.g.1}: 
\begin{subequations}
\begin{align*}
\ii{X,\du{\ii{Y,\gamma}}}_L-\ii{Y,\cc{X,\gamma}}_R
&=\db{i_X,L_Y}_l^\sim(\gamma),
\\
\ii{X,\du{\ii{\beta,Z}}}_L-\ii{\cc{X,\beta},Z}_L
&=\tau_{(132)}\big(\db{L_X,i_Z}_r^\sim(\beta)\big),
\\
\ii{Y,\cc{\alpha,Z}}_R+\ii{\cc{\alpha,Y},Z}_L
&=-\tau_{(123)}\big(\db{i_Z,i_Y}_l^\sim(\du{\alpha})\big). 
\end{align*}
\end{subequations}
Therefore,~\eqref{eq:standard-CD.g.1.a} follows from~\eqref{eq:prop:Cartan-identities.1.1.c}, 
\eqref{eq:standard-CD.g.1.b} follows from \eqref{eq:prop:Cartan-identities.1.1.d} and \eqref{eq:prop:Cartan-identities.1.1.f}, and \eqref{eq:standard-CD.g.1.c} follows from~\eqref{eq:prop:Cartan-identities.1.1.a}. 
This completes the proof of Theorem \ref{thm:standard-DCA}.
\end{proof}

\subsection{Twisting the standard exact double Courant--Dorfman algebra}
\label{sec:twisted}

We construct now a noncommutative analogue of \v{S}evera's deformation~\cite{Sev00} of the standard Courant (or Dorfman) bracket by a closed 3-form on a manifold (see also~\cite{SW01} for the geometry of twisted Poisson structures and~\cite[\S2.4]{Roy09} for the generalization to Courant--Dorfman algebras).


Using the standard exact $R$-linear double Courant--Dorfman algebra
\[
(A,E,\ii{-,-},\partial,\cc{-,-})
\]
of Theorem \ref{thm:standard-DCA}, where $E=\D_R A\oplus\Omega^1_RA$, and a 3-form $H\in \dR_R^3 A$ in the Karoubi--de Rham complex (see~\eqref{eq:def-Karoubi-de-Rham}), we define the \emph{$H$-twisted double Dorfman bracket} 
\begin{equation}\label{eq:def-H-twisted-double-bracket.1}
\cc{-,-}_H\colon E\times E\longrightarrow E\otimes A\oplus A\otimes E,
\end{equation}
by the following formula for all $X+\alpha,Y+\beta\in E$: 
\begin{equation}\label{eq:def-H-twisted-double-bracket.2}
\cc{X+\alpha,Y+\beta}_H\defeq \cc{X+\alpha,Y+\beta}+i_X\iota_YH.
\end{equation}

\begin{theorem}\label{thm:twist-CD-algebra-double}
Let $(A,E,\ii{-,-},\partial,\cc{-,-})$ be the standard exact $R$-linear double Courant--Dorfman algebra of Theorem \ref{thm:standard-DCA}, and $H\in\dR_R^3 A$.
If $H$ is closed in $\dR_R^4A$, then $(A,E,\ii{-,-},\partial,\cc{-,-}_H)$ is an $R$-linear double Courant--Dorfman algebra.
\end{theorem}

\begin{remark}\label{rem:twisted-CDA.1}
In the special case of quiver path algebras, a similar result was proved in~\cite[Lemma 8.5]{ACF17} in the formalism of \emph{noncommutative} differential graded symplectic manifolds.
\end{remark}

\begin{proof}
We need to check that $(A,E,\ii{-,-},\partial,\cc{-,-}_H)$ satisfies the axioms of Definition \ref{CD-Definition}.
Since $\iota_YH\in\Omega^2$, where $\Omega\defeq\Omega^\bullet_RA$ for simplicity, applying the decomposition~\eqref{eq:Sweedler-on-Cartan-calculus.1.b}, 
\[
i_X\iota_YH=i^l_X\iota_YH+i^r_X\iota_YH\in\Omega^1\otimes A\oplus A\otimes\Omega^1,
\]
where $i^l_X\iota_YH=i^{l'}_X\iota_YH\otimes i^{l''}_X\iota_YH\in \Omega^1\otimes A$ and $i^r_X\iota_YH =i^{r'}_X\iota_YH\otimes i^{r''}_X\iota_YH \in A\otimes \Omega^1$, with $i^{l'}_X\iota_YH,  i^{r''}_X\iota_YH\in\Omega^1$ and $i^{r'}_X\iota_YH,  i^{l''}_X\iota_YH\in A$.
We will prove now the identities~\eqref{eq:CD} for $\cc{-,-}_H$. 
While \eqref{eq:CD.b} and \eqref{eq:CD.c} are immediate by the definition of $\cc{-,-}_H$, 
the first one \eqref{eq:CD.a} follows from the corresponding identity for $\cc{-,-}$ and \eqref{eq:identidades-cruciales-reducida.a}:
\begin{align*}
\cc{X+\alpha,Y+\beta}_H&+ \cc{Y+\beta, X+\alpha}^\sigma_H
\\
&= \cc{X+\alpha,Y+\beta}+ i_X\iota_Y H + \cc{Y+\beta, X+\alpha}^\sigma+\big(i_Y\iota_X H\big)^\sigma
\\
&=\du{\ii{X+\alpha,Y+\beta}}.
\end{align*}
The axiom \eqref{eq:CD.d} follows from the corresponding one for $\cc{-,-}$ and the facts that $\iota_{Y\cdot a}H=(\iota_Y H)a$ for $a\in A$ (see the proof of \cite[Lemma 2.8.6]{CBEG07}) and $i_X$ is an $A$-bimodule morphism:
\begin{align*}
\cc{X+\alpha, (Y+\beta)a}_H&=\cc{X+\alpha, Ya+\beta a}+i_X\iota_{Ya}H
\\
&=\cc{X+\alpha, (Y+\beta)a}+i_X\big((\iota_{Y}H)a\big)
\\
&=\cc{X+\alpha, Y+\beta} a + (Y+\beta)\ii{X+\alpha, \du{a}} + \big(i_X\iota_Y H\big)a
\\
&=\cc{X+\alpha, Y+\beta}_H a + (Y+\beta)\ii{X+\alpha, \du{a}}.
\end{align*}
The axiom \eqref{eq:CD.e} is similar, so we leave the proof to the reader.

The double Jacobi identity \eqref{eq:CD.f} for $\cc{-,-}_H$ is
\begin{equation}
\cc{e,\cc{f,g}_H}_{H,L}= \cc{\cc{e,f}_H,g}_{H,L} + \cc{f,\cc{e,g}_H}_{H,R},
\label{eq:Jacobi-doble-twist-inicio}
\end{equation}
where $e=X+\alpha$, $f=Y+\beta$ and $g=Z+\gamma$, with $X,Y,Z\in\D_R A$ and $\alpha,\beta,\gamma\in\Omega^1$.
To prove~\eqref{eq:Jacobi-doble-twist-inicio}, we calculate the left-hand side using~\eqref{eq:contraction.4.a}, \eqref{eq:thm:Cartan-identities.1.b}, \eqref{eq:reduced-homotopic-Cartan}, \eqref{eq:identidades-cruciales-reducida.b}, \eqref{eq:extension-CD-bracket}:
\begin{align*}
\cc{e,\cc{f,g}_H}_{H,L}&=\cc{X+\alpha,\cc{Y+\beta,Z+\gamma}_H}_{H,L}
\\
&=\cc{X+\alpha,\cc{Y+\beta,Z+\gamma}}_{L}
+(i_X\otimes \Id_A)\iota_{\lr{Y,Z}'_l}H\otimes\lr{Y,Z}''_l
\\
&\qquad +\cc{X+\alpha, i^{l'}_Y\iota_ZH}_H\otimes i^{l''}_Y\iota_ZH+\cc{X+\alpha, i^{r'}_Y\iota_ZH}_H\otimes i^{r''}_Y\iota_ZH
\\
&=\cc{e,\cc{f,g}}_L
+(i_X\otimes \Id_{\Omega})\big(\iota_{\lr{Y,Z}'_l}H\otimes\lr{Y,Z}''_l+\lr{Y,Z}'_r\otimes\iota_{\lr{Y,Z}''_r}H\big)
\\
&\qquad +(L_Z\otimes \Id_A)i^l_Y\iota_ZH+i_X(\du{}i^{r'}_Y\iota_ZH)\otimes i^{r''}_Y\iota_ZH
\\
&=\cc{e,\cc{f,g}}_{L}
+(i_X\otimes \Id_{\Omega})\big(i_Y\mathcal{L}_ZH-\sigma_{(12)}L_Z\iota_YH\big)
\\
&\qquad +(L_Z\otimes \Id_A)i^l_Y\iota_ZH+(L_X\otimes \Id_{\Omega^1})i^{r}_Y\iota_ZH
\\
&=\cc{e,\cc{f,g}}_{L}
+(i_X\otimes \Id_{\Omega})\big(i_Y\mathcal{L}_ZH-\sigma_{(12)}L_Z\iota_YH\big) 
+(L_Z\otimes \Id_{\Omega})i_Y\iota_ZH
\\
&=\cc{e,\cc{f,g}}_{L}
+(i_X\otimes \Id_{\Omega})\big(i_Y(\du{\iota_ZH})-\sigma_{(12)}L_Z\iota_YH\big) 
+(L_Z\otimes \Id_{\Omega})i_Y\iota_ZH,
\end{align*}
where in the last identity we used the fact that $H$ is closed.

To simplify the handling of the first term on the right-hand side of  \eqref{eq:Jacobi-doble-twist-inicio}, we perform a preliminary calculation. 
By \eqref{eq:graded-Schouten-Nijenhuis.1.7.a}, \eqref{eq:lemma:Cartan-identities.1.1.a}, \eqref{eq:prop:Cartan-identities.1.1.c} and \eqref{eq:thm:Cartan-identities.1.b}, we have
\begin{align*}
&i_{\lr{X,Y}'_l}\iota_ZH\otimes_1\lr{X,Y}''_l=i^{l\sim}_{\lr{X,Y}_l}(\iota_ZH)=\lr{i_X,L_Y}^\sim_l(\iota_ZH)
\\
&=(i_X\otimes\Id_{\Omega})L_Y(\iota_ZH)
-(\Id_{\Omega}\otimes L_Y)i_X(\iota_ZH)
\\
&=(i_X\du{})(i^{l'}_Y\iota_ZH)\otimes i^{l''}_Y\iota_ZH-i_X(i^{l'}_Y\iota_ZH)\otimes\du{} i^{l''}_Y\iota_ZH+(i_X\du{})(i^{r'}_Y\iota_ZH)\otimes i^{r''}_Y\iota_ZH)
\\
&\qquad +(i_X\otimes \Id_{\Omega})(i_Y\du{})(\iota_ZH)-(\Id_{\Omega}\otimes L_Y)i_X(\iota_ZH)
\\
&=(L_X\otimes\Id_A)i^l_Y\iota_ZH-(\du{}i_X\otimes \Id_A)i^l_Y\iota_ZH-(i_X\otimes \Id_{\Omega^1})(\Id_{\Omega}\otimes \du{})i^l_Y\iota_ZH+(L_X\otimes\Id_{\Omega^1})i^r_Y\iota_ZH
\\
&\qquad +(i_X\otimes\Id_{\Omega})(i_Y\du{})(\iota_ZH)-(\Id_{\Omega}\otimes L_Y)i_X(\iota_ZH).
\end{align*}
Using this calculation and \eqref{eq:extension-CD-bracket}, the first term on the right-hand side of \eqref{eq:Jacobi-doble-twist-inicio} becomes
\begin{align*}
&\cc{\cc{e,f}_H,g}_{H,L}=\cc{\cc{X+\alpha,Y+\beta}_H,Z+\gamma}_{H,L}
\\
&=\cc{\cc{X+\alpha,Y+\beta},Z+\gamma}_{L}+ i_{\lr{X,Y}'_l}\iota_ZH\otimes_1\lr{X,Y}''_l + \cc{i^{l'}_X\iota_YH,Z+\gamma}\otimes_1 i^{l''}_X\iota_YH
\\
&\qquad +\ii{i^{l'}_X\iota_YH,Z+\gamma}\otimes_1\du{}i^{l''}_X\iota_YH-\ii{\du{}i^{r'}_X\iota_YH,Z+\gamma}\otimes_1 i^{r''}_X\iota_YH
\\
&= \cc{\cc{e,f},g}_{L}
+ (L_X\otimes\Id_A)i^l_Y\iota_ZH-(\du{}i_X\otimes \Id_A)i^l_Y\iota_ZH-(i_X\otimes \Id_{\Omega^1})(\Id_{\Omega}\otimes \du{})i^l_Y\iota_ZH\\
&\qquad +(L_X\otimes\Id_{\Omega^1})i^r_Y\iota_ZH
+(i_X\otimes\Id_{\Omega})i_Y\du{}(\iota_ZH)-(\Id_{\Omega}\otimes L_Y)i_X(\iota_ZH)
\\
&\qquad + \tau_{(132)}\big((i_Z\otimes \Id_{\Omega^1})(\Id_{\Omega^1}\otimes\du{})i^l_X\iota_YH-(i_Z\du{}\otimes\Id_{\Omega})i^l_X\iota_YH\big).
\end{align*}
By \eqref{eq:prop:Cartan-identities.1.1.d}, \eqref{eq:thm:Cartan-identities.1.b}, \eqref{eq:identidades-cruciales-reducida.a} and \eqref{eq:extension-CD-bracket}, the last term of \eqref{eq:Jacobi-doble-twist-inicio} is
\begin{align*}
&\cc{f,\cc{e,g}_H}_{H,R}=\cc{y+\beta, \cc{X+\alpha,Z+\gamma}_H}_{H,R}
\\
&=\cc{f,\cc{e,g}}_{R} \!+\!\lr{X,Z}'_r\!\otimes \!i_Y\iota_{\lr{X,Z}''_r}H\!+\!i^{l'}_X\iota_ZH\!\otimes\!\ii{ Y\!+\!\beta,\du{} i^{l''}_X\iota_ZH}+i^{r'}_X\iota_ZH\!\otimes \! L_Y(i^{r''}_X\iota_ZH)
\\
&=\cc{f,\cc{e,g}}_{R}
 \!-\!\lr{X,Z}'_r\!\otimes \!\big(i_{\lr{X,Z}''_r}\iota_YH\big)^\sigma
\!+\! (\Id_{\Omega^1}\!\otimes \! i_Y)(\Id_{\Omega^1}\!\otimes\!\du{})i^l_X\iota_ZH\!+\!(\Id_A\!\otimes\! L_Y)i^r_X\iota_ZH
\\
&=\cc{f,\cc{e,g}}_{R}
 -\tau_{(132)}\lr{i_X,L_Z}^\sim_r(\iota_YH)
+ (\Id_{\Omega^1}\otimes L_Y)i^l_X\iota_ZH+(\Id_A\otimes L_Y)i^r_X\iota_ZH
\\
&=\cc{f,\cc{e,g}}_{R}
 -\tau_{(132)}\Big((\Id_{\Omega}\otimes i_X)L_Z\iota_Y H-(L_Z\otimes\Id_{\Omega})i_X\iota_Y\! H\Big) + (\Id_{\Omega}\otimes L_Y)i_X\iota_Z H 
\\
&=\cc{f,\cc{e,g}}_{R}
 -\tau_{(132)}\Big((\Id_{\Omega}\otimes i_X)L_Z\iota_Y\! H\! +\!((\du{}i_Z+i_Z\du{}) \otimes\Id_{\Omega})i_X\iota_Y\! H\Big) 
 + (\Id_{\Omega}\otimes L_Y)i_X\iota_ZH
 \\
&=\cc{f,\cc{e,g}}_{R}
+ (\Id_{\Omega}\!\otimes\! L_Y\! )i_X\iota_ZH
 \\&
\quad
+\tau_{(132)}\!\((\du{}i_Z\!\otimes\!\Id_A)i^l_X\iota_Y\! H\!+\!(i_Z\!\du{}\!\otimes\!\Id_{\Omega})i_X\iota_Y\! H\!-\!(\Id_{\Omega}\!\otimes\! i_X\!)L_Z\iota_Y\! H\).
%
\end{align*}
Using the above calculations and the identity~\eqref{eq:CD.g} for $\cc{\?,\?}$, one can now check that many cancellations occur, so that \eqref{eq:Jacobi-doble-twist-inicio} reduces to the identity
\begin{equation*}
\begin{aligned}
0&=\tau_{(132)}\big((i_Z\otimes\du{})(\Id_{\Omega^1}\otimes\du{})i^{l}_X\iota_YH+(\du{}i_Z\otimes\Id_A)i^l_X\iota_YH\big)
\\
&\qquad-(\du{}i_X\otimes \Id_A)i^l_Y\iota_ZH -(i_X\otimes\Id_{\Omega^1})(\Id_{\Omega^1}\otimes \du{})i^l_Y\iota_ZH.
\end{aligned}
\label{eq:cancelacion-twist-final-claim}
\end{equation*}
This identity is a consequence of \eqref{eq:prop:Cartan-identities.1.1.a} and the identity $i^r_Z\iota_YH+\sigma_{(12)}i^l_Y\iota_ZH=0$ (that is the projection of \eqref{eq:identidades-cruciales-reducida.a} onto $A\otimes\Omega^1_R A$):
\begin{align*}
0&=\tau_{(132)}\big(\du{}\lr{i_Z,i_X}^\sim_l\iota_YH\big)
\\
&=\tau_{(132)}\Big(\du{}\big((i_Z\otimes \Id_A)i^l_X\iota_YH+(\Id_A\otimes i_X)i^r_Z\iota_YH\big)\Big)
\\
&=\tau_{(132)}\big((\du{}i_Z\otimes\Id_A)i^l_X\iota_YH+(i_Z\otimes\Id_{\Omega^1})(\Id_{\Omega^1}\otimes\du{})i^l_X\iota_YH
 \\
 &\qquad +(\Id_A\otimes\du{}i_X)i^r_Z\iota_YH+(\du{}\otimes i_X)i^r_Z\iota_YH\big)
\\
&=\tau_{(132)}\big((\du{}i_Z\otimes\Id_A)i^l_X\iota_YH+(i_Z\otimes\Id_{\Omega^1})(\Id_{\Omega^1}\otimes\du{})i^l_X\iota_YH
\\
 &\qquad -(\Id_A\otimes\du{}i_X)(i^l_Y\iota_ZH)^\sigma-(\du{}\otimes i_X)(i^l_Y\iota_ZH)^\sigma\big)
 \\
 &=\tau_{(132)}\big( (\du{}i_Z\otimes\Id_A)i^l_X\iota_YH+\big(i_Z\otimes\Id_{\Omega^1})(\Id_{\Omega^1}\otimes\du{})i^{l}_X\iota_YH\big)
 \\
 &\qquad -(\du{}i_X\otimes\Id_A)i^l_Y\iota_ZH-(i_X\otimes\Id_{\Omega^1})(\Id_{\Omega^1}\otimes\du{})i^l_Y\iota_ZH.
\end{align*}

Finally, the identity~\eqref{eq:CD.g} for $\cc{\?,\?}_H$ is
\begin{equation}
\ii{X+\alpha,\partial\ii{Y+\beta,Z+\gamma}}_L=\ii{\cc{X+\alpha,Y+\beta}_H,Z+\gamma}_L+\ii{Y+\beta,\cc{X+\alpha,Z+\gamma}_H}_L.
\label{eq:56g-twist}
\end{equation}
By \eqref{eq:extension-pairing-first-second} and the identity \eqref{eq:CD.g} for $\cc{-,-}$, the right-hand side is
\begin{align*}
&\ii{\cc{X+\alpha,Y+\beta}_H,Z+\gamma}_L+\ii{Y+\beta,\cc{X+\alpha,Z+\gamma}_H}_L
\\
&=\ii{\cc{X\!+\!\alpha,Y\!+\!\beta},Z\!+\!\gamma}_L\!+\!\ii{i_X\iota_YH,Z\!+\!\gamma}_L\!+\!\ii{Y\!+\!\beta,\cc{X\!+\!\alpha,Z\!+\!\gamma}}_R\!+\!\ii{Y\!+\!\beta,i_X\iota_ZH}_R
\\
&=\ii{X+\alpha,\partial\ii{Y+\beta,Z+\gamma}}_L+\sigma_{(132)}\big((i_Z\otimes\Id_A)i^l_X\iota_YH\big)+(\Id_A\otimes i_Y)i^r_X\iota_ZH,
\end{align*}
so~\eqref{eq:56g-twist} is equivalent to the vanishing of $\sigma_{(132)}\big((i_Z\otimes\Id_A)i^l_X\iota_YH\big)+(\Id_A\otimes i_Y)i^r_X\iota_ZH$.
This follows from \eqref{eq:prop:Cartan-identities.1.1.a}, \eqref{eq:prop:Cartan-identities.1.1.b} and \eqref{eq:identidades-cruciales-reducida.a} projected onto $\Omega^1_R A\otimes A$:
\begin{align*}
0&=\sigma_{(132)}\big(\lr{i_Z,i_X}^\sim_l\iota_YH \big)+ \lr{i_Y,i_X}^\sim_r\iota_ZH
\\
&=\sigma_{(132)}\big((i_Z\otimes\Id_A)i^l_X\iota_YH+(\Id_A\otimes i_X)i^r_Z\iota_YH\big)+(\Id_A\otimes i_Y)i^r_X\iota_ZH + (i_X\otimes\Id_A)i^l_Y\iota_ZH
\\
&=\sigma_{(132)}\big((i_Z\otimes\Id_A)i^l_X\iota_YH\!-\!(\Id_A\otimes i_X)\big(i^l_Y\iota_ZH\big)^{\sigma}\big)\!+\!(\Id_A\otimes i_Y)i^r_X\iota_ZH \!+\! (i_X\otimes\Id_A)i^l_Y\iota_ZH
\\
&=\sigma_{(132)}\big((i_Z\otimes\Id_A)i^l_X\iota_YH\big)-( i_X\otimes\Id_A)i^l_Y\iota_ZH\big)+(\Id_A\otimes i_Y)i^r_X\iota_ZH +  (i_X\otimes\Id_A)i^l_Y\iota_ZH
\\
&=\sigma_{(132)}\big((i_Z\otimes\Id_A)i^l_X\iota_YH\big) +(\Id_A\otimes i_Y)i^r_X\iota_ZH.
&\qedhere\end{align*}
\end{proof}

\begin{remark}\label{rem:exact-double-Courant-algbd.1}
As in the geometric situation described by \v Severa~\cite{Sev00}, one can classify exact double Courant algebroids with a fixed underlying \emph{smooth} $R$-algebra.
In addition to Theorems~\ref{thm:standard-DCA} and~\ref{thm:twist-CD-algebra-double}, this requires explicit formulae for the differential calculus of~\secref{sec:diff-calculus} for double derivations over smooth algebras. Details about this calculus and the corresponding classification will appear in a future article.
\end{remark}

\section{The equivalence theorem}
\label{sec:equivalence}

In this section we prove a noncommutative version of \cite[Proposition 4.3]{Hel09} in the case of Courant algebroids:

\begin{theorem}\label{theorem-CD-DPVA}
There exists a one-to-one correspondence between graded double Poisson vertex algebras 
of degree $-1$, freely generated in weights 0 and 1, and double Courant--Dorfman algebras.
\end{theorem}

In Section \secref{V-CD}, we construct a double Courant--Dorfman algebra associated with a double Poisson vertex algebra.
In Section \secref{sec:DCD-DPVA}, we provide the inverse construction. 

\subsection{From double Poisson vertex algebras to double Courant--Dorfman algebras}
\label{V-CD}
\allowdisplaybreaks

We start showing that a graded double Poisson vertex algebra $(\mathcal{V},\partial, \lr{-{}_\lambda-})$ of degree $-1$ (see Definition~\ref{graded-DPVA}) determines canonically a double Courant--Dorfman algebra $(A,E,\ii{-,-},\partial,\cc{-,-})$ (see Definition \ref{CD-Definition}).
Firstly, the fact that $\mathcal{V}$ is $\NN$-graded as an algebra means the multiplication $\mathcal{V}\otimes\mathcal{V}\to\mathcal{V}$ has weight $0$, so it induces a structure of associative (not necessarily commutative) algebra on $A\defeq\mathcal{V}_0$ and a structure of $A$-bimodule on $E:=\mathcal{V}_1$. 
To analyze the remaining algebraic structures of $\mathcal{V}$, given by a derivation $\partial$ and a double $\lambda$-bracket $\lr{-{}_\lambda-}$, from now on we fix $a,b\in A$ and $e,f,g\in E$.
These algebraic structures have the following restrictions to $A$ and $E$:
\begin{enumerate}[label=(\roman{*}), ref=(\roman{*})]
\item 
\label{item:lambda2} 
The derivation $\partial\colon\mathcal{V}\to\mathcal{V}$ has weight $1$, so it restricts to a derivation of $A$ into $E$:
\begin{equation}\label{eq:item:lambda2} 
\partial\colon A\lto E.
\end{equation}
\item 
As the double $\lambda$-bracket has weight $-1$, its restriction to $E\otimes E$ admits an expansion
 \begin{equation}\label{eq:ansatz}
\lr{e{}_\lambda f}=\cc{e,f}+\lambda\ii{e,f},
\end{equation}
for unique linear maps
\begin{subequations}\label{eq:ansatz.2}
\begin{align}\label{eq:ansatz.2.a}
\cc{-,-}\,\colon E\otimes E &\lto E\otimes A\oplus A\otimes E,
\\
\label{eq:ansatz.2.b}
\ii{-,-}\,\colon E\otimes E &\lto A\otimes A.
\end{align}
\end{subequations}
This follows simply because $((\mathcal{V}\otimes\mathcal{V})[\lambda])_1=E\otimes A\oplus A\otimes E\oplus \lambda(A\otimes A)$.
\item 
\label{item:lambda3} 
Since $\lr{-{}_{\lambda}-}$, $e$ and $a$ have weights $-1$, $1$ and $0$, respectively, $\lr{-{}_\lambda-}$ restricts to 
\begin{subequations}\label{eq:bracket-A-f-A}\begin{align}\label{bracket-A-f}&
E\otimes A\lto A\otimes A,\quad  e\otimes a\longmapsto\lr{e{}_\lambda a},
\\&
\label{lambda-bracket-f-A}
A\otimes E\lto A\otimes A,\quad  e\otimes a\longmapsto \lr{a{}_\lambda e}.
\end{align}\end{subequations}
\item 
\label{item:lambda5} 
Since $\lr{a{}_\lambda b}$ has weight $-1$ and $\mathcal{V}_{k}=0$ for $k<0$, 
\begin{equation}\label{bracket-f-g}
\lr{a{}_\lambda b}=0.
\end{equation}
\end{enumerate}

We will show now that the axioms of double Poisson vertex algebras (Definition \ref{DPVA}) applied to the triple $(\mathcal{V},\partial, \lr{-{}_\lambda-})$ imply the axioms of double Courant--Dorfman algebras for the 5-tuple $(A,E,\ii{-,-},\partial,\cc{-,-})$ (Definition~\ref{CD-Definition}).

We start with~\eqref{vertex-skew} applied to elements of $E$, namely
\begin{equation}\label{eq:skew-inicial-5.3a}
\lr{e{}_\lambda f}=-\lr{f{}_{-\lambda-\partial}e}^{\sigma}.
\end{equation}
By~\eqref{eq:ansatz}, the right-hand side of this identity is
\begin{equation}
-\lr{f{}_{-\lambda-\partial}e}^{\sigma}=-\cc{f,e}^\sigma+\lambda\ii{f,e}^\sigma+\big(\partial\ii{f,e}\big)^\sigma.
\label{eq:skew-inicial-5.3a-RHS}
\end{equation}
Identifying the coefficients of $\lambda$ appearing in~\eqref{eq:skew-inicial-5.3a} (using~\eqref{eq:ansatz} and~\eqref{eq:skew-inicial-5.3a-RHS}), we obtain
\begin{equation}
\ii{e,f}=\ii{f,e}^\sigma,
\label{eq:ii-symmetric}
\end{equation}
that is, the pairing $\ii{-,-}$ is symmetric. 
Since the derivation $\partial$ acts on tensor products by the Leibniz rule (see \eqref{eq:partialLeibniz-rule}), \eqref{eq:ii-symmetric} implies
\begin{equation}
\partial \ii{e,f}=\big(\partial \ii{f,e}\big)^\sigma,
\label{eq:ii-symmetric-con-partial}
\end{equation}
so~\eqref{eq:CD.a} follows by identifying the terms without $\lambda$ in~\eqref{eq:skew-inicial-5.3a} (again using~\eqref{eq:ansatz} and~\eqref{eq:skew-inicial-5.3a-RHS}). 

Next, we start with the sesquilinearity conditions~\eqref{eq:sesquilinearity-vertex} of the double $\lambda$-bracket, that is,
\begin{subequations}
\begin{align}\label{eq:sesquilin-5.3b-1}
\cc{\partial a,e}+\lambda\ii{\partial a,e}
&=\lr{\partial a{}_\lambda e}
=-\lambda\lr{a{}_\lambda e},
\\
\cc{e,\partial a}+\lambda\ii{e,\partial a}
&=\lr{e{}_\lambda\partial a}
=(\lambda+\partial)\lr{e{}_\lambda a},
\end{align}\end{subequations}
where the left-hand identities follow from~\eqref{eq:ansatz}.
Since $\lr{a{}_\lambda e},\lr{e{}_\lambda a}$ do not depend on $\lambda$ by~\eqref{eq:bracket-A-f-A}, equating coefficients of these polynomial equations in the variable $\lambda$, we obtain 
\begin{subequations}\label{eq:Axiomas}
\begin{align}\label{Axioma5}
\cc{\partial a,e}&=0,
\\\label{eq:Axiomas.b}
\ii{\partial a,e}&=-\lr{a{}_\lambda e},
\\\label{eq:Axiomas.c}
\cc{e,\partial a}&=\partial\lr{e{}_\lambda a},
\\\label{eq:Axiomas.d}
\ii{e,\partial a}&=\lr{e{}_\lambda a}.
\end{align}\end{subequations}
In particular,~\eqref{Axioma5} is \eqref{eq:CD.b}.
Observe also that~\eqref{eq:Axiomas.b} and~\eqref{eq:Axiomas.d} give closed formulas for~\eqref{eq:bracket-A-f-A} in terms of $\partial$ and $\ii{-,-}$.

To prove \eqref{eq:CD.c}, we apply again the sesquilinearity conditions~\eqref{eq:sesquilinearity-vertex}. By~\eqref{bracket-f-g}, they give
\[
\lr{\partial a{}_\lambda b}=0=\lr{a{}_\lambda\partial b}.
\]
Hence, setting $e=\partial b$ in~\eqref{eq:sesquilin-5.3b-1}, we obtain
\[
\cc{\partial a,\partial b}+\lambda\ii{\partial a,\partial b}
=\lr{\partial a{}_\lambda\partial b}
=-\lambda\lr{a{}_\lambda\partial b}
=0,
\]
and so equating coefficients of this polynomial equation, we obtain \eqref{eq:CD.c} and $\cc{\partial a,\partial b}=0$.

To establish the Leibniz identity \eqref{eq:CD.d}, we start with the corresponding Leibniz identity
\begin{equation}
\lr{e{}_\lambda fa}=\lr{e{}_\lambda f}a+f\lr{e{}_\lambda a}
\label{eq:Leibniz-5.3-d-1}
\end{equation}
for the $\lambda$-bracket (see~\eqref{eq:Leibniz-vertex.a}). By \eqref{eq:ansatz} and~\eqref{eq:Axiomas.d}, the left and the right hand sides are
\begin{gather*}
\lr{e{}_\lambda fa}=\cc{e,fa}+\lambda\ii{e,fa},
\\
\lr{e{}_\lambda f}a+f\lr{e{}_\lambda a}=\cc{e,f}a+\lambda \ii{e,f}a+f\ii{e,\partial a}.
\end{gather*}
Equating coefficients of the polynomial equation~\eqref{eq:Leibniz-5.3-d-1}, we obtain \eqref{eq:CD.d} and moreover
\begin{equation}\label{eq:double-bracket-linear.1}
\ii{e,fa}=\ii{e,f}a.
\end{equation}
By similar arguments, one can see that~\eqref{eq:Leibniz-vertex.a}, 
\eqref{eq:ansatz} and~\eqref{eq:Axiomas.d} imply \eqref{eq:CD.e} and moreover
\begin{equation}\label{eq:double-bracket-linear.2}
\ii{e,af}=a\ii{e,f}.
\end{equation}

We will show now that $\ii{-,-}$ is a pairing, \emph{i.e.},~\eqref{eq:pairing.1} are morphisms of $A$-bimodules. 
First,~\eqref{eq:pairing.1.a} is an $A$-bimodule morphism, by~\eqref{eq:double-bracket-linear.1} and~\eqref{eq:double-bracket-linear.2}.
To prove that~\eqref{eq:pairing.1.b} is a morphism of $A$-bimodules as well, we apply~\eqref{vertex-skew} and \eqref{eq:Leibniz-vertex.a} (with $\partial$ replaced by $-\lambda-\partial$):
\begin{equation}\label{eq:proof.pairing.1}
\lr{ae{}_{\lambda}f}=-\lr{f{}_{-\lambda-\partial}ae}^{\sigma}=-\big(a\lr{f{}_{-\lambda-\partial}e}+\lr{f{}_{-\lambda-\partial}a}e\big)^{\sigma}.
\end{equation}
By \eqref{eq:ansatz}, the left-hand side is $\cc{ae,f}+\lambda\ii{ae,f}$, while by~\eqref{eq:ansatz}, ~\eqref{eq:CD.a}, \eqref{eq:ii-symmetric}, \eqref{eq:ii-symmetric-con-partial} and~\eqref{eq:Axiomas.d}, the right-hand side is
\begin{align*}
-a*&\cc{f,e}^\sigma+\lambda\, a*\ii{f,e}^\sigma+a*\(\partial\ii{f,e}\)^\sigma-\ii{f,\partial a}^\sigma *e
\\&
=a*\cc{e,f}+\lambda\, a*\ii{e,f}-\ii{\partial a,f}*e.
\end{align*}
Hence identifying the coefficients of $\lambda$ appearing in~\eqref{eq:proof.pairing.1}, we obtain 
\begin{equation}\label{eq:proof.pairing.2}
\ii{ae,f}=a*\ii{e,f}.
\end{equation}
By similar arguments, one can show that $\ii{ea,f}=\ii{e,f}*a$.
This identity and~\eqref{eq:proof.pairing.2} mean that~\eqref{eq:pairing.1.b} is an $A$-bimodule morphism, as required, so $\ii{-,-}$ is a pairing. 
In fact, this pairing is symmetric, by~\eqref{eq:ii-symmetric}.

To conclude that $(A,E,\ii{-,-},\partial,\cc{-,-})$ is a double Courant--Dorfman algebra, it remains to prove~\eqref{eq:CD.f} and~\eqref{eq:CD.g}.
We start with the Jacobi identity~\eqref{Jacobi-vertex} applied to $e,f,g\in E$,
\begin{equation}\label{Jacobi-inicio-DPV-CD}
0=\lr{e{}_\lambda\lr{f{}_\mu g}}_L-\lr{f{}_\mu\lr{e{}_\lambda g}}_R-\lr{\lr{e{}_{\lambda} f}{}_{\lambda+\mu}g}_L,
\end{equation}
and calculate each term separately applying~\eqref{eq:extension-pairing-first-second}, \eqref{eq:ext-pairing-zero}, \eqref{eq:extension-vertex-double-bracket}, \eqref{eq:extension-CD-bracket}, \eqref{eq:partialLeibniz-rule}, \eqref{eq:decomposition-Courant-Dorfman-def}, \eqref{eq:decomposition-Courant-Dorfman-def-1} and the identities~\eqref{eq:ansatz}, \eqref{bracket-f-g}, \eqref{eq:Axiomas.b}, \eqref{eq:Axiomas.d}:
\begin{subequations}\label{eq:Jacobi-CD-DPVA-general}
\begin{align}\notag
&\lr{e{}_\lambda\lr{f{}_\mu g}}_L
=\lr{e{}_{\lambda}\big(\cc{f,g}+\ii{f,g}\mu\big)}_L
\\\notag
&\quad=\lr{e{}_{\lambda}\cc{f,g}^{\prime}_l}\otimes\cc{f,g}^{\pprime}_l
+\lr{e{}_{\lambda}\cc{f,g}^{\prime}_r}\otimes\cc{f,g}^{\pprime}_r
+\lr{e{}_{\lambda}\ii{f,g}^{\prime}}\otimes \ii{f,g}^{\pprime}\,\mu
\\\notag
&\quad=\big(\cc{e,\cc{f,g}^{\prime}_l}+\ii{e,\cc{f,g}^{\prime}_l}\lambda\big)\!\otimes\!\cc{f,g}^{\pprime}_l
+\ii{e,\partial \cc{f,g}^{\prime}_r}\!\otimes\!\cc{f,g}^{\pprime}_r
+\ii{e,\partial \ii{f,g}^{\prime}}\!\otimes\!\ii{f,g}^{\pprime}\mu
\\\notag
&\quad=\cc{e,\cc{f,g}_l}_L+\cc{e,\cc{f,g}_r}_L+\ii{e,\cc{f,g}}_L\,\lambda+\ii{e,\partial\ii{f,g}}_L\,\mu
\\\label{eq:Jacobi-CD-DPVA-general-1}
&\quad=\cc{e,\cc{f,g}}_L+\ii{e,\cc{f,g}}_L\,\lambda+\ii{e,\partial\ii{f,g}}_L\,\mu,
\\\notag
&\lr{f{}_\mu\lr{e{}_\lambda g}}_R
=\lr{f{}_\mu\big(\cc{e,g}+\ii{e,g}\lambda\big)}_R
\\\notag
&\quad= \cc{e,g}^{\prime}_l\otimes \lr{f{}_\mu\cc{e,g}^{\pprime}_l}
+\cc{e,g}^{\prime}_r\otimes \lr{f{}_\mu\cc{e,g}^{\pprime}_r}
+ \ii{e,g}^{\prime}\otimes\lr{f{}_\mu\ii{e,g}^{\pprime}}\,\lambda
\\\notag
&\quad=\cc{e,g}^{\prime}_l\!\otimes\! \ii{f,\partial\cc{e,g}^{\pprime}_l}
+\cc{e,g}^{\prime}_r\!\otimes\!\big(\cc{f,\cc{e,g}^{\pprime}_r}+\ii{f,\cc{e,g}^{\pprime}_r}\mu\big)+\ii{e,g}^{\prime}\!\otimes\!\ii{f,\partial\ii{e,g}^{\pprime}}\lambda
\\\notag
&\quad=\cc{f,\cc{e,g}_l}_R+\cc{f,\cc{e,g}_r}_R+\ii{f,\cc{e,g}}_R\,\mu+\ii{f,\partial\ii{e,g}}_R\,\lambda
\\&\label{eq:Jacobi-CD-DPVA-general-2}\quad
=\cc{f,\cc{e,g}}_R+\ii{f,\partial\ii{e,g}}_R\,\lambda+\ii{f,\cc{e,g}}_R\,\mu,
\\\notag
&\lr{\lr{e{}_{\lambda} f}{}_{\lambda+\mu}g}_L
=\lr{\big(\cc{e,f}+\ii{e,f}\lambda\big){}_{\lambda+\mu}g}_L
= \lr{\cc{e,f}^{\prime}_l{}_{\lambda+\mu+\partial}g}_{\to}\!\otimes_1\!\cc{e,f}^{\pprime}_l
\\\notag&\qquad
+\lr{\cc{e,f}^{\prime}_r{}_{\lambda+\mu+\partial}g}_{\to}\!\otimes_1\!\cc{e,f}^{\pprime}_r
+\lr{\ii{e,f}^{\prime}{}_{\lambda+\mu+\partial}g}_{\to}\!\otimes_1\!\ii{e,f}^{\pprime}\,\lambda
\\\notag
&\quad=\cc{\cc{e,f}^\prime_l,g}\otimes_1\cc{e,f}^{\pprime}_l+\ii{\cc{e,f}^\prime_l,g}\otimes_1(\lambda+\mu+\partial)\cc{e,f}^{\pprime}_l
\\\notag
&\qquad -\ii{\partial\cc{e,f}^\prime_r,g}\otimes_1\cc{e,f}^{\pprime}_r
-\ii{\partial\ii{e,f}^\prime,g}\otimes_1\ii{e,f}^{\pprime}\,\lambda
\\\notag
&\quad=\cc{\cc{e,f}^\prime_l,g}\otimes_1\cc{e,f}^{\pprime}_l+\ii{\cc{e,f}^{\prime}_l,g}\otimes_1\partial\cc{e,f}^{\pprime}_l
+\ii{\cc{e,f}_l,g}_L\,(\lambda+\mu)
\\\notag&\qquad
+\cc{\cc{e,f}_r,g}_L
-\ii{\partial\ii{e,f},g}_L\,\lambda
\\\notag&\quad
=\cc{\cc{e,f}_l,g}_L+\ii{\cc{e,f},g}_L\,(\lambda+\mu)
+\cc{\cc{e,f}_r,g}_L
-\ii{\partial\ii{e,f},g}_L\,\lambda
\\\label{eq:Jacobi-CD-DPVA-general-3}&\quad
=\cc{\cc{e,f},g}_L+(\ii{\cc{e,f},g}_L-\ii{\partial\ii{e,f},g}_L)\,\lambda
+\ii{\cc{e,f},g}_L\,\mu.
\end{align}\end{subequations}
Putting the calculations~\eqref{eq:Jacobi-CD-DPVA-general-1},~\eqref{eq:Jacobi-CD-DPVA-general-2} and~\eqref{eq:Jacobi-CD-DPVA-general-3} together into~\eqref{Jacobi-inicio-DPV-CD}, we get
\begin{subequations}
\label{eq:Jacobi-CD-DPVA-general-expression-final}
\begin{align}
0&=\lr{e{}_\lambda\lr{f{}_\mu g}}_L-\lr{f{}_\mu\lr{e{}_\lambda g}}_R-\lr{\lr{e{}_{\lambda} f}{}_{\lambda+\mu}g}_L\nonumber
\\\label{eq:Jacobi-CD-DPVA-general-expression-final.a}
&=\Big(\cc{e,\cc{f,g}}_L-\cc{\cc{e,f},g}_L-\cc{f,\cc{e,g}}_R\Big) 
\\\label{eq:Jacobi-CD-DPVA-general-expression-final.b}
&\quad+\Big(\ii{e,\partial\ii{f,g}}_L-\ii{f,\cc{e,g}}_R-\ii{\cc{e,f},g}_L\Big)\,\mu
\\\label{eq:Jacobi-CD-DPVA-general-expression-final.c}
&\quad+\Big(\ii{e,\cc{f,g}}_L-\ii{f,\partial\ii{e,g}}_R-\ii{\cc{e,f},g}_L+\ii{\partial\ii{e,f},g}_L\Big)\,\lambda . 
\end{align}
\end{subequations}
Equating coefficients of this polynomial equation,~\eqref{eq:CD.f} and~\eqref{eq:CD.g} follow from the vanishings of~\eqref{eq:Jacobi-CD-DPVA-general-expression-final.a} and~\eqref{eq:Jacobi-CD-DPVA-general-expression-final.b}, respectively. 
Note that the vanishing of \eqref{eq:Jacobi-CD-DPVA-general-expression-final.c} is equivalent to~\eqref{eq:CD-identities.1.f}, which follows from the axioms~\eqref{eq:CD} anyway, as shown in Lemma~\ref{lem:CD-identities}.

In conclusion, we have shown that $(A,E,\ii{-,-},\partial,\cc{-,-})$ is a double Courant--Dorfman algebra. 
Note that this construction does not require the condition that the double Poisson vertex algebra is freely generated in weights $0$ and $1$. This condition will appear in~\secref{sec:DCD-DPVA}.

\subsection{From double Courant--Dorfman algebras to double Poisson vertex algebras}
\label{sec:DCD-DPVA}
\allowdisplaybreaks

Conversely, given a double Courant--Dorfman algebra $(A,E,\ii{-,-},\partial,\cc{-,-})$ (see Definition \ref{CD-Definition}), we give now a canonical construction of a graded double Poisson vertex algebra $(\mathcal{V},\partial,\lr{-{}_{\lambda}-})$ of weight (or degree) $-1$, freely generated in weights $0$ and $1$.
By definition, its underlying differential graded algebra $(\mathcal{V},\partial)$ is the
graded algebra with a derivation of weight $1$, freely generated as a differential algebra in degrees $0$ and $1$, such that $\mathcal{V}_0=A$, $\mathcal{V}_1=E$ and the derivation $\partial\colon\mathcal{V}\to\mathcal{V}$ restricts in degree 0 to $\partial\colon A\to E$.
More explicitly, using the free graded $A$-algebra $(\mathcal{V}',T)$ with derivation $T$ of degree 1 generated by the $A$-bimodule $E$ in degree 1, the graded algebra 
$\mathcal{V}$ and the derivation $\partial\colon\mathcal{V}\to\mathcal{V}$ are constructed as the quotient of $\mathcal{V}'$ by the ideal generated by $(T-\partial)a$ for all $a\in A$, and the map obtained by reduction from $T$ on $\mathcal{V}'$, respectively.
We also define the operator 
\begin{equation}\label{eq:vertex-bracket.1}
\lr{-{}_{\lambda}-}\colon\,\mathcal{V}\otimes\mathcal{V}\lto (\mathcal{V}\otimes\mathcal{V})[\lambda] 
\end{equation}
as the linear map given for all $a,b\in A$ and $e,f\in E$ by the formulas
\begin{subequations}
\label{eq:vertex-bracket}
\begin{align}\label{eq:vertex-bracket.a}
\lr{e {}_\lambda f}&=\cc{e,f}+\ii{e,f}\lambda, 		
\\\label{eq:vertex-bracket.b} 
\lr{e{}_\lambda a}&=\ii{e,\partial a},    	
\\\label{eq:vertex-bracket.c}
\lr{a{}_\lambda e}&=-\ii{\partial a,e},  		
\\\label{eq:vertex-bracket.d}
\lr{a{}_\lambda b}&=0,		
\end{align}
\end{subequations}
and extended to all of $\mathcal{V}$ using the Leibniz rules \eqref{eq:Leibniz-vertex} and the sesquilinearity identities~\eqref{eq:sesquilinearity-vertex}.
%

We will now show that the triple $(\mathcal{V},\partial,\lr{-{}_\lambda-})$ is a graded double Poisson vertex algebra of degree $-1$, that is, \eqref{eq:vertex-bracket.1} satisfies the identities~\eqref{eq:sesquilinearity-vertex}, \eqref{eq:Leibniz-vertex}, \eqref{vertex-skew} and \eqref{Jacobi-vertex} and, moreover, $\partial$ and $\lr{-{}_\lambda-}$ have weights $1$ and $-1$, respectively (see Definition~\ref{graded-DPVA}).
%
The conditions about the weights of $\partial$ and $\lr{-{}_\lambda-}$ are satisfied by construction.
Furthermore, since \eqref{eq:vertex-bracket.1} is extended to all of $\mathcal{V}$ using the Leibniz rules \eqref{eq:Leibniz-vertex} and the sesquilinearity identities~\eqref{eq:sesquilinearity-vertex}, it suffices to check~\eqref{eq:sesquilinearity-vertex}, \eqref{eq:Leibniz-vertex}, \eqref{vertex-skew} and \eqref{Jacobi-vertex} on generators, that is, elements of $A\oplus E$, so from now on we fix $a,b\in A$ and $e,f,g\in E$. 

We start checking the sesquilinearity identities~\eqref{eq:sesquilinearity-vertex}.
The identity \eqref{eq:sesquilinearity-vertex.a} holds on $A\otimes A$ because the left and the right hand sides $\lr{\partial a{}_\lambda b}$ and $\lr{a{}_{\lambda}b}$ are both $0$ (by \eqref{eq:vertex-bracket.b}, \eqref{eq:vertex-bracket.d} and \eqref{eq:CD.c}),
%
and on $E\otimes A$ and $E\otimes E$, because in these cases the left-hand sides $\lr{\partial e{}_\lambda a}$ and $\lr{\partial e{}_\lambda f}$ are defined using precisely \eqref{eq:sesquilinearity-vertex.a}, and applying then~\eqref{eq:vertex-bracket.b} and~\eqref{eq:vertex-bracket.a}, respectively.
The identity \eqref{eq:sesquilinearity-vertex.b} holds on $A\otimes A$, $A\otimes E$ and $E\otimes E$ by similar reasons (using~\eqref{eq:vertex-bracket.c} rather than~\eqref{eq:vertex-bracket.b} in the first case, and applying \eqref{eq:sesquilinearity-vertex.b} rather than \eqref{eq:sesquilinearity-vertex.a} in the last two cases).
The identity \eqref{eq:sesquilinearity-vertex.a} holds on $A\otimes E$ because~\eqref{eq:vertex-bracket.a} and~\eqref{eq:CD.b} imply 
\[
\lr{\partial a{}_\lambda e}=\cc{\partial a,e}+\lambda \ii{\partial a, e}=\lambda \ii{\partial a,e}=-\lambda \lr{a {}_{\lambda} e}.
\]
Finally, the identity \eqref{eq:sesquilinearity-vertex.b} holds on $E\otimes A$ because using \eqref{eq:CD.a}, \eqref{eq:CD.b} and~\eqref{eq:ii-symmetric-con-partial}, we obtain
\[
\cc{e,\partial a}=\(\partial\ii{\partial a,e}\)^\sigma-\cc{\partial a,e}^\sigma
=\(\partial\ii{\partial a,e}\)^\sigma
=\partial\ii{e,\partial a}
\]
and so~\eqref{eq:vertex-bracket.a} and~\eqref{eq:vertex-bracket.b} imply
\[
\lr{e{}_\lambda\partial a}=\cc{e,\partial a}+\lambda\ii{\partial a,e}
=(\lambda+\partial)\ii{\partial a,e}=(\lambda+\partial)\lr{e{}_\lambda a}.
\]

We prove now the left Leibniz rule \eqref{eq:Leibniz-vertex.a}.
This rule follows trivially when the three elements are in $A$, as the left and the right hand sides of \eqref{eq:Leibniz-vertex.a} vanish in this case by~\eqref{eq:vertex-bracket.d}.
Using \eqref{eq:pairing.1.a}, \eqref{eq:vertex-bracket.b}, \eqref{eq:vertex-bracket.c}, \eqref{eq:vertex-bracket.d} and the fact that $\partial$ is a derivation, we obtain
\begin{align*}
\lr{a{}_\lambda eb}&=-\ii{\partial a,eb}=-\ii{\partial a,b}b=\lr{a{}_\lambda e}b=\lr{a{}_\lambda e}b+e\lr{a{}_\lambda b},
\\
\lr{a{}_\lambda be}&=-\ii{\partial a,be}=-b\ii{\partial a,e}=b\lr{a{}_\lambda e}=b\lr{a{}_\lambda e}+\lr{a{}_\lambda b}e,
\\
\lr{e{}_\lambda ab}&=\ii{e,\partial(ab)}=\ii{e,(\partial a)b+a(\partial b)}
=\ii{e,\partial a}b+a\ii{e,\partial b}
\\
&=\lr{e{}_\lambda a}b+a\lr{e{}_\lambda b}.
\end{align*}
Using \eqref{eq:vertex-bracket.a}, \eqref{eq:pairing.1.a}, \eqref{eq:CD.d} and \eqref{eq:vertex-bracket.b}, we also have
\begin{align*}
\lr{e{}_{\lambda}fa}
&=\cc{e,fa}+\lambda\ii{e,fa}
=\cc{e,f}a+f\ii{e,\partial a}+\lambda\ii{e,f}a
\\&
=\(\cc{e,f}+\lambda\ii{e,f}\)a+f\lr{e{}_{\lambda}a}
=\lr{e{}_{\lambda}f}a+f\lr{e{}_{\lambda}a},
\end{align*}
and a similar calculation (using \eqref{eq:CD.e}, rather than \eqref{eq:CD.d}), gives $\lr{e{}_{\lambda}af}=a\lr{e{}_{\lambda}f}+\lr{e{}_{\lambda}a}f$. 
Consequently, $\lr{-{}_{\lambda}-}$ satisfies the left Leibniz rule \eqref{eq:Leibniz-vertex.a}.
The proof of the right Leibniz rule \eqref{eq:Leibniz-vertex.b} is similar.
Alternatively, as explained after Definition~\ref{DPVA}, it follows from the left Leibniz rule and the skewsymmetry identity \eqref{vertex-skew}, that we prove now.

The skewsymmetry identity \eqref{vertex-skew} follows because
$\ii{-,-}$ is symmetric, and $\lr{e{}_{\lambda}a}$ and $\lr{a{}_{\lambda}e}$ have degree $0$ as polynomials in $\lambda$, so \eqref{eq:vertex-bracket.b} and \eqref{eq:vertex-bracket.c} imply
\begin{align*}
\lr{e{}_{\lambda}a}&=\ii{e,\partial a}=\ii{\partial a,e}^{\sigma}=-\lr{a{}_{-\lambda-\partial}e}^{\sigma},
\\
\lr{a{}_{\lambda}e}&=-\ii{\partial a,e}=-\ii{e,\partial a}^\sigma=-\lr{e{}_{-\lambda-\partial}a}^{\sigma},
\end{align*}
while~\eqref{eq:CD.a},~\eqref{eq:ii-symmetric-con-partial} and~\eqref{eq:vertex-bracket.a} imply 
\begin{align*}
\lr{e{}_{\lambda}f}&=\cc{e,f}+\lambda\ii{e,f}
=-\cc{f,e}^{\sigma}+\lambda\ii{e,f}+\partial\ii{e,f}
\\&
=-\big(\cc{f,e}-(\lambda+\partial)\ii{f,e}\big)^{\sigma}
=-\lr{f{}_{-\lambda-\partial}e}^{\sigma}.
\end{align*}

To conclude that $(\mathcal{V},\partial,\lr{-{}_\lambda-})$ is a double Poisson vertex algebra, it remains to prove the Jacobi identity \eqref{Jacobi-vertex} on generators of $\mathcal{V}$. 
This follows trivially when two or three elements in \eqref{Jacobi-vertex} belong to $A$, by \eqref{eq:vertex-bracket.d}.
To prove the remaining Jacobi identities
\begin{subequations}\label{eq:DCD-DPVA.Jacobi.1}
\begin{align}\label{eq:DCD-DPVA.Jacobi.1.a}
\lr{e{}_{\lambda}\lr{f{}_{\mu}a}}_L&=\lr{f{}_{\mu}\lr{e{}_{\lambda}a}}_R+\lr{\lr{e{}_{\lambda}f}{}_{\lambda+\mu}a}_L,
\\\label{eq:DCD-DPVA.Jacobi.1.b}
\lr{e{}_{\lambda}\lr{a{}_{\mu}f}}_L&=\lr{a{}_{\mu}\lr{e{}_{\lambda}f}}_R+\lr{\lr{e{}_{\lambda}a}{}_{\lambda+\mu}f}_L,
\\\label{eq:DCD-DPVA.Jacobi.1.c}
\lr{a{}_{\lambda}\lr{e{}_{\mu}f}}_L&=\lr{e{}_{\mu}\lr{a{}_{\lambda}f}}_R+\lr{\lr{a{}_{\lambda}e}{}_{\lambda+\mu}f}_L,
\\\label{eq:DCD-DPVA.Jacobi.1.d}
\lr{e{}_\lambda \lr{f{}_\mu g}}_L&=\lr{f{}_\mu\lr{e{}_\lambda g}}_R+\lr{\lr{e{}_\lambda f}{}_{\lambda+\mu}g}_L,
\end{align}
\end{subequations}
we calculate each term separately.
Regarding~\eqref{eq:DCD-DPVA.Jacobi.1.a}, using \eqref{eq:extension-pairing-first-second}, \eqref{eq:ext-pairing-zero}, \eqref{eq:extension-vertex-double-bracket}, \eqref{eq:decomposition-Courant-Dorfman-def}, \eqref{eq:decomposition-Courant-Dorfman-def-1}, \eqref{eq:vertex-bracket.a}, \eqref{eq:vertex-bracket.b} and \eqref{eq:vertex-bracket.d}, it is straightforward to see that 
\begin{align*}
\lr{e{}_{\lambda}\lr{f{}_{\mu}a}}_L&=
\lr{e{}_{\lambda}\ii{f,\partial a}}_L
=\lr{e{}_{\lambda}\ii{f,\partial a}^{\prime}}\otimes \ii{f,\partial a}^{\pprime}
=\ii{e,\partial\ii{f,\partial a}}_L,
\\
\lr{f{}_{\mu}\lr{e{}_{\lambda}a}}_R&=
\lr{f{}_{\mu}\ii{e,\partial a}}_R
=\ii{e,\partial a}^{\prime}\otimes\lr{f{}_{\mu}\ii{e,\partial a}^{\pprime}}
=\ii{f,\partial\ii{e,\partial a}}_R,
\\
\lr{\lr{e{}_{\lambda}f}{}_{\lambda+\mu}a}_L&=
\lr{\big(\cc{e,f}+\ii{e,f}\lambda\big){}_{\lambda+\mu}a}_L
=\lr{\cc{e,f}^{\prime}_l{}_{\lambda+\mu+\partial}a}_{\to}\otimes_1\cc{e,f}^{\pprime}_l
\\&
=\ii{\cc{e,f}^{\prime}_l,\partial a}\otimes_1\cc{e,f}^{\pprime}_l
=\ii{\cc{e,f}_l,\partial a}_L=\ii{\cc{e,f},\partial a}_L,
\end{align*}
so~\eqref{eq:DCD-DPVA.Jacobi.1.a} follows from~\eqref{eq:CD.g} (with $g=\partial a$) and~\eqref{eq:CD-identities.1.a}.
The identity~\eqref{eq:DCD-DPVA.Jacobi.1.b} follows similarly from~\eqref{eq:CD.g} (with $f$ and $g$ replaced by $\partial a$ and $f$, respectively), and also using~\eqref{eq:CD-identities.1.a} and \eqref{eq:vertex-bracket.c}. 
To prove~\eqref{eq:DCD-DPVA.Jacobi.1.c}, we apply \eqref{eq:extension-pairing-first-second}, \eqref{eq:ext-pairing-zero}, \eqref{eq:extension-vertex-double-bracket}, \eqref{eq:decomposition-Courant-Dorfman-def}, \eqref{eq:decomposition-Courant-Dorfman-def-1}, \eqref{eq:CD-identities.1.b}, \eqref{eq:vertex-bracket}, obtaining
\begin{align*}
\lr{a{}_{\lambda}\lr{e{}_{\mu}f}}_L&=\lr{a{}_{\lambda}\big(\cc{e,f}+\ii{e,f}\mu\big)}_L
=\lr{a{}_{\lambda}\cc{e,f}^{\prime}_l}\otimes\cc{e,f}^{\pprime}_l
\\&
=-\ii{\partial a,\cc{e,f}_l^{\prime}}\otimes\cc{e,f}^{\pprime}_l
=-\ii{\partial a,\cc{e,f}_l}_L=-\ii{\partial a,\cc{e,f}}_L,
\\
\lr{e{}_{\mu}\lr{a{}_{\lambda}f}}_R&=-\lr{e{}_{\mu}\ii{\partial a,f}}_R
=-\ii{\partial a,f}^{\prime}\otimes\lr{e{}_{\mu}\ii{\partial a,f}^{\pprime}}
\\&
=-\ii{\partial a,f}^{\prime}\otimes\ii{e,\partial\ii{\partial a,f}^{\pprime}}
=-\ii{e,\partial\ii{\partial a,f}}_R,
\\
\lr{\lr{a{}_{\lambda}e}{}_{\lambda+\mu}f}_L&=-\lr{\ii{\partial a,e}{}_{\lambda+\mu}f}_L
=-\lr{\ii{\partial a,e}^{\prime}{}_{\lambda+\mu+\partial}f}_{\to}\otimes_1\ii{\partial a,e}^{\pprime}
\\&
=\ii{\partial\ii{\partial a,e}^{\prime},f}\otimes_1 \ii{\partial a,e}^{\pprime}
=\ii{\partial a,e}^{\pprime}\otimes_1\ii{\partial\ii{\partial a,e}^{\prime},f}
\\&
=\ii{\(\partial\ii{\partial a,e}\)^\sigma,f}_R
=\ii{\cc{e,\partial a},f}_R.
\end{align*}
Therefore~\eqref{eq:DCD-DPVA.Jacobi.1.c} is equivalent to the following identity, which follows from~\eqref{eq:CD-identities.1.f}:
\[
\ii{e,\partial\ii{\partial a,f}}_R=\ii{\cc{e,\partial a},f}_R+\ii{\partial a,\cc{e,f}}_L.
\]
To prove~\eqref{eq:DCD-DPVA.Jacobi.1.d}, we calculate each term using
\eqref{eq:extension-pairing-first-second}, 
\eqref{eq:ext-pairing-zero}, 
\eqref{eq:extension-vertex-double-bracket},
\eqref{eq:partialLeibniz-rule}, 
\eqref{eq:decomposition-Courant-Dorfman-def}, 
\eqref{eq:decomposition-Courant-Dorfman-def-1} 
and
\eqref{eq:vertex-bracket}:
\allowdisplaybreaks
\begin{subequations}\label{eq:Jacobi-CD-DPVA-general-bis}
\begin{align}\notag
&\lr{e{}_\lambda\lr{f{}_\mu g}}_L
=\lr{e{}_{\lambda}\big(\cc{f,g}+\ii{f,g}\mu\big)}_L
\\\notag
&\quad
=\lr{e{}_{\lambda}\cc{f,g}^{\prime}_l}\otimes \cc{f,g}^{\pprime}_l
+\lr{e{}_{\lambda}\cc{f,g}^{\prime}_r}\otimes \cc{f,g}^{\pprime}_r
+\lr{e{}_{\lambda}\ii{f,g}^{\prime}}\otimes \ii{f,g}^{\pprime}\mu
\\\notag
&\quad=\big(\cc{e,\cc{f,g}^{\prime}_l}+\ii{e,\cc{f,g}^{\prime}_l}\lambda\big)\!\otimes \!\cc{f,g}^{\pprime}_l
+\ii{e,\partial \cc{f,g}^{\prime}_r}\!\otimes\! \cc{f,g}^{\pprime}_r
+\ii{e,\partial \ii{f,g}^{\prime}}\!\otimes\! \ii{f,g}^{\pprime}\mu
\\\notag
&\quad=\cc{e,\cc{f,g}_l}_L+\ii{e,\cc{f,g}_l}_L\lambda
+\cc{e,\cc{f,g}_r}_L+\ii{e,\partial\ii{f,g}}_L\mu
\\\label{eq:Jacobi-CD-DPVA-general-bis-1}
&\quad=\cc{e,\cc{f,g}}_L+\ii{e,\cc{f,g}}_L\,\lambda+\ii{e,\partial\ii{f,g}}_L\,\mu,
\\\notag
&\lr{f{}_\mu\lr{e{}_\lambda g}}_R
=\lr{f{}_\mu\big(\cc{e,g}+\ii{e,g}\lambda}_R
\\\notag
&\quad=\cc{e,g}^{\prime}_l\otimes \lr{f{}_\mu\cc{e,g}^{\pprime}_l}
+\cc{e,g}^{\prime}_r\otimes \lr{f{}_\mu\cc{e,g}^{\pprime}_r}
+ \ii{e,g}^{\prime}\otimes\lr{f{}_\mu\ii{e,g}^{\pprime}} \lambda
\\\notag
&\quad=\cc{e,g}^{\prime}_l\!\otimes \!\ii{f,\partial \cc{e,g}^{\pprime}_l}
+\cc{e,g}^{\prime}_r\!\otimes\!(\cc{f,\cc{e,g}^{\pprime}_r}+\ii{f,\cc{e,g}^{\pprime}_r}\mu)
+\ii{e,g}^{\prime}\!\otimes\!\ii{f,\partial \ii{e,g}^{\pprime}}\lambda
\\\label{eq:Jacobi-CD-DPVA-general-bis-2}
&\quad=\cc{f,\cc{e,g}}_R+\ii{f,\partial\ii{e,g}}_R\,\lambda+\ii{f,\cc{e,g}}_R\,\mu, 
\\\notag
&\lr{\lr{e{}_{\lambda} f}{}_{\lambda+\mu}g}_L
=\lr{\big(\cc{e,f}+\ii{e,f}\lambda\big){}_{\lambda+\mu}g}_L
\\\notag
&\quad= \lr{\cc{e,f}^\prime_l{}_{\lambda+\mu+\partial}g}_{\to}\otimes_1\cc{e,f}^{\pprime}_l
+\lr{\cc{e,f}^{\prime}_r{}_{\lambda+\mu+\partial}g}_{\to}\otimes_1\cc{e,f}^{\pprime}_r
\\\notag &\qquad 
+\lr{\ii{e,f}^{\prime}{}_{\lambda+\mu+\partial}g}_{\to}\otimes_1\ii{e,f}^{\pprime}\lambda
\\\notag
&\quad=\cc{\cc{e,f}^\prime_l,g}\otimes_1\cc{e,f}^{\pprime}_l
+\ii{\cc{e,f}^\prime_l,g}\otimes_1 (\lambda+\mu+\partial)\cc{e,f}^{\pprime}_l
\\\notag
&\qquad -\ii{\partial\cc{e,f}^{\prime}_r,g}\otimes_1\cc{e,f}^{\pprime}_r
-\ii{\partial\ii{e,f}^{\prime},g}\otimes_1\ii{e,f}^{\pprime}\lambda
\\\notag
&\quad=\cc{\cc{e,f}_l,g}_L
+\ii{\cc{e,f}^\prime_l,g}\otimes_1\cc{e,f}^{\pprime}_l(\lambda+\mu)
+\cc{\cc{e,f}_r,g}_L
-\ii{\partial\ii{e,f},g}_L\lambda
\\\label{eq:Jacobi-CD-DPVA-general-bis-3}
&\quad=\cc{\cc{e,f},g}_L
+\(\ii{\cc{e,f},g}_L-\ii{\partial\ii{e,f},g}_L\)\,\lambda+\ii{\cc{e,f},g}_L\,\mu.
\end{align}\end{subequations}
Since these terms are polynomials in the variables $\lambda$ and $\mu$,~\eqref{eq:DCD-DPVA.Jacobi.1.d} is equivalent to three equations, obtained equating coefficients of this polynomial equation.
Putting together the terms~\eqref{eq:Jacobi-CD-DPVA-general-bis} into~\eqref{eq:DCD-DPVA.Jacobi.1.d} (as in~\eqref{eq:Jacobi-CD-DPVA-general-expression-final}), we see that these equations are precisely~\eqref{eq:CD.f}, 
\eqref{eq:CD-identities.1.f} 
and \eqref{eq:CD.g}, 
and so~\eqref{eq:DCD-DPVA.Jacobi.1.d} follows.

To sum up, given a double Courant--Dorfman algebra, in this subsection we proved that it gives rise via \eqref{eq:vertex-bracket} to a graded double Poisson vertex algebra of weight $-1$, freely generated in weights 0 and 1. Consequently, Theorem \ref{theorem-CD-DPVA} follows.

\section{The Kontsevich--Rosenberg principle}
\label{sec:KR-double-CD}

\subsection{The Kontsevich--Rosenberg principle for double Poisson (vertex) algebras}
\label{sec:KR-DPVA-SECTION}

The Kontsevich--Rosenberg principle \cite{Kon94a,KR00} is a guiding rule used to find noncommutative versions of standard geometric notions. 
It states that a noncommutative geometric structure on a finitely generated associative algebra $A$ should naturally induce the corresponding geometric structure on the representation schemes.
To construct these schemes, throughout~\secref{sec:KR-double-CD}, we fix an $R$-algebra $A$, where the algebra $R$ is semi-simple, given by
\[
R=\bigoplus_{\nu\in I}\kk\nu,
\]
for a complete set of orthogonal idempotents $\nu$ in a finite set $I\subset R$, fix a `dimension vector' $N=(N_\nu)_{\nu\in I}\in\NN^I$, and define $\lvert N\rvert\defeq\sum N_\nu$. 
For each commutative algebra $\cO$, we define $\MM_N(\cO)$ as the $R$-algebra
with underlying algebra $\MM_{\lvert N\rvert\times\lvert N\rvert}(\cO)$ of $\cO$-valued $\lvert N\rvert\times \lvert N\rvert$-matrices, with the $R$-algebra structure homomorphism $R\to\MM_{\lvert N\rvert\times\lvert N\rvert}(\cO)$ mapping $\nu\in I$ to the identity matrix $\mathrm{I}^\nu$ in $\MM_{N_\nu\times N_\nu}(\cO)\subset\MM_{\lvert N\rvert\times\lvert N\rvert}(\cO)$ (here we view an element $x$ of $\MM_{\lvert N\rvert\times\lvert N\rvert}(\cO)$ as a block matrix $(x_{\nu\nu'})_{\nu,\nu'\in I}$ where $x_{\nu\nu'}\in\MM_{N_\nu\times N_{\nu'}}(\cO)$). 
Then the representation scheme $\Rep(A,N)=\Spec A_N$ parametrizes $R$-linear representations of $A$ with dimension vector $N$, that is, the commutative algebra $A_N$ represents the functor
\[
\Rep_N(A)\colon\mathbf{CAlg}\lto\mathbf{Set},\quad \cO\longmapsto \Hom_{R}\!\big(A,\MM_N(\cO)\big),
\]
mapping a (finitely generated) commutative algebra $\cO$ to the set of $R$-algebra homomorphisms from $A$ to $\MM_N(\cO)$ (throughout \secref{sec:KR-double-CD}, all occurring algebras will be finitely generated over the base field $\kk$). The representing property of $A_N$ means there is an adjunction
\begin{equation}\label{eq:KR.adjunction-A_N}
\Hom(A_N,\cO)\cong\Hom_R\!\big(A,\MM_N(\cO)\big),
\end{equation}
therefore $A_N$ can be constructed as the commutative algebra on the set of symbols $\{a_{ij}\mid {a\in A},\; {1\leq i,j\leq\lvert N\rvert}\}$, subject to the relations (see~\cite[p. 5751]{VdB08} or~\cite[\S 4]{CB11})
\begin{equation}\label{eq:KR-relations-defining-AV}
(za)_{ij}=za_{ij},\quad (a+b)_{ij}=a_{ij}+b_{ij},\quad (ab)_{ij}=a_{ik}b_{kj},\quad \nu_{ij}=\mathrm{I}^\nu_{ij},
\end{equation}
for all $a,b\in A$, $\nu\in I$, $z\in\kk$ and $i,j=1,\dots,\lvert N\rvert$ (here and below we sum over repeated indices).
One can show that our standing assumption that $A$ is finitely generated over $\kk$ implies that the commutative algebra $A_N$ is also finitely generated over $\kk$ (cf.~\cite[\S 3]{LeBruyn-VandeWeyer.2002}).

As shown by Van den Bergh~\cite[Proposition 7.5.2]{VdB08}, double Poisson algebras satisfy the Kontsevich--Rosenberg principle. More precisely, for any $R$-linear double Poisson algebra $(A,\lr{-,-})$ (see Definition \ref{def:double-Poisson algebra}) and $N\in\NN^I$, the commutative algebra $A_N$ is Poisson, with Poisson bracket $\{-,-\}$  given on generators by 
\begin{equation}\label{eq:KR-double-Poisson-algebras}
\{a_{ij},b_{uv}\}=\lr{a,b}^{\prime}_{uj}\lr{a,b}^{\pprime}_{iv},
\end{equation}
for $a,b\in A$, $i,j,u,v=1,\dots,\lvert N\rvert$. 
The arrangement of indices in \eqref{eq:KR-double-Poisson-algebras}, called the \emph{standard index convention}, will repeatedly appear below.

When an algebra $\mathcal{V}$ is equipped with a derivation $\partial\colon\mathcal{V}\to\mathcal{V}$,
the commutative algebra $\mathcal{V}_N$ (defined as above) has a derivation $\partial_N$ given on generators by (see, e.g.,~\cite[p. 5713]{VdB08})
\begin{equation}
\partial_N(a_{ij})\defeq(\partial a)_{ij}.
\label{eq:KR-derivation-partial}
\end{equation}
One of the main results in \cite{DSKV15} is the following theorem, whereby double Poisson vertex algebras satisfy the Kontsevich--Rosenberg principle.

\begin{theorem}[{\cite[Proposition 3.20 \and Theorem 3.22]{DSKV15}}]
\label{thm:KR-double-Poisson-vertex-algebras}
Let $\mathcal{V}$ be a differential algebra endowed with a double $\lambda$-bracket $\lr{-{}_\lambda-}$, written as $\lr{a{}_\lambda b}=\sum_{n\in\mathbb{N}}\big((a_{{n}}b)^\prime\otimes (a_{{n}}b)^{\prime\prime}\big)\lambda^{{n}}$, for all $a,b\in\mathcal{V}$. Then there exists a well-defined $\lambda$-bracket on $\mathcal{V}_N$, given by 
\begin{equation}
\{a_{ij}{}_\lambda b_{uv}\}=\sum_{{n}\in\mathbb{N}}(a_{{n}}b)^\prime_{uj}(a_{{n}}b)^{\prime\prime}_{iv}\lambda^{{n}},
\label{lambda-bracket-comm}
\end{equation}
where $a_{ij},b_{uv}\in\mathcal{V}_N$.
In addition, if $\mathcal{V}$ is a double Poisson vertex algebra, then $\mathcal{V}_N$ is a Poisson vertex algebra with $\lambda$-bracket defined by \eqref{lambda-bracket-comm}.
\end{theorem}

\subsection{The Kontsevich--Rosenberg principle for double Courant--Dorfman algebras}
\label{sec:KR-double-CD-section}
\allowdisplaybreaks

The \emph{Van den Bergh functor}~\cite[\S 3.3]{VdB08a} (see also~\cite[\S 6.3]{CBEG07}) is the additive functor 
\begin{equation}\label{eq:VdB-functor-def}
(-)_N\colon \mathbf{bimod}(A)\lto\mathbf{mod}(A_N)
\end{equation}
that maps a finitely generated $A$-bimodule $E$ to the finitely generated $A_N$-module $E_N$ generated by a set of symbols $\{e_{ij}\mid {e\in E},\, {1\leq i,j\leq \lvert N\rvert}\}$ that satisfy
\begin{subequations}
\label{eq:KR-relaciones-bimodules}
\begin{align}\label{eq:KR-relaciones-bimodules.a}
(e+f)_{ij}&=e_{ij}+f_{ij},
\\
\label{eq:KR-relaciones-bimodules.b}
(ae)_{ij}&=a_{ik}e_{kj},
\\\label{eq:KR-relaciones-bimodules.c}
(ea)_{ij}&=a_{kj}e_{ik},
\end{align}
\end{subequations}
for all $a\in A$, $e,f\in E$ and $i,j=1,\dots,\lvert N\rvert$.
More intrinsically, 
\[
E_N=E\otimes_{A^{\e}}\MM_N(A_N),
\]
where we use the $A$-bimodule structure on $\MM_N(A_N)$ obtained via the universal $R$-algebra morphism $A\to \MM_N(A_N)$, $a\mapsto(a_{ij})$ (given by the counit of the adjunction~\eqref{eq:KR.adjunction-A_N}), and the $A_N$-module structure obtained via the diagonal embedding $A_N\to \MM_N(A_N)$ (see~\cite[Lemma 3.3.1]{VdB08a} for details).

Van den Bergh~\cite[Lemma 3.3.7 and Proposition 6.1]{VdB08a} used this functor to prove that appropriate noncommutative structures on $A$-bimodules, such as double Lie algebroids 
and quasi-bisymplectic structures, 
satisfy the Kontsevich--Rosenberg principle. 
Here we use it to prove that double Courant--Dorfman algebras satisfy the Kontsevich--Rosenberg principle (see Definition \ref{CD-Definition}).


\begin{theorem}\label{thm:KR-double-Courant-Dorfman}
Let $(A,E,\ii{-,-},\partial,\cc{-,-})$ be an $R$-linear double Courant--Dorfman algebra.
For all $N\in\NN^I$, we define
\begin{subequations}\label{eq:KR-CD-thm}
\begin{align}\label{eq:KR-CD-thm.a}
\partial_N a_{ij}&\defeq (\partial a)_{ij},
\\\label{eq:KR-CD-thm.b}
\langle e_{ij},f_{uv}\rangle&\defeq \ii{e,f}^{\prime}_{uj}\ii{e,f}^{\pprime}_{iv},
\\\label{eq:KR-CD-thm.c}
[e_{ij},f_{uv}]&\defeq \cc{e,f}^{\prime}_{l,uj}  \cc{e,f}^{\pprime}_{l,iv}+  \cc{e,f}^{\prime}_{r,uj}  \cc{e,f}^{\pprime}_{r,iv},
\end{align}\end{subequations}
for all $a\in A$ and $e,f\in E$.
Then the 5-tuple $(A_N,E_N,\langle-,-\rangle,\partial_N, [-,-])$ is a Courant--Dorfman algebra in the sense of Definition \ref{def:CD-algebra-def-comm}.
\end{theorem}

\begin{proof}
Let $B_N$ be the commutative algebra freely generated by the set of symbols $\{a_{ij}\mid {a\in A},\; {1\leq i,j\leq\lvert N\rvert }\}$. Consider the $B_N$-module $F_N$ freely generated by the set of symbols $\{e_{ij}\mid {e\in E},\, {1\leq i,j\leq \lvert N\rvert}\}$.
Since $A_N$ and $E_N$ are the quotients of $B_N$ and $F_N$ modulo the ideal and the submodule generated by the relations~\eqref{eq:KR-relations-defining-AV} and~\eqref{eq:KR-relaciones-bimodules}, respectively, we can start defining maps
\begin{equation}\label{eq:thm:KR-double-Courant-Dorfman.1}
\partial_N\colon B_N\lto F_N,\quad\langle-,-\rangle\colon F_N\otimes_{B_N} F_N\lto B_N,\quad [-,-]\colon F_N\otimes F_N\lto F_N,
\end{equation}
and then show they satisfy appropriate relations associated by~\eqref{eq:KR-relaciones-bimodules}, so they descend to 
\begin{equation}\label{eq:thm:KR-double-Courant-Dorfman.2}
\partial_N\colon A_N\lto E_N,\quad \langle-,-\rangle\colon E_N\otimes_{A_N} E_N\lto A_N,\quad [-,-]\colon E_N\otimes E_N\lto E_N.
\end{equation}
The maps $\partial_N$, $\langle-,-\rangle$ and $[-,-]$ in~\eqref{eq:thm:KR-double-Courant-Dorfman.1} are respectively defined extending~\eqref{eq:KR-CD-thm} by imposing the Leibniz rule on $\partial_N$, the condition that $\langle-,-\rangle$ is bilinear over $B_N$, and the condition that $[-,-]$ satisfies~\eqref{eq:CD-comm.a} and~\eqref{eq:CD-comm.g}.
To show that the maps~\eqref{eq:thm:KR-double-Courant-Dorfman.1} descend to maps~\eqref{eq:thm:KR-double-Courant-Dorfman.2}, we need to prove the following relations for all $a,b\in A$, $\nu\in I$, $e,f,g\in E$, $i,j,u,v,t=1,\ldots,\lvert N\rvert$:
\begin{subequations}\label{eq:thm:KR-double-Courant-Dorfman.3}
\begin{gather}\label{eq:thm:KR-double-Courant-Dorfman.3.a}
\partial_N((a+b)_{ij})=\partial_N(a_{ij}+b_{ij}),
\quad
\partial_N((ab)_{ij})=\partial_N(a_{it}b_{tj}),
\quad
\partial_N(\nu_{ij})=\partial_N(\mathrm{I}_{ij}),
\\\label{eq:thm:KR-double-Courant-Dorfman.3.b}
\langle e_{ij},(f+g)_{uv}\rangle=\langle e_{ij},f_{uv}\rangle+\langle e_{ij},g_{uv}\rangle,
\quad
\langle(e+f)_{ij},g_{uv}\rangle=\langle e_{ij},g_{uv}\rangle+\langle f_{ij},g_{uv}\rangle,
\\\label{eq:thm:KR-double-Courant-Dorfman.3.c}
\langle e_{ij},(af)_{uv}\rangle=\langle e_{ij},a_{ut}f_{tv}\rangle,
\quad
\langle e_{ij},(fa)_{uv}\rangle=\langle e_{ij},a_{tv}f_{uv}\rangle,
\\\label{eq:thm:KR-double-Courant-Dorfman.3.d}
\langle(ae)_{ij},f_{uv}\rangle=\langle a_{it}e_{tj},f_{uv}\rangle,
\quad
\langle (ea)_{ij},f_{uv}\rangle=\langle a_{tj}e_{it},f_{uv}\rangle,
\\\label{eq:thm:KR-double-Courant-Dorfman.3.e}
[e_{ij},(f+g)_{uv})]=[e_{ij},f_{uv}]+[e_{ij},g_{uv}],
\quad
[(e+f)_{ij},g_{uv}]=[e_{ij},g_{uv}]+[f_{ij},g_{uv}],
\\\label{eq:thm:KR-double-Courant-Dorfman.3.f}
[e_{ij},(af)_{uv}]=[e_{ij},a_{ut}f_{tv}],
\quad
[e_{ij},(fa)_{uv}]=[e_{ij},a_{tv}f_{ut}],
\\\label{eq:thm:KR-double-Courant-Dorfman.3.g}
[(ae)_{ij},f_{uv}]=[a_{it}e_{tj},f_{uv}],
\quad
[(ea)_{ij},f_{uv}]=[a_{tj}e_{it},f_{uv}].
\end{gather}\end{subequations}
Note that these equations take place in $A_N$ and $E_N$, but involve sum and multiplication operations on $B_N$ and $F_N$. 
The relations~\eqref{eq:thm:KR-double-Courant-Dorfman.3.a} follow from~\eqref{eq:KR-CD-thm.a}, because $\partial$ is an $R$-linear derivation.
The relations~\eqref{eq:thm:KR-double-Courant-Dorfman.3.b} and~\eqref{eq:thm:KR-double-Courant-Dorfman.3.e} respectively follow because $\ii{-,-}$ and $\cc{-,-}$ are additive in each variable.
The relations~\eqref{eq:thm:KR-double-Courant-Dorfman.3.c} and~\eqref{eq:thm:KR-double-Courant-Dorfman.3.d} follow because $\ii{e,af}=a\ii{e,f}$, $\ii{e,fa}=\ii{e,f}a$, $\ii{e,fa}=\ii{e,f}a$, $\ii{ae,f}=a*\ii{e,f}$ and $\ii{ea,f}=\ii{e,f}*a$ (see~\eqref{eq:pairing.1}).
To prove~\eqref{eq:thm:KR-double-Courant-Dorfman.3.f} and~\eqref{eq:thm:KR-double-Courant-Dorfman.3.g}, we project~\eqref{eq:CD.d}, \eqref{eq:CD.e}, \eqref{eq:CD-identities.1.c} and~\eqref{eq:CD-identities.1.d} onto $E\otimes A$ and $A\otimes E$, so using Sweedler's notation, we obtain
\begin{subequations}\label{eq:proof:thm:KR-double-Courant-Dorfman.1}
\begin{align}\label{eq:proof:thm:KR-double-Courant-Dorfman.1.a}
\cc{e,fa}^{\prime}_l\otimes \cc{e,fa}^{\pprime}_l&= \cc{e,f}^{\prime}_l\otimes \cc{e,f}^{\pprime}_la +f\ii{e,\partial a}^{\prime}\otimes \ii{e,\partial a}^{\pprime},
\\\label{eq:proof:thm:KR-double-Courant-Dorfman.1.b}
\cc{e,fa}^{\prime}_r\otimes \cc{e,fa}^{\pprime}_r&= \cc{e,f}^{\prime}_r\otimes \cc{e,f}^{\pprime}_ra,
\\\label{eq:proof:thm:KR-double-Courant-Dorfman.1.c}
\cc{e,af}^{\prime}_l\otimes \cc{e,af}^{\pprime}_l&=a \cc{e,f}^{\prime}_l\otimes \cc{e,f}^{\pprime}_l,
\\\label{eq:proof:thm:KR-double-Courant-Dorfman.1.d}
\cc{e,af}^{\prime}_r\otimes \cc{e,af}^{\pprime}_r&=a \cc{e,f}^{\prime}_r\otimes \cc{e,f}^{\pprime}_r+\ii{e,\partial a}^{\prime}\otimes \ii{e,\partial a}^{\pprime}f,
\\\label{eq:proof:thm:KR-double-Courant-Dorfman.1.e}
\cc{ae,f}^{\prime}_l\otimes \cc{ae,f}^{\pprime}_l&= \cc{e,f}^{\prime}_l\otimes a\cc{e,f}^{\pprime}_l -\ii{\partial a,f}^{\prime}e\otimes \ii{\partial a,f}^{\pprime},
\\\label{eq:proof:thm:KR-double-Courant-Dorfman.1.f}
\cc{ae,f}^{\prime}_r\otimes \cc{ae,f}^{\pprime}_r&=\cc{e,f}^{\prime}_r\otimes a\cc{e,f}^{\pprime}_r+\ii{e,f}^{\prime}\otimes \partial a\ii{e,f}^{\pprime},
\\\label{eq:proof:thm:KR-double-Courant-Dorfman.1.g}
\cc{ea,f}^{\prime}_l\otimes \cc{ea,f}^{\pprime}_l&= \cc{e,f}^{\prime}_la\otimes\cc{e,f}^{\pprime}_l +\ii{e,f}^{\prime}\partial a\otimes \ii{e,f}^{\pprime},
\\\label{eq:proof:thm:KR-double-Courant-Dorfman.1.h}
\cc{ea,f}^{\prime}_r\otimes \cc{ae,f}^{\pprime}_r&=\cc{e,f}^{\prime}_ra\otimes\cc{e,f}^{\pprime}_r-\ii{\partial a,f}^{\prime}\otimes e\ii{\partial a,f}^{\pprime}.
\end{align}\end{subequations}  
The left-hand equation in~\eqref{eq:thm:KR-double-Courant-Dorfman.3.f} now follows from~\eqref{eq:KR-relaciones-bimodules.b}, \eqref{eq:KR-CD-thm}, \eqref{eq:proof:thm:KR-double-Courant-Dorfman.1.c} and~\eqref{eq:proof:thm:KR-double-Courant-Dorfman.1.d}:
\begin{align*}
[e_{ij},(af)_{uv}]&
=\cc{e,af}^{\prime}_{l,uj}\cc{e,af}^{\pprime}_{l,iv}+\cc{e,af}^{\prime}_{r,uj}\cc{e,af}^{\pprime}_{r,iv}
\\&
=(a\cc{e,f}^{\prime}_{l})_{uj}\cc{e,f}^{\pprime}_{l,iv} + (a\cc{e,f}^{\prime}_{r})_{uj}\cc{e,f}^{\pprime}_{r,iv} + \ii{e,\partial a}^{\prime}_{uj}(\ii{e,\partial a}^{\pprime}f)_{iv}
\\&
=a_{ut}\big(\cc{e,f}^{\prime}_{l,tj}\cc{e,f}^{\pprime}_{l,iv} + \cc{e,f}^{\prime}_{r,tj}\cc{e,f}^{\pprime}_{r,iv}\big)+\ii{e,\partial a}^{\prime}_{uj}\ii{e,\partial a}^{\pprime}_{it} f_{tv}
\\&
=a_{ut}[e_{ij},f_{tv}]+\langle e_{ij},\partial_N a_{ut}\rangle f_{tv}=[e_{ij},a_{ut}f_{tv}]
\end{align*}
(the last identity is a consequence of our definition of $[-,-]$ on $F_N\otimes F_N$, by~\eqref{eq:CD-comm.a}). 
The right-hand equation in~\eqref{eq:thm:KR-double-Courant-Dorfman.3.f} follows by a similar calculation from~\eqref{eq:KR-relaciones-bimodules.c}, \eqref{eq:KR-CD-thm}, \eqref{eq:proof:thm:KR-double-Courant-Dorfman.1.a} and~\eqref{eq:proof:thm:KR-double-Courant-Dorfman.1.b}. 
The left-hand relation in~\eqref{eq:thm:KR-double-Courant-Dorfman.3.g} follows from~\eqref{eq:KR-relaciones-bimodules.b}, \eqref{eq:KR-CD-thm}, \eqref{eq:proof:thm:KR-double-Courant-Dorfman.1.e} and~\eqref{eq:proof:thm:KR-double-Courant-Dorfman.1.f}:
\begin{align*}&
[(ae)_{ij},f_{uv}]
=\cc{ae,f}^{\prime}_{l,uj}\cc{ae,f}^{\pprime}_{l,iv}+\cc{ae,f}^{\prime}_{r,uj}\cc{ae,f}^{\pprime}_{r,iv}
=\cc{e,f}^{\prime}_{l,uj}(a\cc{e,f}^{\pprime}_{l})_{iv}
\\&\qquad 
-(\ii{e,\partial a}^{\prime}e)_{uj}\ii{\partial a,f}^{\pprime}_{iv}+\cc{e,f}^{\prime}_{r,uj}(a\cc{e,f}^{\pprime}_{r})_{iv}+\ii{e,f}^{\prime}_{uj}(\partial a\ii{e,f}^{\pprime})_{iv}
\\&\quad
=a_{it}(\cc{e,f}^{\prime}_{l,uj}\cc{e,f}^{\pprime}_{l,tv}+\cc{e,f}^{\prime}_{r,uj}\cc{e,f}^{\pprime}_{r,tv})
+\partial_Na_{it}\ii{e,f}^{\prime}_{uj}\ii{e,f}^{\pprime}_{tv}
\\&\qquad 
-\ii{\partial a,f}^{\prime}_{ut}\ii{\partial a,f}^{\pprime}_{iv}e_{tj}
\\&\quad 
=a_{it}[e_{tj},f_{uv}]+\langle e_{tj},f_{uv}\rangle\partial_Na_{it}-\langle \partial_Na_{it},f_{uv}\rangle e_{tj}
=[a_{it}e_{tj},f_{uv}]
\end{align*}
(the last identity follows from the definition of $[-,-]$ on $F_N\otimes F_N$, by~\eqref{eq:CD-comm.g}).
The right-hand relation in~\eqref{eq:thm:KR-double-Courant-Dorfman.3.g} follows by a similar calculation from~\eqref{eq:KR-relaciones-bimodules.c}, \eqref{eq:KR-CD-thm}, \eqref{eq:proof:thm:KR-double-Courant-Dorfman.1.g} and~\eqref{eq:proof:thm:KR-double-Courant-Dorfman.1.h}. 
Thus the maps~\eqref{eq:thm:KR-double-Courant-Dorfman.1} descend to maps~\eqref{eq:thm:KR-double-Courant-Dorfman.2}, 
and $\partial_N\colon A_N\to E_N$ is a derivation, because so is $\partial_N\colon B_N\to F_N$. 
Moreover, the bilinear form $\langle-,-\rangle$ on $E_N$ is symmetric, because so is the pairing $\ii{-,-}$ on $E$ and hence the bilinear form $\langle-,-\rangle$ on $F_N$ (cf.~\cite[Lemma 3.3.3]{VdB08a}). 

To conclude the proof, we need to show the maps~\eqref{eq:thm:KR-double-Courant-Dorfman.2} satisfy the axioms~\eqref{eq:CD-comm}.
The bracket $[-,-]$ on $E_N$ satisfies~\eqref{eq:CD-comm.a} (and~\eqref{eq:CD-comm.g}), because so does the bracket $[-,-]$ on $F_N$ by construction.
The identity \eqref{eq:CD-comm.e} follows because \eqref{eq:CD.b} and \eqref{eq:KR-CD-thm.c} imply 
\[
[\partial_N a_{ij}, e_{uv}]= [(\partial a)_{ij}, e_{uv}]=\cc{\partial a, e}^{\prime}_{l,uj} \cc{\partial a, e}^{\pprime}_{l,iv} + \cc{\partial a, e}^{\prime}_{r,uj} \cc{\partial a, e}^{\pprime}_{r, iv}=0. 
\]
Similarly, the identity \eqref{eq:CD-comm.f} follows because by \eqref{eq:CD.c}, \eqref{eq:KR-CD-thm.a} and \eqref{eq:KR-CD-thm.b}, we get
\[
\langle \partial_N a_{ij},\partial_N b_{uv}\rangle= \langle (\partial a)_{ij}, (\partial b)_{uv}\rangle=\ii{\partial a,\partial b}^{\prime}_{uj} \ii{\partial a,\partial b}^{\pprime}_{iv}=0.
\]

To prove~\eqref{eq:CD-comm.c}, we project \eqref{eq:CD.a} onto $E\otimes A$ and $A\otimes E$, so using~\eqref{eq:partialLeibniz-rule}, we obtain 
\begin{subequations}
\label{eq:KR-skew-symm}
\begin{align}\label{eq:KR-skew-symm.a}
\partial\ii{e,f}^{\prime}\otimes \ii{e,f}^{\pprime}&=\cc{e,f}^{\prime}_l\otimes \cc{e,f}^{\pprime}_l+\cc{f,e}^{\pprime}_r\otimes \cc{f,e}^{\prime}_r,
\\\label{eq:KR-skew-symm.b}
\ii{e,f}^{\prime}\otimes \partial \ii{e,f}^{\pprime}&=\cc{e,f}^{\prime}_r\otimes \cc{e,f}^{\pprime}_r+\cc{f,e}^{\pprime}_l\otimes \cc{f,e}^{\prime}_l.
\end{align}
\end{subequations}
The identity~\eqref{eq:CD-comm.c} follows now because using \eqref{eq:KR-CD-thm} and \eqref{eq:KR-skew-symm}, we have
\begin{align*}
&[e_{ij},f_{uv}]+[f_{uv},e_{ij}]
\\&\quad
=\cc{e,f}^{\prime}_{l,uj} \cc{e,f}^{\pprime}_{l,iv}
+ \cc{f,e}^{\pprime}_{r,uj}\cc{f,e}^{\prime}_{r,iv}
+ \cc{e,f}^{\prime}_{r,uj} \cc{e,f}^{\pprime}_{r,iv}
+\cc{f,e}^{\pprime}_{l,uj}\cc{f,e}^{\prime}_{l,iv}
\\
&\quad=(\partial\ii{e,f}^{\prime})_{uj}\ii{e,f}^{\pprime}_{iv}+\ii{e,f}^{\prime}_{uj}(\partial\ii{e,f}^{\pprime})_{iv}
\\
&\quad=\partial_N(\ii{e,f}^{\prime}_{uj})\ii{e,f}^{\pprime}_{iv}+\ii{e,f}^{\prime}_{uj}\partial_N(\ii{e,f}^{\pprime})_{iv}
\\
&\quad=\partial_N(\ii{e,f}^{\prime}_{uj}\ii{e,f}^{\pprime}_{iv})=\partial_N\langle e_{ij},f_{uv}\rangle.
\end{align*}

To prove~\eqref{eq:CD-comm.b}, we need to check this identity on generators of $E_N$. 
Since $\langle-,-\rangle$ is $A_N$-bilinear, using~\eqref{eq:KR-CD-thm.a} and~\eqref{eq:KR-CD-thm.b}, we see that the left-hand side of \eqref{eq:CD-comm.b} is 
\begin{align}\notag
\big\langle e_{ij},\partial_N\langle f_{uv}, g_{pq} \rangle\big\rangle
&=\langle e_{ij},\partial_N (\ii{f,g}^{\prime}_{pv}\ii{f,g}^{\pprime}_{uq})\rangle
\\\label{eq:KR-CD7-LHS}
&=\langle e_{ij}, \partial\ii{f,g}^{\prime}_{pv}\rangle\ii{f,g}^{\pprime}_{uq}+\ii{f,g}^{\prime}_{pv}\langle e_{ij},\partial\ii{f,g}^{\pprime}_{uq}\rangle.
\end{align}
To calculate the two terms in the right-hand side of~\eqref{eq:KR-CD7-LHS}, we expand the identities~\eqref{eq:CD.g} and~\eqref{eq:CD-identities.1.f} in Sweedler's notation (with $e$ and $f$ swapped in~\eqref{eq:CD-identities.1.f}):
\begin{subequations}\label{eq:auxiliar-CD7-KR.1}
\begin{align}\label{eq:auxiliar-CD7-KR.1.a}
\begin{split}&
\ii{e,\partial\ii{f,g}^{\prime}}^{\prime}\otimes\ii{e,\partial\ii{f,g}^{\prime}}^{\pprime}\otimes\ii{f,g}^{\pprime}
\\&\quad
=\ii{\cc{e,f}^{\prime}_l,g}^{\prime}\otimes\cc{e,f}^{\pprime}_l\otimes\ii{\cc{e,f}^{\prime}_l,g}^{\pprime}+\cc{e,g}^{\prime}_r\otimes\ii{f,\cc{e,g}^{\pprime}_r}^{\prime}\otimes\ii{f,\cc{e,g}^{\pprime}_r}^{\pprime},
\end{split}
\\\label{eq:auxiliar-CD7-KR.1.b}
\begin{split}&
\ii{f,g}^{\prime}\otimes\ii{e,\partial\ii{f,g}^{\pprime}}^{\prime}\otimes\ii{e,\partial\ii{f,g}^{\pprime}}^{\pprime}
\\&\quad
=\ii{\cc{e,f}^{\pprime}_r,g}^{\prime}\otimes\cc{e,f}^{\prime}_r\otimes\ii{\cc{e,f}^{\pprime}_r,g}^{\pprime}
+\ii{f,\cc{e,g}^{\prime}_l}^{\prime}\otimes\ii{f,\cc{e,g}^{\prime}_l}^{\pprime}\otimes\cc{e,g}^{\pprime}_l.
\end{split}\end{align}\end{subequations}
Using~\eqref{eq:auxiliar-CD7-KR.1} and~\eqref{eq:KR-CD-thm}, we can now calculate the terms in the right-hand side of~\eqref{eq:KR-CD7-LHS}:
\begin{subequations}\label{eq:auxiliar-CD7-KR.2}
\begin{align}\notag&
\langle e_{ij}, \partial\ii{f,g}^{\prime}_{pv}\rangle\ii{f,g}^{\pprime}_{uq}
=\ii{e,\partial \ii{f,g}^{\prime}}^{\prime}_{pj} \ii{e,\partial \ii{f,g}^{\prime}}^{\pprime}_{iv}\ii{f,g}^{\pprime}_{uq}
\\\notag&\qquad
=\ii{\cc{e,f}^{\prime}_l,g}^{\prime}_{pj}\cc{e,f}^{\pprime}_{l,iv}\ii{\cc{e,f}^{\prime}_l,g}^{\pprime}_{uq}+\cc{e,g}^{\prime}_{r,pj}\ii{f,\cc{e,g}^{\pprime}_r}^{\prime}_{iv} \ii{f,\cc{e,g}^{\pprime}_r}^{\pprime}_{uq} 
\\\label{eq:auxiliar-CD7-KR.2.a}&\qquad
=\langle\cc{e,f}^{\prime}_{l,uj},g_{pq}\rangle\cc{e,f}^{\pprime}_{l,iv}+\cc{e,g}^{\prime}_{r,pj}\langle f_{uv},\cc{e,g}^{\pprime}_{r,iq}\rangle,
\\\notag&
\ii{f,g}^{\prime}_{pv}\langle e_{ij},\partial\ii{f,g}^{\pprime}_{uq}\rangle
=\ii{f,g}^{\prime}_{pv}\ii{e,\partial\ii{f,g}^{\pprime}}^{\prime}_{uj} \ii{e,\partial\ii{f,g}^{\pprime}}^{\pprime}_{iq}
\\\notag&\qquad
=\ii{\cc{e,f}^{\pprime}_r,g}^{\prime}_{pv}\cc{e,f}^{\prime}_{r,uj}\ii{\cc{e,f}^{\pprime}_r,g}^{\pprime}_{iq}+\ii{f,\cc{e,g}^{\prime}_l}^{\prime}_{pv}  \ii{f,\cc{e,g}^{\prime}_l}^{\pprime}_{uj}\cc{e,g}^{\pprime}_{l,iq}
\\\label{eq:auxiliar-CD7-KR.2.b}&\qquad
=\cc{e,f}^{\prime}_{r,uj}\langle\cc{e,f}^{\pprime}_{r,iv},g_{pq}\rangle+\langle f_{uv},\cc{e,g}^{\prime}_{l,pj}\rangle\cc{e,g}^{\pprime}_{l,iq}.
\end{align}\end{subequations}
Putting~\eqref{eq:auxiliar-CD7-KR.2} together into~\eqref{eq:KR-CD7-LHS} and using~\eqref{eq:KR-CD-thm} again, we obtain the identity~\eqref{eq:CD-comm.b}:
\begin{align*}
\big\langle e_{ij},\partial_N\langle f_{uv}, g_{pq} \rangle\big\rangle
&=\langle\cc{e,f}^{\prime}_{l,uj}\cc{e,f}^{\pprime}_{l,iv},g_{pq}\rangle
 +\langle\cc{e,f}^{\prime}_{r,uj}\cc{e,f}^{\pprime}_{r,iv},g_{pq}\rangle
\\&
\qquad +\langle f_{uv},\cc{e,g}^{\prime}_{l,pj}\cc{e,g}^{\pprime}_{l,iq}\rangle
+\langle f_{uv},\cc{e,g}^{\prime}_{r,pj}\cc{e,g}^{\pprime}_{r,iq}\rangle
\\&
=\langle[e_{ij},f_{uv}],g_{pq}\rangle+\langle f_{uv},[e_{ij},g_{pq}]\rangle.
\end{align*}

Finally, we need to prove the Jacobi identity~\eqref{eq:CD-comm.d} on generators of $E_N$, that is, 
\begin{equation}
\label{eq:KR-Jacobi-inicial}
[e_{ij},[f_{uv},g_{pq}]]=[[e_{ij},f_{uv}],g_{pq}]+[f_{uv},[e_{ij},g_{pq}]].
\end{equation}
We start calculating the three terms in \eqref{eq:KR-Jacobi-inicial} 
using~\eqref{eq:CD-comm.a},~\eqref{eq:CD-comm.g} and~\eqref{eq:KR-CD-thm}:
\begin{subequations}\label{eq:Jacobi-KR.1}
\begin{align}\notag
[e_{ij},[f_{uv},g_{pq}]]&
=[e_{ij},\cc{f,g}^{\prime}_{l,pv}\cc{f,g}^{\pprime}_{l,uq}]+[e_{ij},\cc{f,g}^{\prime}_{r,pv}\cc{f,g}^{\pprime}_{r,uq}]
\\\notag
&=\cc{f,g}^{\pprime}_{l,uq}[e_{ij},\cc{f,g}^{\prime}_{l,pv}]+\langle e_{ij},\partial_N\cc{f,g}^{\pprime}_{l,uq}\rangle\cc{f,g}^{\prime}_{l,pv}
\\\notag
&\quad + \cc{f,g}^{\prime}_{r,pv}[e_{ij},\cc{f,g}^{\pprime}_{r,uq}]+\langle e_{ij},\partial_N\cc{f,g}^{\prime}_{r,pv}\rangle \cc{f,g}^{\pprime}_{r,uq}
\\\label{eq:Jacobi-KR.1.1}
&=x_l+x_r,
\\\notag
[[e_{ij},f_{uv}],g_{pq}]&=[\cc{e,f}^{\prime}_{l,uj}\cc{e,f}^{\pprime}_{l,iv},g_{pq}]+[\cc{e,f}^{\prime}_{r,uj}\cc{e,f}^{\pprime}_{r,iv},g_{pq}]
\\\notag &
=\cc{e,f}^{\pprime}_{l,iv}][\cc{e,f}^{\prime}_{l,uj},g_{pq}]+\cc{e,f}^{\prime}_{r,uj}[\cc{e,f}^{\pprime}_{r,iv},g_{pq}]
\\\notag &\quad
+\langle\cc{e,f}^{\prime}_{l,uj},g_{pq}\rangle\partial_N\cc{e,f}^{\pprime}_{l,iv}
+\langle\cc{e,f}^{\pprime}_{r,iv},g_{pq}\rangle\partial_N\cc{e,f}^{\prime}_{r,uj}
\\\notag &\quad
-\langle\partial_N\cc{e,f}^{\prime}_{r,uj},g_{pq}\rangle\cc{e,f}^{\pprime}_{r,iv}
-\langle\partial_N\cc{e,f}^{\pprime}_{l,iv},g_{pq}\rangle\cc{e,f}^{\prime}_{l,uj}
\\\label{eq:Jacobi-KR.2.1.a}&
=y_l+y_r,
\\\notag 
[f_{uv},[e_{ij},g_{pq}]]&
=[f_{uv},\cc{e,g}^{\prime}_{l,pj}\cc{e,g}^{\pprime}_{l,iq}]+[f_{uv},\cc{e,g}^{\prime}_{r,pj}\cc{e,g}^{\pprime}_{r,iq}]
\\\notag &
=\cc{e,g}^{\pprime}_{l,iq}[f_{uv},\cc{e,g}^{\prime}_{l,pj}]
+\langle f_{uv},\partial\cc{e,g}^{\prime}_{r,pj}\rangle\cc{e,g}^{\pprime}_{r,iq}
\\\notag &\quad
+\cc{e,g}^{\prime}_{r,pj}[f_{uv},\cc{e,g}^{\pprime}_{r,iq}]
+\langle f_{uv},\partial_N\cc{e,g}^{\pprime}_{l,iq}\rangle\cc{e,g}^{\prime}_{l,pj}
\\\label{eq:Jacobi-KR.2.1.b}&
=z_l+z_r,
\end{align}
\end{subequations}
where we define 
\begin{subequations}\label{eq:Jacobi-KR.1.2}
\begin{align}\label{eq:Jacobi-KR.1.2.a}
\begin{split}
x_l&
\defeq\(\cc{e,\cc{f,g}^{\prime}_l}^{\prime}_{l,pj} \cc{e,\cc{f,g}^{\prime}_l}^{\pprime}_{l,iv}
+\cc{e,\cc{f,g}^{\prime}_l}^{\prime}_{r,pj} \cc{e,\cc{f,g}^{\prime}_l}^{\pprime}_{r,iv}\)\cc{f,g}^{\pprime}_{l,uq}
\\ &
\quad +\ii{e,\partial\cc{f,g}^{\prime}_r}^{\prime}_{pj} \ii{e,\partial\cc{f,g}^{\prime}_r}^{\pprime}_{iv}\cc{f,g}^{\pprime}_{r,uq},
\end{split}
\\\label{eq:Jacobi-KR.1.2.b}
\begin{split}
x_r&\defeq\cc{f,g}^{\prime}_{l,pv}\ii{e,\partial\cc{f,g}^{\pprime}_l}^{\prime}_{uj} \ii{e,\partial\cc{f,g}^{\pprime}_l}^{\pprime}_{iq}
\\ &
\quad +\cc{f,g}^{\prime}_{r,pv}\(\cc{e,\cc{f,g}^{\pprime}_r}^{\prime}_{l,uj}\cc{e,\cc{f,g}^{\pprime}_r}^{\pprime}_{l,iq} +\cc{e,\cc{f,g}^{\pprime}_r}^{\prime}_{r,uj}\cc{e,\cc{f,g}^{\pprime}_r}^{\pprime}_{r,iq}\).
\end{split}
\\\label{eq:Jacobi-KR.2.2.a}
\begin{split}
y_l&
\defeq\cc{e,f}^{\pprime}_{l,iv}\(\cc{\cc{e,f}^{\prime}_l,g}^{\prime}_{l,pj}\cc{\cc{e,f}^{\prime}_l,g}^{\pprime}_{l,uq}+\cc{\cc{e,f}^{\prime}_l,g}^{\prime}_{r,pj}\cc{\cc{e,f}^{\prime}_l,g}^{\pprime}_{r,uq}\)
\\ &\hspace*{-1.3ex}
+\ii{\cc{e,f}^{\prime}_l,g}^{\prime}_{pj}\ii{\cc{e,f}^{\prime}_l,g}^{\pprime}_{uq}\(\partial\cc{e,f}^{\pprime}_l\)_{iv}-\ii{\partial\cc{e,f}^{\prime}_r,g}^{\prime}_{pj}\ii{\partial\cc{e,f}^{\prime}_r,g}^{\pprime}_{uq}\cc{e,f}^{\pprime}_{r,iv},
\end{split}
\\\label{eq:Jacobi-KR.2.2.b}
\begin{split}
y_r&
\defeq\cc{e,f}^{\prime}_{r,uj}\(\cc{\cc{e,f}^{\pprime}_r,g}^{\prime}_{l,pv}\cc{\cc{e,f}^{\pprime}_r,g}^{\pprime}_{l,iq}+\cc{\cc{e,f}^{\pprime}_r,g}^{\prime}_{r,pv}\cc{\cc{e,f}^{\pprime}_r,g}^{\pprime}_{r,iq}\)
\\ &
\hspace*{-1.3ex}  +\ii{\cc{e,f}^{\pprime}_r,g}^{\prime}_{pv}\ii{\cc{e,f}^{\pprime}_r,g}^{\pprime}_{iq}\(\partial\cc{e,f}^{\prime}_r\)_{uj}-\ii{\partial\cc{e,f}^{\pprime}_l,g}^{\prime}_{pv}\ii{\partial\cc{e,f}^{\pprime}_l,g}^{\pprime}_{iq}\cc{e,f}^{\prime}_{l,uj},
\end{split}
\\\label{eq:Jacobi-KR.2.2.c}
\begin{split}
z_l&
\defeq\cc{e,g}^{\prime}_{r,pj}\(\cc{f,\cc{e,g}^{\pprime}_r}^{\prime}_{l,iv}\cc{f,\cc{e,g}^{\pprime}_r}^{\pprime}_{l,uq}+\cc{f,\cc{e,g}^{\pprime}_r}^{\prime}_{r,iv}\cc{f,\cc{e,g}^{\pprime}_r}^{\pprime}_{r,uq}\)
\\ &
\quad +\ii{f,\partial\cc{e,g}^{\pprime}_l}^{\prime}_{iv}\ii{f,\partial\cc{e,g}^{\pprime}_l}^{\pprime}_{uq}\cc{e,g}^{\prime}_{l,pj},
\end{split}
\\\label{eq:Jacobi-KR.2.2.d}
\begin{split}
z_r&
\defeq\cc{e,g}^{\pprime}_{l,iq}\(\cc{f,\cc{e,g}^{\prime}_l}^{\prime}_{l,pv}\cc{f,\cc{e,g}^{\prime}_l}^{\pprime}_{l,uj}+\cc{f,\cc{e,g}^{\prime}_l}^{\prime}_{r,pv}\cc{f,\cc{e,g}^{\prime}_l}^{\pprime}_{r,uj}\)
\\ &
\quad +\cc{e,g}^{\pprime}_{r,iq}\ii{f,\partial\cc{e,g}^{\prime}_r}^{\prime}_{pv}\ii{f,\partial\cc{e,g}^{\prime}_r}^{\pprime}_{uj}.
\end{split}\end{align}
\end{subequations}
We consider now the Jacobi identities~\eqref{eq:CD.f} and~\eqref{eq:CD-identities.1.e} for the double Dorfman bracket. As they take place in $E\otimes A\otimes A+c.p.$, they can be projected onto each direct summand.
Using \eqref{eq:decomposition-Courant-Dorfman-def-1}, \eqref{eq:partialLeibniz-rule}, \eqref{eq:extension-CD-bracket} and \eqref{eq:ext-pairing-zero}, we see that \eqref{eq:CD.f} is equivalent to the three identities
\begin{subequations}\label{eq:CD-explicit-left-double-Jacobi}
\begin{align}\begin{split} \label{eq:CD-explicit-left-double-Jacobi.a}
\cc{e,\cc{f,g}^{\prime}_l}_l\otimes\cc{f,g}^{\pprime}_l&=\cc{e,g}^{\prime}_l\otimes\ii{f,\partial\cc{e,g}^{\pprime}_l} 
  +\cc{\cc{e,f}^{\prime}_l,g}_l\otimes_1 \cc{e,f}^{\pprime}_l,
 \end{split}
\\
\begin{split}\label{eq:CD-explicit-left-double-Jacobi.b}
\cc{e,\cc{f,g}^{\prime}_l}_r\otimes\cc{f,g}^{\pprime}_l&=\cc{e,g}^{\prime}_r\otimes\cc{f,\cc{e,g}^{\pprime}_r}_l
+\ii{\cc{e,f}^{\prime}_l,g}\otimes_1\partial\cc{e,f}^{\pprime}_l
\\
&\qquad -\ii{\partial\cc{e,f}^{\prime}_r,g}\otimes_1\cc{e,f}^{\pprime}_r,
\end{split}
\\
\begin{split}\label{eq:CD-explicit-left-double-Jacobi.c}  
\ii{e,\partial\cc{f,g}^{\prime}_r}\otimes \cc{f,g}^{\pprime}_r&=\cc{e,g}^{\prime}_r\otimes\cc{f,\cc{e,g}^{\pprime}_r}_r
 +\cc{\cc{e,f}^{\prime}_l,g}_r\otimes_1\cc{e,f}^{\pprime}_l,
\end{split}\end{align}\end{subequations}
taking place in $E\otimes A\otimes A$,  $A\otimes E\otimes A$ and $A\otimes A\otimes E$, respectively. 
By a similar calculation, we can show that~\eqref{eq:CD-identities.1.e} is equivalent to the three identities
\begin{subequations}\label{eq:CD-explicit-right-double-Jacobi}
\begin{align}\label{eq:CD-explicit-right-double-Jacobi.a}
\cc{f,g}^{\prime}_l\otimes\ii{e,\partial\cc{f,g}^{\pprime}_l}
&=\cc{e,f}^{\prime}_r\otimes_1\cc{\cc{e,f}^{\pprime}_r,g}_l+\cc{f,\cc{e,g}^{\prime}_l}_l\otimes\cc{e,g}^{\pprime}_l,
\\\label{eq:CD-explicit-right-double-Jacobi.b}
\begin{split}
\cc{f,g}^{\prime}_r\otimes\cc{e,\cc{f,g}^{\pprime}_r}_l&=-\cc{e,f}^{\prime}_l\otimes_1\ii{\partial\cc{e,f}^{\pprime}_l,g}+\cc{f,\cc{e,g}^{\prime}_l}_r\otimes\cc{e,g}^{\pprime}_l
\\
&\qquad +\partial\cc{e,f}^{\prime}_r\otimes_1\ii{\cc{e,f}^{\pprime}_r,g},
\end{split}
\\\label{eq:CD-explicit-right-double-Jacobi.c}
\cc{f,g}^{\prime}_r\otimes\cc{e,\cc{f,g}^{\pprime}_r}_r&=\cc{e,f}^{\prime}_r\otimes_1\cc{\cc{e,f}^{\pprime}_r,g}_r+\ii{f,\partial\cc{e,g}^{\prime}_r}\otimes\cc{e,g}^{\pprime}_r.
\end{align}\end{subequations}
Applying 
$(\?)_{pj}\otimes(\?)_{iv}\otimes(\?)_{uq}$ to~\eqref{eq:CD-explicit-left-double-Jacobi}, we obtain the identities
\begin{align*}\begin{split}
\cc{e,\cc{f,g}^{\prime}_l}^{\prime}_{l,pj}\otimes\cc{e,\cc{f,g}^{\pprime}_l}^{\pprime}_{l,iv}&\otimes\cc{f,g}^{\pprime}_{l,uq}=\cc{e,g}^{\prime}_{l,pj}\otimes\ii{f,\partial\cc{e,g}^{\pprime}_l}^{\prime}_{iv}\otimes\ii{f,\partial\cc{e,g}^{\pprime}_l}^{\pprime}_{uq}
\\&
\quad+\cc{\cc{e,f}^{\prime}_l,g}^{\prime}_{l,pj}\otimes\cc{e,f}^{\pprime}_{l,iv}\cc{\cc{e,f}^{\prime}_l,g}^{\pprime}_{l,uq},
 \end{split}
\\
\begin{split}
\cc{e,\cc{f,g}^{\prime}_l}^{\prime}_{r,pj}\otimes\cc{e,\cc{f,g}^{\prime}_l}^{\pprime}_{r,iv}&\otimes\cc{f,g}^{\pprime}_{l,uq}=\cc{e,g}^{\prime}_{r,pj}\otimes\cc{f,\cc{e,g}^{\pprime}_r}^{\prime}_{l,iv}\otimes\cc{f,\cc{e,g}^{\pprime}_r}^{\pprime}_{l,uq}
\\
&\quad +\ii{\cc{e,f}^{\prime}_l,g}^{\prime}_{pj}\otimes\partial\cc{e,f}^{\pprime}_{l,iv}\otimes\ii{\cc{e,f}^{\prime}_l,g}^{\pprime}_{uq}
\\
&\quad -\ii{\partial\cc{e,f}^{\prime}_r,g}^{\prime}_{pj}\otimes\cc{e,f}^{\pprime}_{r,iv}\otimes\ii{\partial\cc{e,f}^{\prime}_r,g}^{\pprime}_{uq},
\end{split}
\\
\begin{split}
\ii{e,\partial\cc{f,g}^{\prime}_r}^{\prime}_{pj}\otimes\ii{e,\partial\cc{f,g}^{\prime}_r}^{\pprime}_{iv}&\otimes\cc{f,g}^{\pprime}_{r,uq}=\cc{e,g}^{\prime}_{r,pj}\otimes\cc{f,\cc{e,g}^{\pprime}_r}^{\prime}_{r,iv}\otimes\cc{f,\cc{e,g}^{\pprime}_r}^{\pprime}_{r,uq}
\\&
\quad+\cc{\cc{e,f}^{\prime}_l,g}^{\prime}_{r,pj}\otimes\cc{e,f}^{\pprime}_{l,iv}\otimes\cc{\cc{e,f}^{\prime}_l,g}^{\pprime}_{r,uq},
\end{split}\end{align*}
taking place in the spaces $E_N\otimes A_N\otimes A_N$, $A_N\otimes E_N\otimes A_N$ and $A_N\otimes A_N\otimes E_N$, respectively.
Applying the multiplication maps on these spaces into $E_N$, these identities determine three identities on $E_N$.
Comparing with~\eqref{eq:Jacobi-KR.1.2.a}, \eqref{eq:Jacobi-KR.2.2.a} and~\eqref{eq:Jacobi-KR.2.2.c}, it is slightly tedious but straightforward to show that the equation obtained by adding these three identities is
\begin{equation}\label{eq:Jacobi-KR.4.1}
x_l=y_l+z_l.
\end{equation}
By a similar argument, applying $(\?)_{pv}\otimes(\?)_{uj}\otimes(\?)_{iq}$ to~\eqref{eq:CD-explicit-right-double-Jacobi} and the multiplication maps from 
$E_N\otimes A_N\otimes A_N$, $A_N\otimes E_N\otimes A_N$ and $A_N\otimes A_N\otimes E_N$ into $E_N$, we obtain three identities on $E_N$. Adding them and comparing with~\eqref{eq:Jacobi-KR.1.2.b}, \eqref{eq:Jacobi-KR.2.2.b} and~\eqref{eq:Jacobi-KR.2.2.d}, it is again slightly tedious but straightforward to show that the resulting equation is
\begin{equation}\label{eq:Jacobi-KR.4.2}
x_r=y_r+z_r.
\end{equation}
The Jacobi identity~\eqref{eq:KR-Jacobi-inicial} now follows from~\eqref{eq:Jacobi-KR.1}, \eqref{eq:Jacobi-KR.4.1} and \eqref{eq:Jacobi-KR.4.2}:
\begin{align*}
[e_{ij},[f_{uv},g_{pq}]]&=x_l+x_r=(y_l+z_l)+(y_r+z_r)=(y_l+y_r)+(z_l+z_r)
\\&
=[[e_{ij},f_{uv}],g_{pq}]+[f_{uv},[e_{ij},g_{pq}]].  
\qedhere
\end{align*}
\end{proof}

\begin{corollary}\label{cor:KR-non-degenerate-double-Courant-Dorfman}
If an $R$-linear double Courant--Dorfman algebra $(A,E,\ii{\?,\?},\partial,\cc{\?,\?})$ is non-degenerate, then so is the induced Courant--Dorfman algebra $(A_N,E_N,\langle\?,\?\rangle, \partial_N, [-,-])$.
\end{corollary}

\begin{proof}
Immediate by Theorem~\ref{thm:KR-double-Courant-Dorfman} and~\cite[Lemma 3.3.3]{VdB08a}.

\end{proof}

\subsection{The Kontsevich--Rosenberg principle for exact double Courant--Dorfman algebras}
\label{sec:KR-exact-double-CD}

Here we show that the standard exact double Courant--Dorfman algebras and their twists by Karoubi--de Rham closed 3-forms, constructed in Theorems~\ref{thm:standard-DCA} and~\ref{thm:twist-CD-algebra-double}, respectively, satisfy the Kontsevich--Rosenberg principle for smooth algebras.

For any commutative $R$-algebra $\cO$, $\Omega^{\bullet}_{\cO}=\wedge^{\bullet}_{\cO}\Omega^1_{\cO}$ is the algebraic de Rham complex of $\cO$ (relative over $R$), with ($R$-linear) exterior differential $\Wd \colon \Omega^{\bullet}_{\cO}\to\Omega^{\bullet+1}_{\cO}$, where $\Omega^1_{\cO}$ is the $\cO$-module of (commutative) K\"ahler differentials (relative over $R$)~\cite[\S II. 8]{Hartshorne}, $[\commX,\commY]\in\Der\cO$ is the Lie bracket of two ($R$-linear) derivations $\commX,\commY\in\Der\cO$, and $\Wi_{\commX}\colon \Omega^{\bullet}_{\cO}\to \Omega^{\bullet-1}_{\cO}$ and  $\WL_{\commX}=[\Wd,\Wi_{\commX}]\colon\Omega^{\bullet}_{\cO}\to\Omega^{\bullet}_{\cO}$  are the contraction and the Lie derivative along any $\commX\in\Der\cO$.

To bring together the noncommutative and commutative realms, we define linear maps
\begin{subequations}\label{eq:KR-exact-double-CD.1}
\begin{align}\label{eq:KR-exact-double-CD.1.a}
(\?)_{ij}\colon\Omega_R^\bullet A&\lto\Omega_{A_N}^\bullet,\quad \omega\longmapsto\omega_{ij}, 
\\\label{eq:KR-exact-double-CD.1.b}
(\?)_{ij}\colon\DDer_RA&\lto\Der A_N,\quad X\longmapsto X_{ij}, 
\end{align}\end{subequations}
for all integers $ i$ and $j$ such that $1\leq i,j\leq \lvert N\rvert $, by the condition that they are given on
\[
\omega=a_0\du{a_1}\cdots\du{a_n}\in\Omega_R^nA,\quad X\in\DDer_RA,
\]
with $a_0,\ldots,a_n\in A$, by the following formulae~\cite[\S\S 7.3, 7.4]{VdB08} for all $a\in A$, $1\leq u,v\leq \lvert N\rvert$:
\begin{subequations}\label{eq:KR-exact-double-CD.2}
\begin{align}\label{eq:KR-exact-double-CD.2.a}
\omega_{ij}&=a_{0,ik_1}\du{a_{1,k_1k_2}}\cdots\du{a_{n,k_{n}j}},
\\\label{eq:KR-exact-double-CD.2.b}
X_{ij}(a_{uv})&=(X'a)_{uj}(X''a)_{iv}.
\end{align}\end{subequations}


\begin{proposition}[{\cite[Proposition 3.3.4 and Corollary 3.3.5]{VdB08a} (see also \cite[\S 5]{BKR13})}]
\label{prop:KR-double-deriv}
The map~\eqref{eq:KR-exact-double-CD.1.a} determines an $A_N$-module isomorphism
\[
\big(\Omega_R^\bullet A\big)_N\cong\Omega^\bullet_{A_N}.
\]
Furthermore, if $A$ is smooth over $R$, then~\eqref{eq:KR-exact-double-CD.1.b} determines an $A_N$-module isomorphism
\[
\big(\DDer_RA\big)_N\cong \Der A_N. 
\]
\end{proposition}

Note now that there is an action of the group $\GL_N\defeq\prod_{\nu}\GL_{N_\nu}$ by conjugation on $M_N(\kk)$, and that this action induces another one on $A_N$ via~\eqref{eq:KR.adjunction-A_N}, and hence on $\Omega^\bullet_{A_N}$ and $\Der A_N$. 
As observed in~\cite[\S 6]{CBEG07} and \cite[\S\S 2.2, 2.3]{VdB08a}, the \emph{traces}
\[
\tr\,a\defeq a_{ii}\in A_N,\,\tr\,\omega\defeq\omega_{ii}\in\Omega^\bullet_{A_N},\,\tr\,X\defeq X_{ii}\in\Der A_N
\]
(sum over repeated indices) of any $a\in A$, $\omega\in\Omega_R^\bullet A$ and $X\in\DDer_RA$, only depend on the values of $a$, $\omega$ and $X$ modulo (graded) commutators. Furthermore, they are $\GL_N$-invariant~\cite[\S 7.7]{VdB08}. In particular, the assignment $\omega\mapsto\tr\,\omega$ descends to a linear map 
\begin{equation}
\tr\colon\dR_R^\bullet A\lto\(\Omega^{\bullet}_{A_N}\)^{\GL_N}
\label{eq:KR-Karoubi-de-Rham}
\end{equation}
that commutes with the de Rham differentials (see~\cite[eq. (6.2.2)]{CBEG07},~\cite[\S 2.2]{VdB08a}), and so provides a realization of the  Kontsevich--Rosenberg principle for the Karoubi--de Rham complex $\dR_R^\bullet A$.
The following result will be important in the proof of Proposition \ref{prop:KR.standard-complex-DCA.1}:

\begin{proposition}\label{prop: KR-VdB-SN}
Let $X,Y\in\DDer_RA$. Then the Lie bracket of $X_{ij},Y_{uv}\in\Der A_N$ is
\[
[X_{ij},Y_{uv}]=\db{X,Y}'_{uj}\db{X,Y}''_{iv},
\]
where $\db{-,-}$ is the double Schouten--Nijenhuis bracket on $\DDer_RA$.
\end{proposition}

\begin{proof}
This result holds more generally for $X,Y\in T_A(\DDer_RA)$ and the Schouten--Nijenhuis bracket $[\?,\?]$ of polyvector fields on $\Rep(A,N)$. See~\cite[Proposition 7.6.1]{VdB08} for details.
\end{proof}

We are now in a position to prove a Kontsevich--Rosenberg principle for exact double Courant--Dorfman algebras.

\begin{proposition}\label{prop:KR.standard-complex-DCA.1}
Let $(A,E,\ii{\?,\?},\partial,\cc{\?,\?})$ be the non-degenerate standard exact $R$-linear double Courant--Dorfman algebra of Theorem~\ref{thm:standard-DCA}, where $A$ is smooth over $R$.
Then the non-degenerate Courant--Dorfman algebra $(A_N,E_N,\langle\?,\?\rangle,\partial,[\?,\?])$ of Corollary~\ref{cor:KR-non-degenerate-double-Courant-Dorfman} is the standard exact Courant--Dorfman algebra, 
\emph{i.e.}, it is given by the following data:
\begin{enumerate}
\item[\textup{(1)}]
the commutative algebra $A_N$, 
\item[\textup{(2)}]
the $A_N$-module
\begin{equation}\label{eq:prop:KR.standard-complex-DCA.1.1}
E_N=\Der A_N\oplus\Omega^1_{A_N},
\end{equation}
whose elements will be denoted as sums $\commX+\commalpha$, with $\commX\in\Der A_N$ and $\commalpha\in\Omega^1_{A_N}$, 
\item[\textup{(3)}]
the pairing $\langle\?,\?\rangle\colon E_N\otimes E_N\to A_N$ given for all $\commX+\commalpha,\commY+\commbeta\in E_N$ by
\begin{equation}\label{eq:prop:KR.standard-complex-DCA.1.2}
\langle\commX+\commalpha,\commY+\commbeta\rangle=\Wi_{\commX}\commbeta+\Wi_{\commY}\commalpha,
\end{equation}
\item[\textup{(4)}]
the derivation $\partial_N\,\colon A_N\to E_N$ given for all $\comma\in A_N$ by
\begin{equation}\label{eq:prop:KR.standard-complex-DCA.1.3}
\partial_N\comma=0+\Wd{\comma},
\end{equation}
\item[\textup{(5)}]
the Dorfman bracket $[\?,\?]\colon E_N\otimes E_N\to E_N$ given for all $\commX+\commalpha,\commY+\commbeta\in E_N$ by
\begin{equation}\label{eq:prop:KR.standard-complex-DCA.1.4}
[\commX+\commalpha,\commY+\commbeta]=[\commX,\commY]+\WL_{\commX}\commbeta-\Wi_{\commY}\Wd{\commalpha}.
\end{equation}
\end{enumerate}
Furthermore, for all $H\in\dR_R^3A$ such that $\du_{\operatorname{DR}}{H}=0$, the double Courant--Dorfman algebra $(A,E,\ii{-,-},\partial,\cc{-,-}_H)$ of Theorem~\ref{thm:twist-CD-algebra-double} induces, via Theorem~\ref{thm:KR-double-Courant-Dorfman}, the $(\tr\,H)$-twisted standard exact Courant--Dorfman algebra $(A_N,E_N,\langle\?,\?\rangle,\partial_N,[\?,\?]_{\tr H})$ over $A_N$, 
\emph{i.e.}, $E_N, \langle\?,\?\rangle$ and $\partial_N$ are as in \eqref{eq:prop:KR.standard-complex-DCA.1.1}--\eqref{eq:prop:KR.standard-complex-DCA.1.3} above,
and $[\?,\?]_{\tr H}$ is given by 
\begin{equation}\label{eq:prop:KR.standard-complex-DCA.1.5}
[\commX+\commalpha,\commY+\commbeta]_{\tr H}=[\commX+\commalpha,\commY+\commbeta]+\Wi_{\commX}\Wi_{\commY}(\tr H).
\end{equation}
\end{proposition}

We start with some preliminary calculations.

\begin{lemma}\label{lem:KR.standard-complex-DCA.1}
For all $X\in\DDer_RA$, $\alpha\in\Omega_R^\bullet A$, $\omega\in\dR_R^\bullet A$ and $1\leq i,j,u,v\leq \lvert N\rvert$, we have
\begin{subequations}\label{eq:lem:KR.standard-complex-DCA.1.1}
\begin{align}\label{eq:lem:KR.standard-complex-DCA.1.1.a}
\Wd{\alpha_{ij}}&=(\du{\alpha})_{ij},
\\\label{eq:lem:KR.standard-complex-DCA.1.1.b}
\Wi_{X_{ij}}\alpha_{uv}&=(i'_X\alpha)_{uj}(i''_X\alpha)_{iv},
\\\label{eq:lem:KR.standard-complex-DCA.1.1.c}
\Wi_{X_{ij}}\tr\,\omega &=(\iota_X\omega)_{ij},
\\\label{eq:lem:KR.standard-complex-DCA.1.1.d}
\WL_{X_{ij}}\alpha_{uv}&=(L'_X\alpha)_{uj}(L''_X\alpha)_{iv},
\end{align}
\end{subequations}
where $i_X\alpha=i'_X\alpha\otimes i''_X\alpha$ and $L_X\alpha=L'_X\alpha\otimes L''_X\alpha$, using Sweedler's notation to omit the summation symbol, with $i'_X\alpha,i''_X\alpha,L'_X\alpha,L''_X\alpha\in\Omega^\bullet_RA$ (cf.~\eqref{eq:Sweedler-on-Cartan-calculus.3}).
\end{lemma}

\begin{proof}
The identity \eqref{eq:lem:KR.standard-complex-DCA.1.1.a} was obtained in \cite[\S7.3]{VdB08} (see also the proof of \cite[Proposition 3.3.4]{VdB08a}). To prove~\eqref{eq:lem:KR.standard-complex-DCA.1.1.b}, it suffices to show that \eqref{eq:lem:KR.standard-complex-DCA.1.1.b} holds on generators (see also~\cite[Lemma 3.3.6]{VdB08a}). For $\alpha=a\in A$, we have $\Wi_{X_{ij}}a_{uv}=0=(i'_Xa)_{uj}(i''_Xa)_{iv}$, by~\eqref{eq:contraction.4.a}, while for $\alpha=\du{a}\in\Omega_R^1A$, \eqref{eq:contraction.4.b} and \eqref{eq:KR-exact-double-CD.2} imply
\[
\Wi_{X_{ij}}(\du{a})_{uv}=\Wi_{X_{ij}}\Wd{a_{uv}}=X_{ij}(a_{uv})=(X'a)_{uj}(X'' a)_{iv}=(i'_X\du{a})_{uj}(i''_X\du{a})_{iv}, 
\]
as required. To prove~\eqref{eq:lem:KR.standard-complex-DCA.1.1.c}, we apply \eqref{eq:def-contraction-reduced}, the last identity of \eqref{eq:KR-relations-defining-AV} and~\eqref{eq:lem:KR.standard-complex-DCA.1.1.b}:
\begin{align*}
(\iota_X\omega)_{ij}&=(-1)^{|i'_X\omega||i''_X\omega|}\big((i''_X\omega)(i'_X\omega)\big)_{ij}=(-1)^{|i'_X\omega||i''_X\omega|}(i''_X\omega)_{ik}(i'_X\omega)_{kj}
\\
&=(i'_X\omega)_{kj}(i''_X\omega)_{ik}=\Wi_{X_{ij}}\omega_{kk}=\Wi_{X_{ij}}\tr\,\omega.
\end{align*}

Finally, by the Cartan homotopic formula \eqref{eq:Lie-derivative.2} and the Leibniz rule, we have
\begin{equation}\label{eq:KR-aux-Lie-derivative}
L_X\alpha=\du{(i_X\alpha)}+i_X(\du{\alpha})=\du{i'_X\alpha}\otimes i''_X\alpha+(-1)^{|i'_X\alpha|}i'_X\alpha\otimes\du{i''_X\alpha}+i'_X(\du{\alpha})\otimes i''_X(\du{\alpha}).
\end{equation}
Using the commutative Cartan homotopic formula, \eqref{eq:lem:KR.standard-complex-DCA.1.1.a}, the Leibniz rule and~\eqref{eq:KR-aux-Lie-derivative}, we obtain~\eqref{eq:lem:KR.standard-complex-DCA.1.1.d}:
\begin{align*}
\WL_{X_{ij}}\alpha_{uv}&=\Wd{\Wi_{X_{ij}}\alpha_{uv}}+\Wi_{X_{ij}}\Wd{\alpha_{uv}}
=\Wd\big((i'_X\alpha)_{uj})(i''_X\alpha)_{iv}\big)+\Wi_{X_{ij}}\Wd{\alpha_{uv}}
\\
&=\Wd{(i'_X\alpha})_{uj}(i''_X\alpha)_{iv}+(-1)^{|i'_X\alpha|}(i'_X\alpha)_{uj}\Wd({i''_X\alpha})_{iv}+\Wi_{X_{ij}}\Wd{\alpha_{uv}}
\\
&=(\du{i'_X\alpha})_{uj}(i''_X\alpha)_{iv}+(-1)^{|i'_X\alpha|}(i'_X\alpha)_{uj}(\du{i''_X\alpha})_{iv}+(i'_{X}\du{\alpha})_{ij}(i''_{X}\du{\alpha})_{uv}
\\
&=(L'_X\alpha)_{uj}(L''_X\alpha)_{iv}.
\qedhere\end{align*}
\end{proof}

\begin{proof}[Proof of Proposition~\ref{prop:KR.standard-complex-DCA.1}]
Note that $E_N$ is an $A_N$-module, as it is obtained applying the functor~\eqref{eq:VdB-functor-def}. 
Since this functor is additive and $E=\DDer_RA\oplus\Omega_R^1A$, we have
\[
E_N=\big(\DDer_RA\big)_N\oplus \big(\Omega_R^1A\big)_N=\Der A_N\oplus\Omega^1_{A_N},
\]
where we have applied Proposition \ref{prop:KR-double-deriv} in the last identity. 

To prove the assertions about the (untwisted) standard exact double Courant--Dorfman algebra, it suffices to check the identities~\eqref{eq:prop:KR.standard-complex-DCA.1.2}--\eqref{eq:prop:KR.standard-complex-DCA.1.4} on generators of $A_N$ and $E_N$, \emph{i.e.}, 
\begin{subequations}\label{pf:eq:lem:KR.standard-complex-DCA.1.1}
\begin{align}\label{pf:eq:lem:KR.standard-complex-DCA.1.1.a}
\partial_N a_{ij}&=\Wd{a_{ij}},
\\\label{pf:eq:lem:KR.standard-complex-DCA.1.1.b}
\langle X_{ij}+\alpha_{ij},Y_{uv}+\beta_{uv}\rangle&=\Wi_{X_{ij}}\beta_{uv}+\Wi_{Y_{uv}}\alpha_{ij},
\\\label{pf:eq:lem:KR.standard-complex-DCA.1.1.c}
[X_{ij}+\alpha_{ij},Y_{uv}+\beta_{uv}]&=[X_{ij},Y_{uv}]+\WL_{X_{ij}}\beta_{uv}-\Wi_{Y_{uv}}\Wd{\alpha_{ij}},
\end{align}
\end{subequations}
for all $a\in A$, $X,Y\in\DDer_RA$, $\alpha,\beta\in\Omega_R^\bullet A$ and $1\leq i,j,u,v\leq\lvert N\rvert$, where the left-hand sides are defined by~\eqref{eq:KR-CD-thm}. 

Firstly, \eqref{pf:eq:lem:KR.standard-complex-DCA.1.1.a} is immediate, by \eqref{eq:thm:standard-DCA.3}, \eqref{eq:lem:KR.standard-complex-DCA.1.1} and \eqref{eq:KR-CD-thm.a}.
Secondly, \eqref{pf:eq:lem:KR.standard-complex-DCA.1.1.b} follows from \eqref{eq:KR-CD-thm.b}, \eqref{eq:thm:standard-DCA.4}, the commutativity of $A_N$ and \eqref{eq:lem:KR.standard-complex-DCA.1.1.b}:
\begin{align*}
\langle X_{ij}+\alpha_{ij},Y_{uv}+\beta_{uv}\rangle&=\ii{X+\alpha, Y+\beta}'_{uj}\ii{X+\alpha,Y+\beta}''_{iv}
\\
&=(i'_X\beta)_{uj}(i''_X\beta)_{iv}+(i''_Y\alpha)_{uj}(i'_Y\alpha)_{iv}
\\
&=(i'_X\beta)_{uj}(i''_X\beta)_{iv}+(i'_Y\alpha)_{iv}(i''_Y\alpha)_{uj}
\\
&=\Wi_{X_{ij}}\beta_{uv}+\Wi_{Y_{uv}}\alpha_{ij}.
\end{align*}

Next, to prove \eqref{pf:eq:lem:KR.standard-complex-DCA.1.1.c}, we first compute
\begin{align*}
\Wi_{Y_{uv}}\Wd\alpha_{ij}&=\Wi_{Y_{uv}}(\du{\alpha})_{ij}
=(i'_Y\du{\alpha})_{iv} (i''_Y\du{\alpha})_{uj}
\\
&=(-1)^{|i'_Y\du{\alpha}||i''_Y\du{\alpha}|} (i''_Y\du{\alpha})_{uj}(i'_Y\du{\alpha})_{iv} =\big((i_Y\du{\alpha})^\sigma\big)'_{uj}\big((i_Y\du{\alpha})^\sigma\big)''_{iv},
\end{align*}
where we used \eqref{eq:lem:KR.standard-complex-DCA.1.1.a}, \eqref{eq:lem:KR.standard-complex-DCA.1.1.b} and the fact that $(i_Y\du{\alpha})^\sigma=(-1)^{|i'_Y\du{\alpha}||i''_Y\du{\alpha}|}i''_Y\du{\alpha}\otimes i'_Y\du{\alpha}$.
The identity \eqref{pf:eq:lem:KR.standard-complex-DCA.1.1.c} follows now from this, \eqref{eq:KR-CD-thm.c}, \eqref{eq:thm:standard-DCA.4}, Proposition \ref{prop: KR-VdB-SN} and \eqref{eq:lem:KR.standard-complex-DCA.1.1.d}:
\begin{align*}
[X_{ij}+&\alpha_{ij},Y_{uv}+\beta_{uv}]
=\lr{X,Y}'_{l,uj}\lr{X,Y}''_{l,iv}+\lr{X,Y}'_{r,uj}\lr{X,Y}''_{r,iv}
\\&\quad
+ (L^{l'}_X\beta)_{uj} (L^{l''}_X\beta)_{iv}+(L^{r'}_X\beta)_{uj} (L^{r''}_X\beta)_{iv}
-(i^{l''}_Y\du{\alpha})_{uj}(i^{l'}_Y\du{\alpha})_{iv}-(i^{r''}_Y\du{\alpha})_{uj}(i^{r'}_Y\du{\alpha})_{iv}
\\
&=\lr{X,Y}'_{uj}\lr{X,Y}''_{iv}+ (L^{'}_X\beta)_{uj} (L^{''}_X\beta)_{iv}-\big((i_Y\du{\alpha})^\sigma\big)'_{uj}\big((i_Y\du{\alpha})^\sigma\big)''_{iv}
\\
&=[X_{ij},Y_{uv}]+\WL_{X_{ij}}\beta_{uv}-\Wi_{Y_{uv}}\Wd\alpha_{ij}.
\end{align*}

Before we prove the assertions about the $H$-twisted standard exact double Courant--Dorfman algebra, note that formula~\eqref{eq:prop:KR.standard-complex-DCA.1.5} defines a Courant--Dorfman algebra, because 
\[
\Wd \tr\, H=\Wd H_{ii}=(\du_{\dR}H)_{ii}=0, 
\]
since $\du_\dR{H}=0$ by hypothesis. 
Hence it is enough to prove \eqref{eq:prop:KR.standard-complex-DCA.1.5} on generators, that is, 
\begin{equation}\label{pf:KR-Courant-deformado.1}
[X_{ij}+\alpha_{ij},Y_{uv}+\beta_{uv}]_{{\tr H}}=[X_{ij}+\alpha_{ij},Y_{uv}+\beta_{uv}]+\Wi_{X_{ij}}\Wi_{Y_{uv}}\tr\, H.
\end{equation}
Now, the left-hand side of~\eqref{pf:KR-Courant-deformado.1} is defined by~\eqref{eq:KR-CD-thm.c}, so~\eqref{pf:eq:lem:KR.standard-complex-DCA.1.1.c} implies that it equals
\begin{equation}\label{pf:KR-Courant-deformado.2}
[X_{ij}+\alpha_{ij},Y_{uv}+\beta_{uv}]_{{\tr H}}=[X_{ij}+\alpha_{ij},Y_{uv}+\beta_{uv}]+(i'_X\iota_Y H)_{uj}(i''_X\iota_Y H)_{iv}.
\end{equation}
Applying \eqref{eq:lem:KR.standard-complex-DCA.1.1.c} and \eqref{eq:lem:KR.standard-complex-DCA.1.1.d}, we see that the last term of~\eqref{pf:KR-Courant-deformado.2} is 
\begin{equation}\label{pf:KR-Courant-deformado.3}
(i'_X\iota_Y H)_{uj}(i''_X\iota_Y H)_{iv}=\Wi_{X_{ij}}(\iota_Y H)_{uv}=\Wi_{X_{ij}}\Wi_{Y_{uv}}\tr\, H.
\end{equation}
The identity~\eqref{pf:KR-Courant-deformado.1} follows now from~\eqref{pf:KR-Courant-deformado.2} and~\eqref{pf:KR-Courant-deformado.3}.
\end{proof}


\begin{thebibliography}{12}

\frenchspacing\smallbreak
  
 
\bibitem{ACF15}
L.\ \'Alvarez-C\'onsul \and D.\ Fern\'andez, 
\emph{Noncommutative bi-symplectic $\mathbb{N}Q$-algebras of weight 1}.
In \emph{Dynamical Systems, Differential Equations and Applications. AIMS Proceedings}, (2015), 19--28.
 
\bibitem{ACF17}
L.\ \'Alvarez-C\'onsul \and D.\ Fern\'andez, 
\emph{Non-commutative Courant algebroids and quiver algebras}.
(2017), 50 pp., available at \href{https://arxiv.org/pdf/1705.04285.pdf}{arXiv:1705.04285}.

 


\bibitem{arakawa}
T.\ Arakawa,
\emph{A remark on the {$C_2$}-cofiniteness condition on vertex algebras}.
Math. Z. \textbf{270} (2012), no. 1-2, 559--575.

\bibitem{BDSK09}
A.\ Barakat, A.\ De\ Sole \and V.\ G.\ Kac,
\emph{Poisson vertex algebras in the theory of Hamiltonian equations},
Jpn. J. Math. \textbf{4} (2009) 141--252.

\bibitem{bd}
A.\ Beilinson \and  V.\ Drinfeld, 
\emph{Chiral algebras},
American Mathematical Society Colloquium Publications, \textbf{51}, American Mathematical Society, Providence, RI, 2004.
        

\bibitem{BCER12}
Y.\ Berest, X.\ Chen, F.\ Eshmatov \and A.\ Ramadoss,
\emph{Noncommutative Poisson structures, derived representation schemes and Calabi-Yau algebras}.
Contemporary Mathematics, American Mathematical Society, \textbf{583} (2012), pp. 219--246.

 
\bibitem{BKR13}
 Y.\ Berest, G.\ Khachatryan \and A.\ Ramadoss,
 \emph{Derived representation schemes and cyclic homology}.
 Adv. Math. \textbf{245} (2013), 625--689.

       
\bibitem{bressler1}
P.\ Bressler.
\emph{The first Pontryagin class}.
Compos. Math. \textbf{143} (2007), no. 5, 1127--1163.


\bibitem{Brion12}
M. Brion (ed.),
\emph{Geometric methods in representation theory I, II}. 
S\'eminaires et Congr\`es \textbf{24}. Soci\'et\'e Math\'ematique de France, Paris, 2012.

\bibitem{CW}
M.\ Casati \and J. P.\ Wang,  
\emph{Hamiltonian structures for integrable nonabelian difference equations}.
(2021), 58 pp., available at \href{https://arxiv.org/pdf/2101.06191.pdf}{arXiv:2101.06191}.
      
       
    
      

      
\bibitem{Cou90}
T.\ J.\ Courant,
\emph{Dirac manifolds}.
Trans. Amer. Math. Soc. \textbf{319}(2) (1990), 631--661.

       

\bibitem{CB11}
W.\ Crawley-Boevey,
\emph{A note on noncommutative Poisson structures}.
J. Algebra \textbf{325} (2011) 205--215.

\bibitem{CBEG07}
W.\ Crawley-Boevey, P.\ Etingof \and V.\ Ginzburg,
\emph{Noncommutative geometry and quiver algebras}.
Adv. Math., \textbf{209} (2007), no. 1, 274--336.

\bibitem{CQ95}
J.\ Cuntz \and D.\ Quillen,
\emph{Algebra extensions and nonsingularity}.
J. Amer. Math. Soc., \textbf{8} (1995) no. 2, 251--289.

\bibitem{DSK13}
A.\ De\ Sole \and V.\ G.\ Kac, 
\emph{Non-local Poisson structures and applications to the theory of integrable systems}.
Jpn. J. Math. \textbf{8} (2013) 233--347.

\bibitem{DSKV15}
A.\ De Sole, V. G.\ Kac \and D.\ Valeri,
\emph{Double Poisson vertex algebras and non-commutative Hamiltonian equations}.
Adv. Math. \textbf{281} (2015), 1025--1099.
       
\bibitem{Do93} 
I. Dorfman. 
\emph{Dirac structures and integrability of nonlinear evolution equations}. John Wiley \& Sons, Ltd., Chichester, 1993.

\bibitem{EZ11}
J.\ Ekstrand \and M.\ Zabzine,
\emph{Courant-like brackets and loop spaces}.
J. High Energy Phys., 03, 2011 (074).



\bibitem{FV}
M.\ Fairon \and D.\ Valeri,
\emph{Double Multiplicative Poisson Vertex Algebras}.
Accepted for publication in Int. Math. Res. Not. 
\href{https://doi.org/10.1093/imrn/rnac245}{https://doi.org/10.1093/imrn/rnac245}.

\bibitem{Fer16}
D.\ Fern\'andez,
\emph{Non-commutative symplectic $\mathbb{N}Q$-geometry and Courant algebroids}.
PhD Thesis. Universidad Aut\'onoma de Madrid, 2016.

\bibitem{FH}
D.\ Fern\'andez \and E.\ Herscovich, 
\emph{Cyclic $A_\infty$-algebras and double Poisson algebras}. 
J. Noncommut. Geom. \textbf{15} (2021), 241--278.

\bibitem{gerbes2}
V.\ Gorbounov, F.\ Malikov \and V.~Schechtman,
\emph{Gerbes of chiral differential operators. II. Vertex algebroids}. 
Invent. Math. \textbf{155} (2004), no. 3, 605--680. 




\bibitem{Hartshorne}
        R.\ Hartshorne, 
        \emph{Algebraic Geometry}.
        Springer, New York--Heidelberg, 1977.


\bibitem{Hel09}
R.\ Heluani,
\emph{Supersymmetry of the chiral de Rham complex II: Commuting sectors}.
Int. Math. Res. Not. \textbf{6} (2009), 953--987.

\bibitem{kac:vertex}
V. G.\ Kac,
\newblock {\em Vertex algebras for beginners},  
University Lecture Series, 10. American Mathematical Society, Providence, RI, 1997. viii+141 pp.

 \bibitem{Ke11}
 B.\ Keller,
 \emph{Deformed Calabi-Yau completions}. 
 With an appendix by Michel Van den Bergh. J. Reine Angew. Math. \textbf{654} (2011), 125--180.

\bibitem{Kon94a}
M.\ Kontsevich,
\emph{Formal (non)commutative symplectic geometry}.
In \emph{The Gelfand Mathematical Seminars, 1990-1992}, pages 173--187. Bikh\"auser Boston, Boston, MA, 1993.
       
\bibitem{KR00}
M.\ Kontsevich \and A.\ Rosenberg,
\emph{Noncommutative smooth spaces}.
In \emph{The Gelfand Mathematical Seminars, 1996--1999}, Gelfand Math. Sem., pages 85--108. Birkh\"auser Boston, Boston, MA, 2000.

\bibitem{Kupershmidt.2000}
B.\ A.\ Kupershmidt
\emph{KP or mKP. Noncommutative mathematics of Lagrangian, Hamiltonian, and integrable systems}. Mathematical Surveys and Monographs \textbf{78}. Amer. Math. Soc., Providence, RI, 2000.



       



\bibitem{LeBruyn-VandeWeyer.2002}
L.\ Le\ Bruyn \and G.\ Van\ de\ Weyer,
\emph{Formal structures and representation spaces}.
J. Algebra \textbf{247} (2002) 616--635.

\bibitem{li05}
H.~Li.
\emph{Abelianizing vertex algebras}.
Comm. Math. Phys. \textbf{259} (2005) 391--411.

\bibitem{LWX97}
Z.-J.\ Liu, A.\ Weinstein \and P.\ Xu,
\emph{Manin triples for Lie bialgebroids},
J. Differential Geom. \textbf{45}(3) (1997) 547--574.


 \bibitem{MSV99}
 F.\ Malikov, V.\ Schechtman, A.\ Vaintrob,
 \emph{Chiral de Rham complex},
 Comm. Math. Phys. \textbf{204} (1999), 439--473.

\bibitem{Mikhailov-Sokolov.2000}
A.\ V.\ Mikhailov \and V.\ V.\ Sokolov,
\emph{Integrable ODEs on associative algebras},
Comm. Math. Phys. \textbf{211} (2000) 231--251.



\bibitem{ORS}
A.\ Odesskii, V.\ Rubtsov \and V.\ Sokolov, 
\emph{Double Poisson brackets on free associative algebras}.
 Noncommutative birational geometry, representations and combinatorics, 225--239, Contemp. Math., \textbf{592}, Amer. Math. Soc., Providence, RI, 2013.
      
      \bibitem{Olver-Sokolov.1998}
P.\ J.\ Olver \and V.\ V.\ Sokolov,
\emph{Integrable evolution equations on associative algebras},
Comm. Math. Phys. \textbf{193} (1998) 245--268

\bibitem{Roy00}
D.\ Roytenberg,
\emph{On the structure of graded symplectic supermanifolds and Courant algebroids}.
In \emph{Quantization, Poisson brackets and Beyond}, Th. Voronov (ed.), Contemp. Math. \textbf{315}, Amer. Math. Soc., Providence, RI, 2002.
       
       
\bibitem{Roy09}
D.\ Roytenberg,
\emph{Courant--Dorfman algebras and their cohomology}. 
Lett. Math. Phys. \textbf{90} (2009), no. 1--3, 311--351.

                                                                               
        
\bibitem{Sev00}
P.\ \v Severa,
\emph{Letters to Alan Weinstein about Courant algebroids}. 
29 pp., available at \href{https://arxiv.org/pdf/1707.00265.pdf}{arXiv:1707.00265}.

\bibitem{SW01}
P.\ \v Severa \and A. Weinstein,
\emph{Poisson geometry with a 3-form background}. Prog. Theor. Phys. Suppl.
\textbf{144} (2001), 145--154
  

\bibitem{vaintrob}
A.\ Vaintrob, 
\emph{Lie algebroids and homological vector fields}. 
Uspekhi Mat. Nauk \textbf{52} (1997), no. 2, 428--429.
       
\bibitem{VdB08}
M.\ Van den Bergh,
\emph{Double Poisson algebras}.
Trans. Amer. Math. Soc., \textbf{360} (2008) no. 11, 555--603.
        
\bibitem{VdB08a}
M.\ Van den Bergh,
\emph{Non-commutative quasi-Hamiltonian spaces}.
In: \emph{Poisson geometry in mathematics and physics}. Contemp. Math., vol. 450, Amer. Math. Soc., Providence, RI, pp. 273--299, 2008.



\end{thebibliography}
\end{document}